\documentclass{article}
\usepackage[utf8]{inputenc} %
\usepackage[T1]{fontenc}    %
\usepackage{fullpage}

\usepackage[round]{natbib}

\usepackage[margin = 1in]{geometry}

\usepackage{lipsum}
\usepackage{amsfonts}
\usepackage{graphicx}
\usepackage{epstopdf}

\usepackage{amsmath, amssymb}
\usepackage{amsthm, thmtools, thm-restate}

\newtheorem{thm}{Theorem}

\newtheorem{prop}{Proposition}
\newtheorem{lemma}{Lemma}

\newtheorem{claim}{Claim}

\newtheorem{definition}{Definition}
\newtheorem{assumption}{Assumption}

\newtheorem{remark}{Remark}

\newtheorem{example}{Example}

\usepackage{url}            %
\usepackage{booktabs}       %
\usepackage{wrapfig}
\usepackage{makecell, cellspace}
\usepackage[mathscr]{euscript}
\usepackage{float}
\usepackage[nodisplayskipstretch]{setspace}
\usepackage[normalem]{ulem}

\usepackage{times}
\usepackage{algorithm}
\usepackage{algorithmicx}
\usepackage[noend]{algpseudocode}
\usepackage{adjustbox}
\usepackage{amssymb}
\usepackage{bm,bbm, mathtools}
\usepackage{latexsym}
\usepackage{multirow}
\usepackage{paralist}
\usepackage{minitoc} %
\usepackage{xspace}
\usepackage{enumitem}
\usepackage{thm-restate}
\usepackage{soul}
\usepackage{mdframed}

\usepackage{rotating}

\usepackage[table, svgnames, dvipsnames]{xcolor}
\usepackage{array, colortbl}
\definecolor{darkpink}{rgb}{0.91, 0.33, 0.5}
\definecolor{puorange}{rgb}{0.80,0.20,0}
\definecolor{bluegray}{rgb}{0.04,0,0.7}
\definecolor{greengray}{rgb}{0.05,0.50,0.15}
\definecolor{darkbrown}{rgb}{0.40,0.2,0.05}
\definecolor{darkcyan}{rgb}{0,0.4,1}
\definecolor{black}{rgb}{0,0,0}
\definecolor{grey}{rgb}{0.93,0.93,0.93}
\definecolor{royalazure}{rgb}{0.0, 0.22, 0.66}
\definecolor{blueviolet}{RGB}{138,43,226}
\definecolor{lighttan}{rgb}{0.97,0.92,0.92}
\definecolor{cadmiumgreen}{rgb}{0.0, 0.42, 0.24}
\definecolor{harvardcrimson}{rgb}{0.79, 0.0, 0.09}

\usepackage{hyperref}

\usepackage[capitalize,noabbrev]{cleveref}
\crefname{section}{Section}{Sections}
\crefname{appendix}{Appendix}{Appendices}
\crefname{thm}{Theorem}{Theorems}
\crefname{theorem}{Theorem}{Theorems}
\crefname{lemma}{Lemma}{Lemmas}
\crefname{claim}{Claim}{Claims}
\crefname{lem}{Lem.}{Lems.}
\crefname{corollary}{Corollary}{Corollaries}
\crefname{proposition}{Proposition}{Propositions}
\crefname{prop}{Proposition}{Propositions}
\crefname{assumption}{Assumption}{Assumptions}
\crefname{asm}{Assumption}{Assumptions}
\crefname{algorithm}{Algorithm}{Algorithms}
\Crefname{algorithm}{Algorithm}{Algorithms}
\crefname{figure}{Figure}{Figures}
\crefname{table}{Table}{Tables}

\usepackage{color}
\definecolor{lightblue}{RGB}{0,160,200}
\definecolor{gray}{rgb}{0.4,0.4,0.4}
\usepackage{listings}
\newcommand{\red}[1]{{\color{purple}#1}}

\newcommand{\blue}[1]{{\color{royalazure}#1}}

\newcommand{\edit}[1]{{#1}}

\newcommand{\ones}{\boldsymbol{1}}

\renewcommand{\epsilon}{\varepsilon}
\newcommand \eps \epsilon

\newcommand \prob {{\mathbb{P}}}

\DeclarePairedDelimiterX{\inp}[2]{\langle}{\rangle}{#1, #2} 
\newcommand{\norm}[1]{\left\Vert #1 \right\Vert}

\newcommand{\abs}[1]{\left\lvert #1 \right\rvert}

\newcommand \argmin {\operatorname*{arg\,min}} %
\newcommand \argmax {\operatorname*{arg\,max}} %
\newcommand \dom {\operatorname*{dom}} %
\newcommand \Gap {\mathop{\mathrm{Gap}}\nolimits}  %
\newcommand \grad {\nabla}

\newcommand \Lcal {\mathcal L}

\newcommand{\R}{\mathbb{R}}

\newcommand{\F}{\ensuremath{\mathcal{F}}}

\newcommand{\p}[1]{\ensuremath{\left(#1\right)}}
\newcommand{\br}[1]{\ensuremath{\left\{#1\right\}}}
\newcommand{\sbr}[1]{\ensuremath{\left[#1\right]}}
\newcommand{\sse}{\subseteq}
\newcommand{\E}[2]{\ensuremath{{\mathbb E}_{#1}\left[#2\right]}}

\newcommand{\mc}[1]{\ensuremath{\mathcal{#1}}}

\newcommand{\ip}[1]{\ensuremath{\left\langle #1 \right\rangle}}
\newcommand{\ipsmall}[1]{\ensuremath{\langle #1 \rangle}}
\newcommand{\breg}[3]{\Delta_{#1}(#2, #3)}

\newcommand{\Ex}{\mathbb{E}}

\newcommand{\A}{\boldsymbol{A}}

\newcommand{\fb}{\boldsymbol{f}}
\newcommand{\bfb}{\bar{\boldsymbol{f}}}
\newcommand{\hfb}{\hat{\boldsymbol{f}}}

\newcommand{\Lb}{L}
\newcommand{\Gb}{G}
\newcommand{\Gbb}{\boldsymbol{G}}
\newcommand{\Lbb}{\boldsymbol{L}}
\newcommand{\Lc}{M}
\newcommand{\Gc}{H}
\newcommand{\Dc}{D}
\newcommand{\lamb}{\lambda}
\newcommand{\lambb}{\boldsymbol{\lambda}}
\newcommand{\hgb}{\hat{\boldsymbol{g}}}
\newcommand{\hJb}{\hat{\boldsymbol{J}}}

\newcommand{\Dmax}{D_{\max{}}}
\newcommand{\Gmax}{H_{\max{}}}
\newcommand{\Mmax}{M_{\max{}}}

\newcommand{\X}{\mc{X}}
\newcommand{\Y}{\mc{Y}}
\newcommand{\Z}{\mc{Z}}
\renewcommand{\u}{\boldsymbol{u}}
\renewcommand{\v}{\boldsymbol{v}}

\newcommand{\x}{\boldsymbol{x}}

\newcommand{\y}{\boldsymbol{y}}

\newcommand{\e}{\boldsymbol{e}}
\newcommand{\by}{\bar{y}}
\newcommand{\byb}{\bar{\boldsymbol{y}}}
\newcommand{\z}{\boldsymbol{z}}

\newcommand{\g}{\bar{\boldsymbol{g}}}
\newcommand{\s}{\boldsymbol{s}}
\newcommand{\pb}{\boldsymbol{p}}
\newcommand{\qb}{\boldsymbol{q}}
\newcommand{\rb}{\boldsymbol{r}}
\newcommand{\sbold}{\boldsymbol{s}}
\newcommand{\hf}{\hat{f}}
\newcommand{\barf}{\bar{f}}
\newcommand{\tx}{\tilde{\boldsymbol{x}}}
\newcommand{\ty}{\tilde{\boldsymbol{y}}}

\newcommand{\hyb}{\hat{\boldsymbol{y}}}
\newcommand{\hy}{\hat{y}}
\newcommand{\hxb}{\hat{\x}}

\newcommand{\gap}{\operatorname{Gap}}
\newcommand{\ri}{\operatorname{ri}}

\newcommand{\Lavg}{\boldsymbol{\lambda}_{\mathrm{avg}}}
\newcommand{\Lunif}{\boldsymbol{\lambda}_{\mathrm{unif}}}
\newcommand{\Limp}{\boldsymbol{\lambda}_{\mathrm{imp}}}

\newcommand{\bgam}{\boldsymbol{\gamma}}
\newcommand{\TP}{\mathcal{T}^{\mathrm{P}}}
\newcommand{\TD}{\mathcal{T}^{\mathrm{D}}}
\newcommand{\hTP}{\hat{\mathcal{T}}^{\mathrm{P}}}
\newcommand{\hTD}{\hat{\mathcal{T}}^{\mathrm{D}}}
\newcommand{\CP}{\mathcal{C}^{\mathrm{P}}}
\newcommand{\CD}{\mathcal{C}^{\mathrm{D}}}
\newcommand{\hCP}{\hat{\mathcal{C}}^{\mathrm{P}}}
\newcommand{\hCD}{\hat{\mathcal{C}}^{\mathrm{D}}}

\newcommand{\EP}{\mathcal{E}^{\mathrm{P}}}
\newcommand{\ED}{\mathcal{E}^{\mathrm{D}}}
\newcommand{\IP}{\mathcal{I}^{\mathrm{P}}}
\newcommand{\ID}{\mathcal{I}^{\mathrm{D}}}

\newcommand{\Fro}{\mathrm{Fro}}

\newcommand{\wP}{w^{\mathrm{P}}}
\newcommand{\wD}{w^{\mathrm{D}}}
\newcommand{\sub}{\textsc{subroutine}\xspace}

\def\vw{{\bm{w}}}
\def\vx{{\bm{x}}}

\def\vy{{\bm{y}}}
\def\vz{{\bm{z}}}

\title{Efficient Methods for Min-Max Optimization\\with Dual-Linear Coupling}
\author{Ronak Mehta$^{1}$ \qquad Jelena Diakonikolas$^{2}$ \qquad Zaid Harchaoui$^1$ \vspace{0.3cm} \\
{
\small $^1$University of Washington, Seattle
\qquad
$^2$University of Wisconsin, Madison
}
}
\date{\today}

\begin{document}

\maketitle

\begin{abstract}
    We study a class of convex-concave min-max problems in which the coupled component of the objective is linear in at least one of the two decision vectors.  We identify such problem structure as interpolating between the bilinearly and nonbilinearly coupled problems, motivated by key applications in areas such as distributionally robust optimization and convex optimization with functional constraints. Leveraging the considered nonlinear-linear coupling of the primal and the dual decision vectors, we develop a general algorithmic framework leading to fine-grained complexity bounds exploiting separability properties of the problem, whenever present. The obtained complexity bounds offer potential improvements over state-of-the-art scaling with $\sqrt{n}$ or $n$ in some of the considered problem settings, which even include bilinearly coupled problems, where $n$ is the dimension of the dual decision vector. On the algorithmic front, our work provides novel strategies for combining randomization with extrapolation and multi-point anchoring in the mirror descent-style updates in the primal and the dual, which we hope will find further applications in addressing related optimization problems. %

\end{abstract}

\section{Introduction}\label{sec:intro}

Min-max (a.k.a.\ saddle point) optimization represents a fundamental class of problems with long history in optimization and game theory, spanning a plethora of applications in machine learning and operations research \citep{facchinei2003finite}. Among more recent applications in the data sciences, min-max optimization problems arise in adversarial training 
\citep{goodfellow2014generative, madry2018towards}, 
algorithmic fairness \citep{Chen2023AGuide, agarwal2018areductions, mehrabi2021asurvey}, 
distributionally robust optimization \citep{shapiro2017distributionally, Kuhn2019Wasserstein, Sagawa2020Distributionally},
and reinforcement learning \citep{kotsalis2022simple, moos2022robust}.

In their general form, min-max optimization problems ask for solving %
\begin{align}
    \min_{\x \in \X} \max_{\y \in \Y} \ c(\x, \y) - \psi(\y) + \phi(\x),\label{eq:spp}
\end{align}
with respect to \emph{primal variables} $\x \in \X \sse \R^d$ and \emph{dual variables} $\y \in \Y \sse \R^n$. We refer to $c(\x, \y)$ as the \emph{coupled} component, whereas $\phi(\x)$ and $\psi(\y)$ are the \emph{uncoupled} components of the objective.
Such problems are further categorized as \emph{bilinearly coupled} if $c(\x, \y) = \ip{\y, \A \x}$ for $\A \in \R^{n \times d}$ and standard Euclidean inner product $\ip{\cdot, \cdot}$, and are called \emph{nonbilinearly coupled} otherwise. They are said to be \emph{convex-concave} if the overall objective is convex in $\x$ for fixed $\y$ and concave in $\y$ for fixed $\x$. Convex-concave problems, or those with a convex-concave coupled component $c(\x, \y)$ and convex uncoupled components $\phi(\x)$ and $\psi(\y)$, are the focus of this work. 

General convex-concave problems of the form \eqref{eq:spp} are possible to abstract away and study within the framework of variational inequalities with monotone operators. Classical algorithms developed for variational inequalities with monotone operators such as the extragradient \citep{korpelevich1976extragradient}, Popov's method \citep{popov1980amodification}, mirror-prox \citep{nemirovski2004mirrorprox}, and dual-extrapolated method \citep{nesterov2007dual} can address such general problems \eqref{eq:spp} under minimal, standard assumptions. Randomized variants of these methods are also considered, for instance, in \citet{juditsky2011first, juditsky2013randomized, baes2013randomized}. More recent methods, such as those obtained in \citet{alacaoglu2022stochastic,cai2024variance,diakonikolas2025block}, offer more fine-grained complexity bounds and potential speedups for min-max objectives (and corresponding monotone operators) that are sum-decomposable. However, the broad applicability of these results comes at the cost of possibly missed opportunities for further speed improvements when the coupled component $c(\x, \y)$ exhibits further structural properties.

On the other side of the spectrum, splitting primal-dual methods specialized to problems with a bilinearly coupled component such as the primal-dual hybrid gradient (PDHG) method of \citet{chambolle2011afirstorder} and its many variants and specializations \citep{chambolle2018stochastic,alacaoglu2022ontheconvergence,song2021variance,song2022coordinate,dang2014randomized,fercoq2019coordinate} have been quite effective in offering both theoretical and practical speedups based on the structure of the objective's coupled component. For instance, methods based on PDHG are practical state-of-the-art for solving large-scale linear programs \citep{applegate2025pdlp}.  The question that then arises is whether the success of splitting primal-dual methods can be transferred to objectives with more general primal-dual coupling.

The starting point of our work is the observation that in many applications of core practical interest, the coupled component of the objective is neither bilinear nor fully nonbilinear. In particular, a wide range of problems studied within the framework of distributionally robust optimization (DRO), both for supervised learning \citep{Namkoong2016Stochastic, shapiro2017distributionally, carmon2022distributionally, Levy2020Large-Scale, mehta2024drago, mehta2024distributionally} and reinforcement learning \citep{Kallus2022Doubly, Yang2023Distributionally, Wang2023AFinite, Yu2023FastBellman}, fall into this category.
Similarly, fully composite problems (see \citet{cui2018composite, drusvyatskiy2019efficiency, doikov2022high, vladarean2023linearization} and references therein) and minimization problems with functional constraints can all be cast as min-max optimization problems in which the primal-dual coupling is linear in one of the two decision vectors (see \Cref{ex:dro,ex:fco,ex:fcm,ex:eigenvalue} for more details). Hence, the question motivating our work is whether insights from the bilinearly coupled case can be used to design efficient algorithms for these ``dual linear'' min-max problems\footnote{While we fix the convention that the coupled term is linear in the dual variables, our methods extend by analogy to primal linear objectives.}.  
Formally, we study the optimization problem
\begin{align}
    \min_{\x \in \X} \max_{\y \in \Y} \sbr{\Lcal(\x, \y) := \ip{\y, \fb(\x)} - \psi(\y) + \phi(\x)},
    \label{eq:semilinear}
\end{align}
wherein $\fb = (f_1, \ldots, f_n)$ has convex components, $\phi$ is $\mu$-strongly convex ($\mu \geq 0$), and $\psi$ is $\nu$-strongly convex ($\nu \geq 0$). We further require that $\Y \sse \br{\y \in \R^n: y_j \geq 0 \text{ if } f_j \text{ is non-linear}}$, so that~\eqref{eq:semilinear} constitutes a convex-concave saddle point problem. 

\subsection{Technical Motivation}
To motivate the usefulness of specialized algorithms for~\eqref{eq:semilinear}, consider treating it generically as a nonbilinearly-coupled saddle-point problem (or more generally still as a variational inequality (VI) problem). To quantify the resulting complexity, let $\lambda$ be the Lipschitz parameter of $(\x, \y) \mapsto (\grad_{\x} \Lcal, -\grad_{\y} \Lcal)$, when it exists. Classical approaches such as \citep{korpelevich1976extragradient,popov1980amodification,nemirovski2004mirrorprox,nesterov2007dual}, and their linearly convergent variants \citep{nesterov2006solving, marcotte1991application, tseng1995onlinear, mokhtari2020aunified, hai2025arefined} will achieve a total arithmetic complexity of %
\begin{align*}
    \tilde{O}\p{nd\min\br{\tfrac{\lambda}{\epsilon}, \tfrac{\lambda}{\mu \wedge \nu} \ln\p{\tfrac{1}{\epsilon}}}},
\end{align*}
to achieve a primal-dual gap bounded above by $\epsilon > 0$ (see~\eqref{eq:gap}, \Cref{sec:template}), where we assume that the per-iteration complexity is $\tilde{O}(nd)$ and we interpret $\lambda / 0 := +\infty$ . 

Furthermore, using generic acceleration Catalyst schemes \citep{lin2018catalyst} adapted for min-max problems, improved complexity bounds can be obtained. \citet{he2020acatalyst} achieve an $\tilde{O}(nd\lambda/\sqrt{\mu\epsilon})$ runtime in the strongly convex-concave regime. 
Recently~\citet{lan2023novel} obtained a complexity of $\tilde{O}(nd\lambda/\sqrt{\mu \nu} \ln\p{\tfrac{1}{\epsilon}})$ for strongly convex-strongly concave problems. 
A disadvantage of classical approaches that is overcome by the Catalyst approach is the dependence on the minimum of the two strong convexity constants. In the examples above, $\psi$ often plays the role of a strongly convex smoothing penalty whose strong convexity constant may be near-zero (e.g., $O(\epsilon)$ for some applications). Thus, guarantees depending on $(\mu \wedge \nu)$ are undesirable. 
A second disadvantage incurred by all of the approaches above is the dependence on a coarse Lipschitz constant $\lambda$ for the entire vector field. When it is only known that each component function $f_j$ is $\Gmax$-Lipschitz continuous and $\Mmax$-smooth, and that $\abs{y_j} \leq \Dmax$ for any $\y = (y_1, \ldots, y_n) \in \Y$, then $\lambda$ can be estimated as $n(\Gmax + \Dmax\Mmax)$. Combined with the order-$nd$ per-iteration cost, the total arithmetic complexity can have an $O(n^2d)$ dependence on the dimensions $(n, d)$. 
These two issues highlight the price to pay when using generic acceleration schemes: the loss of adaptation to the structure of the problem and the related constants. In contrast, by leveraging non-uniformity in the various Lipschitz/smoothness constants for each $f_j$ and randomized updates, we may achieve complexities that are linear in $(n + d)$. 

Moreover, the Lipschitz continuity and smoothness parameters of the component functions represent quantities of practical interest in applications, which we review via the examples in \Cref{sec:preliminaries}. In \Cref{ex:dro}, if we assume that $f_j(\x) = \ell(\z_j^\top \x)$ for a univariate loss function $\ell$ and data instance $\z_j$, then under appropriate conditions on $\ell$, the continuity properties of $f_j$ will depend on the norm of $\z_j$. In \Cref{ex:fco}, we have that $\fb(\x) = (\boldsymbol{h}(\x), \x)$ for some $\R^{n-d}$-valued function $\boldsymbol{h}$ so that $f_{n-d+1}, \ldots, f_n$ are $1$-Lipschitz continuous and $0$-smooth; a complexity result that only depends on the $n$ times the maxima of these constants over the components of $\fb$ could be overly pessimistic.
In our analysis, we handle every combination of $\mu = 0$ versus $\mu > 0$ and $\nu = 0$ versus $\nu > 0$ to achieve dependences on $\epsilon$ in a unified way, with complexity bounds that enjoy a transparent dependence on these component-wise Lipschitz and smoothness constants.
The dependence on component-wise problem constants has been shown in recent works on stochastic variance-reduced methods for variational inequality problems \citep{alacaoglu2022stochastic, cai2024variance, pichugin2024method, alizadeh2024variance, diakonikolas2025block}, for which we offer a detailed technical comparison in \Cref{sec:discussion}.

 \subsection{Contributions}
We study the dual linear min-max problem~\eqref{eq:semilinear} as an intermediary between bilinearly and nonbilinearly coupled problem classes. 
\begin{itemize}
    \item {\bf Unified convergence analysis for algorithms and problem classes:} We provide a general template for a first-order proximal gradient algorithm motivated by the analysis and instantiate it in three forms: a full vector update variant, a stochastic variant for generic objectives, and a variant that utilizes ``separable'' structure in the objective and feasible set. 
    \item {\bf Fine-grained dependence on component-wise constants:} We prove convergence rates for algorithms for this ``dual-linear'' (or equivalently, ``primal-linear'') problem class, with a transparent dependence on all individual problem parameters, including individual smoothness and strong convexity constants. On the suboptimality parameter $\epsilon > 0$, we achieve the standard $O(1/\epsilon)$ and $O(\ln(1/\epsilon))$ complexities for the convex-concave and strongly convex-strongly concave settings, and $O(1/\sqrt{\epsilon})$ in the strongly convex-concave setting.
    \item {\bf Improved dimension dependence over generic methods:} We achieve up to an $\sqrt{n}$ factor improvement over the recent methods of \citet{alacaoglu2022stochastic} and \citet{cai2024variance} by leveraging additional structure in the convex-concave case. 
    \item {\bf Complexity improvements in specific problem classes:} Furthermore, while motivated to solve problems more general than bilinear coupled min-max problems, our algorithm achieves complexity improvements in more settings such as matrix games. 
    \edit{For separable matrix games (see \Cref{def:separable}), we may improve over state-of-the-art methods by up to an $n$ factor when the coupling matrix is dense and square.}
\end{itemize}
The main technical ingredient of the algorithm is a combination of non-uniform sampling for randomized extrapolation and non-uniform weighting of multiple prox anchors/centers (which we refer to as ``historical regularization'') within mirror descent-style updates.
The randomized extrapolation component can also be interpreted as a variance reduction method for objectives with sum-decomposable coupled components. Notably, the historical regularization component, outlined in the upcoming \Cref{algo:drago_v2}, overcomes a barrier that was previously not handled by non-uniform sampling alone \citep{chambolle2018stochastic, song2021variance, alacaoglu2022complexity}.

\subsection{Related Work}
Among methods similar to ours for general nonbilinearly-coupled objectives, \citet{hamedani2021aprimal} achieve an $O(1/\epsilon)$ iteration complexity for the convex-concave case using a similar extrapolation term to ours with full vector updates, which can be seen as a generalization of iterate extrapolation in Chambolle-Pock-style algorithms \citep{chambolle2011afirstorder, chambolle2016ontheergodic}. Interestingly, they also consider the dual linear case in which $\phi$ is strongly convex, matching our $O(1/\sqrt{\epsilon})$ complexity in this case. Beyond the dependence on the error parameter $\epsilon$, a key focus of our work is achieving improved, fine-grained, complexity with stochastic algorithms, dependent on Lipschitz continuity and smoothness constants of individual component functions. When the coupled term is expressible as a finite sum of component functions (as in \eqref{eq:semilinear}), \citet{hamedani2023astochastic} also apply an SVRG-style variance reduction scheme \citep{johnson2013accelerating} to achieve an $O(n\epsilon^{-1/2} + \epsilon^{-3/2})$ iteration complexity. By leveraging non-uniformity of the various Lipschitz/smoothness constants of the component functions, our $O(1/\epsilon)$ iteration complexity (independent of $n$) in the convex-concave regime improves over these results. One challenge of our approach in more specialized problem classes, such as \Cref{ex:fcm}, is the possible unboundedness of the dual domain in the non-strongly concave regime. Issues such as these have been addressed by line search/adaptive step size methods, which were implemented in \citet{hamedani2021aprimal} as well as in randomized block coordinate-wise algorithms in \citet{hamedani2023randomized}. Further discussions on both classical and recent works are provided in \Cref{sec:discussion}.

We also comment that a subset of technical ideas we develop in this manuscript appeared in a preliminary NeurIPS conference paper written by the same authors \citep{mehta2024drago}. Compared to this prior work, the common theme consists in using a primal-dual method featuring a multiple prox-center scheme (i.e., historical regularization) and a randomized updating scheme akin to SAGA-style variance reduction \citep{defazio2014saga}. In contrast, the algorithmic framework from this work allows for importance-based sampling and employs randomized (instead of cyclic) updates to gradient and function value tables common to SAGA-style variance-reduction approaches, further allowing for more general choices of prox-centers. These properties of our new algorithmic framework are crucial for establishing more fine-grained complexity bounds, leading to the aforementioned complexity improvements over state of the art (including our aforementioned past work).

\subsection{Organization} The remainder of the paper is organized as follows. The problem parameters are introduced formally in \Cref{sec:preliminaries}. The high-level description of the analysis and algorithm template is given in \Cref{sec:template}. Stochastic algorithms and their guarantees for generic objectives are given in \Cref{sec:nonsep:stochastic}. A modified analysis for separable problems is given in \Cref{sec:sep}. We conclude with comparisons and discussion in \Cref{sec:discussion}.

\section{Preliminaries}\label{sec:preliminaries}
We now introduce the main notation and assumptions used in this work. Let $\norm{\cdot}_\X$ denote a norm on $\R^d$. Let $\norm{\cdot}_\Y$ denote an $\ell_p$-norm on $\R^n$ for $p \in [1, 2]$. The associated dual norms are denoted by $\norm{\cdot}_{\X^*}$ and $\norm{\cdot}_{\Y^*}$ and defined in the usual way as $\norm{\vw}_{\X^*} = \sup_{\vx: \norm{\vx}_\X \leq 1} \ip{\vw, \vx}$, $\norm{\vz}_{\Y^*} = \sup_{\vy: \norm{\vy}_\Y \leq 1} \ip{\vz, \vy}$. We use the notation $A \lesssim B$ to mean $A \leq C B$ for a universal positive constant $C$, independent of $A$ and $B.$ Notation `$\gtrsim$' is defined analogously.

The following assumptions about the objective in~\eqref{eq:semilinear} are made throughout the paper. We employ block coordinate-wise updates in the upcoming stochastic algorithms, of which coordinate-wise updates are a special case with block size one. 
To do so, we introduce a partitioning of the $n$ components of $\fb$ into $N$ blocks and define the relevant Lipschitz constants for each one. In discussions of arithmetic complexity, we may assume a uniform block size $n/N$, but the analysis natively handles blocks of possibly non-uniform size.
\begin{assumption}\label{asm:smoothness}
    Assume that each $f_j: \R^d \rightarrow \R$ is convex and the indices $[n] = \br{1, \ldots, n}$ are partitioned into \emph{blocks} $(B_1, \ldots, B_N)$. 
    There exist constants $\Gb_1, \ldots, \Gb_N \geq 0$ and $\Lb_1, \ldots \Lb_N \geq 0$ such that for each $J = 1, \ldots, N$:
    \begin{align}
        \sup_{\x \in \X} \norm{\textstyle\sum_{j \in B_J} z_j \grad f_j(\x)}_{\X^*} &\leq \Gb_J \norm{\z}_{2}, \quad &\forall \z \in \R^{\abs{B_J}},\label{eq:ind_lip}\\
        \sup_{\y \in \Y} \norm{\textstyle\sum_{j \in B_J} y_j(\grad f_j(\x) - \grad f_j(\x'))}_{\X^*} &\leq \Lb_J \norm{\x - \x'}_{\X}, \quad &\forall \x, \x' \in \X.\label{eq:ind_smth}
    \end{align}
    In addition, for $\grad \fb(\x) := (\grad f_1(\x), \ldots, \grad f_n(\x))^\top \in \R^{n \times d},$ there exist $G \geq 0$ and $L \geq 0$ such that
    \begin{align}
        \sup_{\x \in \X} \norm{\grad \fb(\x)^\top \z}_{\X^*} &\leq G \norm{\z}_{\Y}, \quad &\forall \z \in \R^n,  \label{eq:agg_lip}\\
        \sup_{\y \in \Y}\norm{(\grad \fb(\x) - \grad \fb(\x'))^\top \y}_{\X^*} &\leq L \norm{\x - \x'}_{\X}, \quad &\forall \x, \x' \in \X \label{eq:agg_smth}.
    \end{align}
\end{assumption}
We will state the upcoming results in terms of $\ell_p$-norms of the vectors $\Gbb = (\Gb_1, \ldots, \Gb_N)$ and $\Lbb = (\Lb_1, \ldots, \Lb_N)$. To understand the relative scale of these constants, we note that they can also be estimated using component-wise Lipschitz and smoothness constants. Precisely, if for any $\x, \x' \in \X$ and $j \in [n]$, it holds that $\abs{f_j(\x) - f_j(\x')} \leq \Gc_j \norm{\x - \x'}_\X$, then $\Gb_J \leq \sqrt{\sum_{j \in B_J} \Gc_j^2}$. For $\Lb_J$, suppose that $\norm{\grad f_j(\x) - \grad f_j(\x')}_{\X^*} \leq \Lc_j \norm{\x - \x'}_\X$ and that for $\y = (y_1, \ldots, y_n) \in \Y$, the coordinate $y_j$ is bounded by $\Dc_j < \infty$ when $\Lc_j > 0$ (by convention, $\Lc_j \Dc_j = 0$ whenever $\Lc_j = 0$). %
This ensures that $\Lb_1, \ldots, \Lb_N < +\infty$, as $\Lb_J \leq \sum_{j \in B_J} \Dc_j \Lc_j$. 
To succinctly express our results, we also define the summary vector of Lipschitz constants
\begin{align*}
    \lambb := (\lamb_1, \ldots, \lamb_N)^\top \text{ with } \lamb_I := \sqrt{\Gb_I^2 + \Lb_I^2} \text{ for } I = 1, \ldots, N,
\end{align*}
which can be seen as a block-wise aggregation of the vectors $\Gbb$ and $\Lbb$. 
We place the following standard assumption on the uncoupled components of the objective, where we interpret values of the strong convexity modulus being zero ($\mu = 0$ or $\nu = 0$) as the corresponding function being simply convex. Let $\ri(\cdot)$ denote the relative interior of a set.
\begin{assumption}\label{asm:str_cvx}
    Assume that $\phi$ is proper (with $\X \sse \dom(\phi) := \br{\x \in \R^d: \phi(\x) < +\infty}$ and $\X \cap \ri(\dom(\phi))$ non-empty), closed (i.e., has a closed epigraph in $\R^{d+1}$), and $\mu$-strongly convex ($\mu \geq 0$) with respect to $\norm{\cdot}_\X$, that is, for any $\s \in \partial \phi(\u)$, we have that $\phi(\z) \geq \phi(\u) + \ip{\s, \z - \u} + \frac{\mu}{2}\norm{\z - \u}_\X^2$. Similarly, $\psi$ is proper (with $\Y \sse \dom(\psi)$ and $\Y \cap \ri(\dom(\psi))$ non-empty), closed, and $\nu$-strongly convex ($\nu \geq 0$) with respect to $\norm{\cdot}_\Y$. 
\end{assumption}

Next, we rely upon the following definition of Bregman divergences for possibly non-differentiable distance generating functions. Consider a convex subset $\Z \sse \R^m$, and let $\varphi$ be a proper, closed, and $1$-strongly convex function satisfying $\Z \sse \dom(\varphi)$. 
Define the \emph{Bregman divergence} $\Delta_\varphi: \dom(\varphi) \times \ri(\dom(\varphi)) \rightarrow \R$ as 
\begin{align*}
    \breg{\varphi}{\z}{\z'} := \varphi(\z) - \varphi(\z') - \ip{\grad \varphi(\z'), \z - \z'}.
\end{align*}
The notation $\grad \varphi(\z') \in \partial \varphi(\z')$ denotes an arbitrary, but consistently chosen subgradient at $\z'$ when applied to a convex but possibly non-differentiable function. This slight modification is made for purely technical reasons, as we perform a mirror descent-style analysis with Bregman divergences generated by the (possibly non-smooth) component functions $\phi$ and $\psi$. \Cref{lem:3pt} provides a modified three-point inequality used in both the upper and lower bounds on the objective used in bounding the initial gap estimates. 
\begin{restatable}{lemma}{threepoint}\label[lemma]{lem:3pt}
    Let $h$, $g$, and $\varphi$ be proper, closed, and convex functions whose domains contain $\Z$ and that map to $\R \cup \br{+\infty}$. Assume that $g$ is relatively $\gamma$-strongly convex with respect to $\varphi$ on $\Z$, i.e., $\Delta_g(\u, \z) \geq \gamma\Delta_\varphi(\u, \z)$ for $\gamma \geq 0$ and $\u, \z \in \Z$.  Let $A \geq 0$, $a > 0$, $\gamma_0 > 0$ be constants, and let $\z_1, \ldots, \z_r \in \ri(\dom(\varphi))$. Let
    \begin{align}
        \z^+ = \argmin_{\u \in \Z} \br{m(\u) := h(\u) + a g(\u) + \tfrac{A\gamma + \gamma_0}{2} \textstyle \sum_{i=1}^r w_i \breg{\varphi}{\u}{\z_i}},\label{eq:proximal}
    \end{align}
    where each $w_i \geq 0$ and $\sum_{i=1}^r w_i = 1$.
    Then, for any $\u \in \Z$,
    \begin{align*}
        m(\u) \geq m(\z^+) + \p{\tfrac{(A + a)\gamma + \gamma_0}{2}}\breg{\varphi}{\u}{\z^+} + \tfrac{a\gamma}{2} \breg{\varphi}{\u}{\z^+}. 
    \end{align*}
\end{restatable}
\begin{proof}
    By the definition of the Bregman divergence generated by $m$, we have that
    \begin{align*}
        m(\u) &= m(\z^+) + \ip{\nabla m(\z^+), \u - \z^+} + \Delta_{m}(\u, \z^+),\\
        &\geq m(\z^+) + \Delta_{m}(\u, \z^+),
    \end{align*}
    where we use that $\ip{\nabla m(\z^+), \u - \z^+} \geq 0$ for any subgradient $\nabla m(\z^+)$ as $\z^+ \in \argmin_{\u \in \Z} m(\u)$.
    Then, by using the definition of $m$,
    and that $\Delta_{\Delta_\varphi(\cdot, \z)} = \Delta_\varphi$ for any fixed $\z \in \ri(\dom(\varphi))$, we have that
    \begin{align*}
        m(\u) &\geq m(\z^+) + \Delta_{m}(\u, \z^+)\\
        &= m(\z^+) + \Delta_h(\u, \z^+) + a\Delta_{g}(\u, \z^+) + \frac{A\gamma + \gamma_0}{2} \sum_{i=1}^r w_i \breg{\Delta_{\varphi}(\cdot, \z_i)}{\u}{\z^+}\\
        &= m(\z^+) + \Delta_h(\u, \z^+) + a\Delta_{g}(\u, \z^+) + \frac{A\gamma + \gamma_0}{2} \breg{\varphi}{\u}{\z^+}.
    \end{align*}
    Use $\Delta_h(\u, \z^+) \geq 0$ and then relative strong convexity $a \Delta_g(\u, \z^+) \geq a \gamma\Delta_\varphi(\u, \z^+)$ to prove the desired result.
\end{proof}
Henceforth, we use the notation $\breg{\X}{\cdot}{\cdot}$ to denote the Bregman divergence on $\X$ that is both $1$-strongly convex with respect to $\norm{\cdot}_\X$ and that satisfies $\Delta_\phi \geq \mu \Delta_{\X}$ (i.e., $\phi$ is relatively $\mu$-strongly convex with respect to $\Delta_{\X}$). As an example, this is satisfied by $\Delta_\X := \Delta_{\phi/\mu}$ for $\mu > 0$. We define $\breg{\Y}{\cdot}{\cdot}$ analogously. 
As a technical consideration, we also assume that the unique solution of~\eqref{eq:proximal} lies in $\ri(\dom(\varphi))$, which is satisfied for common choices of the generator $\varphi$. 
Finally, we introduce additional structure on the problem~\eqref{eq:semilinear} which can be exploited when available.
\begin{definition}\label[definition]{def:separable}
    We call $\Lcal$ a \emph{dual-separable objective} if the dual component decomposes as $\psi(\y) = \sum_{J=1}^N \psi_J(\y_{J})$ where $\y_J$ denotes the components of $\y \in \Y$ corresponding to the indices in block $B_J$. We call $\X \times \Y$ a \emph{dual-separable feasible set} if $\Y = \Y_1 \times \ldots \times \Y_N$ and $\y_J \in \Y_J$ for $J = 1, \ldots, N$. We call the problem~\eqref{eq:semilinear} a \emph{dual-separable problem} if its objective and feasible set are both dual-separable.
\end{definition}
Dual-separability of the objective is commonly satisfied, such as when $\psi$ represents an $\ell_2$ or negative entropy penalty. 
In this case, we have that $\breg{\Y}{\y}{\y'} := \sum_{J=1}^N \breg{J}{\y_J}{\y'_J}$, where each $\breg{J}{\cdot}{\cdot}$ is a Bregman divergence on $\dom(\psi_J) \times \ri(\dom(\psi_J))$. As before, $\psi$ is relatively $\nu$-strongly convex with respect to $\Delta_{\Y}$.
In later sections, we may also use the subscript on $\y_k \in \Y$ to denote a particular time index of an algorithm, as opposed to the block index $J$ on $\y_J \in \Y_J$; the difference will be clear from context. Dual-separability of the feasible set is a less common assumption. It is not satisfied, for instance, on simplicial domains such as the one in Example 1. %
Based on this observation, \Cref{sec:nonsep:stochastic} and \Cref{sec:sep} are dedicated to the proposed algorithms for non-separable and separable problems, respectively.

We conclude this section with a discussion of example problems that can be addressed in the considered framework of dual linear min-max optimization problems. 

\begin{example}[Distributionally robust optimization (DRO)]\label{ex:dro}
    The DRO problem can be written as
    \begin{align*}
        \min_{\x \in \X} \max_{\substack{\y \geq 0 \\ \ones^\top \y = 1}} \sum_{j=1}^n y_j f_j(\x)- \psi(\y) + \phi(\x), 
    \end{align*}
    where $\x$ represents the parameters of a predictive model, each $f_j$ represents the loss incurred on training example $j$, and  $\phi$ is a regularizer (e.g., the $\ell_2$-norm). The dual vector $\y \in \Y$ denotes a collection of weights in the $(n-1)$-probability simplex and $\psi(\y)$ penalizes $\y$ from shifting too far from the original uniform weights $\ones / n = (1/n, \ldots, 1/n)$ over the training set. This can either be a soft penalty (e.g., $\psi(\y)$ is the Kullback-Leibler divergence between $\y$ and $\ones/n$) or a hard penalty that further constrains the dual feasible set.
\end{example}

\begin{example}[Fully composite optimization]\label{ex:fco}
    The fully composite optimization problem is defined by
    \begin{align*}
        \min_{\x \in \X} F(\boldsymbol{h}(\x), \x),
    \end{align*}
    where $\boldsymbol{h}: \X \rightarrow \R^{n-d}$ is component-wise convex and $F: \R^n \rightarrow \R \cup \br{+\infty}$ is closed and convex. 
    It is assumed that $\boldsymbol{h}$ is smooth and ``hard'' to compute, whereas $F$ may be non-differentiable but ``easy'' to compute. 
    This formulation can be viewed as a generalization of the typical additive composite problem in which $n - d = 1$ and $F(\boldsymbol{h}(\x), \x) = \boldsymbol{h}(\x) + g(\x)$ for a smooth component $\boldsymbol{h}$ and non-smooth component $g$.
    By taking the Fenchel conjugate $F^*(\y) := \sup_{\z \in \R^{n-d} \times \X} \ip{\z, \y} - F(\z)$ of $F$, we achieve the formulation~\eqref{eq:semilinear} with $\fb(\x) := (\boldsymbol{h}(\x), \x) \in \R^n$, $\psi(\y) = F^*(\y)$, and $\phi(\x) = 0$. To ensure that the overall problem is convex-concave we also assume that $F(\cdot, \x)$ is monotone in that $\u \leq \v \text{ element-wise } \implies F(\u, \x) \leq F(\v, \x)$. 
\end{example}

\begin{example}[Problems with functional constraints]\label{ex:fcm}
    Consider the classical convex minimization problem with constraints defined by sublevel sets of convex functions
    \begin{align*}
        \min_{\x \in \X} \phi(\x) \text{ s.t. } f_j(\x) \leq 0 \text{ for all } j = 1, \ldots, n.
    \end{align*}
    Then, the Lagrangian formulation yields the expression~\eqref{eq:semilinear} by letting $\y \in \Y = \R^n_+$ denote the Lagrange multipliers and setting $\psi \equiv 0$. The objective~\eqref{eq:semilinear} also encompasses the related setting of ``soft'' functional constraints, where we set $\psi$ as any $\nu$-strongly convex function $\psi$ with $\nu > 0$ to produce a faster convergence rate at an approximation cost governed by the parameter $\nu$.
    In this case, the primal solution resulting from this smoothed problem may only approximately satisfy the functional constraints.
\end{example}

\begin{example}[Maximal eigenvalue minimization]\label{ex:eigenvalue}
    Given a collection of $d$ symmetric matrices $\A_1, \ldots, \A_d \in \R^{m \times m}$, the classical problem of minimizing the maximal eigenvalue $\lambda_{\max{}}(\sum_{i=1}^d x_i \A_i)$ in $\x = (x_1, \ldots, x_d) \in \R^d$ can be formulated \citep{nesterov2007smoothing, baes2013randomized} as the saddle-point problem 
    \begin{align*}
        \min_{\x \in \X} \max_{\y \in \Y} \sum_{i=1}^d \operatorname{tr}\p{(\A_i x_i) \y},
    \end{align*}
    where $\Y$ is the set of positive semi-definite matrices satisfying $\operatorname{tr}(\y) = 1$ and $\X$ is any convex, compact subset of $\R^d$. Here, $n = m^2$ depends quadratically on the height/width of the matrices $(\A_i)_{i=1}^d$, and blocks may naturally correspond to matrix structure, such as rows or columns. This constitutes a non-separable matrix game, a specific problem class discussed in \Cref{sec:discussion}. 
\end{example}

\section{Our Method and a General Analysis Template}\label{sec:template}
We consider an algorithm to be a sequence of primal-dual iterates $(\x_k, \y_k)_{k\geq 0}$, 
with fixed initial point $(\x_0, \y_0)$. This sequence may be random, in which case the relevant probabilistic information is introduced when we analyze stochastic algorithms. As is standard for primal-dual algorithm of the same type \citep{chambolle2011afirstorder,alacaoglu2020random,song2021variance}, the analysis tracks the \emph{primal-dual gap} function %
at fixed $\u \in \X$ and $\v \in \Y$:
\begin{align}
    \gap^{\u, \v}(\x, \y) := \Lcal(\x, \v) - \Lcal(\u, \y).
    \label{eq:gap}
\end{align}
To measure suboptimality in the strongly convex case (when $\mu > 0$ and $\nu > 0$), we set $\u = \x_\star$ and $\v = \y_\star$, where $(\x_\star, \y_\star)$ denotes the unique saddle point of~\eqref{eq:semilinear}, satisfying
\begin{align*}
    \Lcal(\x_\star, \y) +\frac{\nu}{2}\norm{\vy - \vy_\star}_\Y^2 \leq \Lcal(\x_\star, \y_\star) \leq \Lcal(\x, \y_\star) -\frac{\mu}{2}\norm{\vx - \vx_\star}_\X^2
\end{align*}
for all $(\x, \y) \in \X \times \Y$. When $\mu = \nu = 0$, we measure suboptimality using the supremum of (possibly the expectation of)~\eqref{eq:gap} over $\u \in \mc{U}$ and $\v \in \mc{V}$ for compact sets $\mc{U} \sse \X$ and $\mc{V} \sse \Y$. When $\mu = 0$ or $\nu = 0$ exclusively, we naturally take the supremum over $\u$ or $\v$ exclusively.

We first fix $\u, \v$ and aim to show that $\limsup_{t \rightarrow \infty}\gap^{\u, \v}(\x_t, \y_t) \leq 0$ (possibly in expectation), with a convergence rate in terms of the problem constants from \Cref{sec:preliminaries} and iteration count $t \geq 0$. To this end, we introduce an averaging sequence $(a_k)_{k \geq 1}$ of positive constants with $a_0 = 0$, and their aggregation $A_t := \sum_{k=0}^t a_k$, and then pursue an upper bound of the form
\begin{align}
    \sum_{k=1}^t a_k \gap^{\u, \v}(\x_k, \y_k) \leq G_0(\u, \v),
    \label{eq:gap_bound}
\end{align}
where $G_0(\u, \v)$ is a constant independent of $t$, so that when dividing by $A_t$, the average expected gap decays at rate $A_t^{-1}$. Accordingly, we wish for $A_t$ to grow as fast as possible with $t$. By way of convexity, we have that $\gap^{\u, \v}(\tx_t, \ty_t) \leq G_0(\u, \v)/A_t$ for $(\tx_t, \ty_t) := A_t^{-1} \sum_{k=1}^t a_k (\x_k, \y_k)$, which can be returned by the algorithm to realize the gap bound~\eqref{eq:gap_bound}. We may conclude by taking the supremum of $G_0(\u, \v)$. In the randomized scenarios, we also discuss in \Cref{sec:discussion} how the analysis can be adapted to an even stronger convergence criterion for which the supremum is taken prior to the expectation. Because many of the technical ideas remain similar when proving convergence with respect to this stronger criterion, we describe only the parts that change in \Cref{sec:discussion}.

\begin{algorithm}[t]
\setstretch{1.2}
   \caption{Template Method}
   \label{algo:drago_v2}
    \begin{algorithmic}[1]
       \State {\bfseries Input:} Initial point $(\x_0, \y_0)$, averaging sequence $(a_k)_{k=0}^t$, non-negative weights $(\gamma_I)_{I=1}^N$ that sum to one, balancing sequences $(\wP_k)_{k=1}^t$ and $(\wD_k)_{k=1}^t$, functions $\sub_1$, $\sub_2$, and $\sub_3$.
       \State Initialize the comparison points $\red{\hxb_{0, I}} = \x_0$ for all $I \in [N]$ and $\red{\hyb_0} = \y_0$.
       \For {$k=1$ {\bfseries to} $t$}
           \State $\sub_1$: Compute $\blue{\g_{k-1}}$ using stored information and oracle calls to $\grad f_i(\x_{k-1})$, $i \in [n]$%
           \State Perform the primal update
           \begin{align}
           \x_k &= \argmin_{\x \in \X} \Big\{a_k \ip{\blue{\g_{k-1}}, \x} + a_k\phi(\x) \notag \\
           &\quad + (A_{k-1}\mu + \mu_0)\Big(\underbrace{\tfrac{1-\wP_k}{2}\breg{\X}{\x}{\x_{k-1}}}_{\text{standard proximity term}} + \underbrace{\tfrac{\wP_k}{2}\textstyle\sum_{I=1}^N \gamma_I \breg{\X}{\x}{\red{\hxb_{k-1, I}}}}_{\text{primal historical regularization}}\Big)\Big\}.\label{eq:prim_update}
           \end{align}
           \State $\sub_2$: Compute $\blue{\bfb_{k-1/2}}$ using stored information and some calls to $f_1(\x_{k}), \ldots, f_n(\x_{k})$.
           \State Perform the dual update
           \begin{align}
               \y_k &= \argmax_{\y \in \Y} \Big\{ a_k \ip{\y, \blue{\bfb_{k-1/2}}} - a_k\psi(\y) \notag\\ 
               &\quad - (A_{k-1}\nu + \nu_{0}) \Big(\underbrace{\tfrac{1-\wD_k}{2} \breg{\Y}{\y}{\y_{k-1}}}_{\text{standard proximity term}} + \underbrace{\tfrac{\wD_k}{2} \breg{\Y}{\y}{\red{\hyb_{k-1}}}}_{\text{dual historical regularization}}\Big)\Big\}. \label{eq:dual_update}
           \end{align}
           \State $\sub_3$: Update comparison points $(\red{\hxb_{k, I}})_{I=1}^N$ and $\red{\hyb_k}$.
       \EndFor
       \Return $A_t^{-1} \sum_{k=1}^t a_k(\x_k, \y_k)$.
    \end{algorithmic}
\end{algorithm}

For any algorithm we consider, the analysis will proceed by constructing an upper bound on $a_k\gap^{\u, \v}(\x_k, \y_k)$ containing telescoping and non-positive terms, by first lower bounding $a_k\Lcal(\u, \y_k)$ and upper bounding $a_k\Lcal(\x_k, \v)$. As we will see, the update for $\x_k$ will be used to produce the lower bound while the update for $\y_k$ will be used to produce the upper bound. The update rules are motivated directly by the analysis. Similar steps will take place in the stochastic setting, except using the expectation of~\eqref{eq:gap} under algorithmic randomness. We start with an arbitrary point $(\x_0, \y_0) \in \ri(\dom(\phi))\times \ri(\dom(\psi))$. The parameters $\mu_0 > 0$ and $\nu_0 > 0$ appearing below are employed in order to handle the strongly convex and non-strongly convex settings in a unified manner. 

This general, high-level idea and analysis template are in line with prior work providing constructive arguments for the analysis of optimization methods \citep{diakonikolas2019approximate,mehta2024drago,li2024learning,diakonikolas2025block}. As such, it provides a clear guiding principle for the analysis and motivation for the algorithmic choices. It is of note, however, that while the general principle is common to all these works, the specifics of the analysis and associated algorithms differ significantly, as the technical obstacles they need to address are problem-specific. For instance, the stochastic algorithms in the present work have a unique combination of randomized updates and historical regularization that are conceptually novel and interesting in their own right. It is the combination of \emph{both} non-uniform sampling and non-uniform regularization that leads to complexity improvements even in more specific problem classes such as bilinearly coupled problems. For the separable case, our proof technique relies on an auxiliary sequence of dual variables that offers an extension of previous meta-analyses when using coordinate-wise updates.  

\subsection{Algorithm Template and a First Gap Bound}\label{sec:template:first}
We fix the convention that the update for $\x_k$ occurs before the update for $\y_k$. Thus, we require that only information available up to and including time $k-1$ is used in the update. In the stochastic setting, this requirement will be formalized in the language of measurability. Both the primal and dual updates will resemble those of a proximal gradient-type algorithm, wherein $\x_k$ and $\y_k$ are defined by minimizing or maximizing an approximation of~\eqref{eq:semilinear} with a proximity term (i.e.~a Bregman divergence). In the primal update, the proximity term promotes $\x_k$ being close to not only $\x_{k-1}$, but several additional to-be-specified comparison points $\hxb_{k-1, 1}, \ldots, \hxb_{k-1, N}$. Similarly, $\y_k$ will be made close to $\y_{k-1}$ along with a single comparison point $\hyb_{k-1}$. The use of multiple comparison points is in fact the motivation for the modified three-point inequality (\Cref{lem:3pt}). The remaining components to specify are $\g_{k-1} \in \R^d$, a vector which will be used to linearize the objective~\eqref{eq:semilinear} in the primal update, and $\bfb_{k-1/2} \in \R^n$, an analogous vector used in the dual update. The subscript $k-1/2$ indicates that ``half'' of the information in iteration $k$ (namely, the value of $\x_k$) can be used in the update. The components introduced so far generate a template algorithm which can combine them in various ways; the pseudocode for this method is shown in \Cref{algo:drago_v2}, in which the algorithm-specific content is abstracted into three subroutines.

The hyperparameters are shown explicitly to better accompany the theoretical analysis. They are set to specific values over the course of the proofs. 
We motivate the updates~\eqref{eq:prim_update} and~\eqref{eq:dual_update} using a lower bound on $a_k\Lcal(\u, \y_k)$ and an upper bound on $a_k\Lcal(\x_k, \v)$. Several terms will appear that either telescope when summed or are used to cancel errors incurred at each iteration. For the reader's convenience, we summarize this notation below.
\begin{mdframed}[userdefinedwidth=\linewidth, align=center, linewidth=0.3mm]
    \textbf{Notation:} Throughout the analyses for each algorithm, the primal and dual ``distance-to-opt'' terms will be written as
    \begin{align}
        \TP_k := (A_{k}\mu + \mu_0)\breg{\X}{\u}{\x_{k}}, \quad \TD_k := (A_{k}\nu + \nu_0)\breg{\Y}{\u}{\y_{k}},\label{eq:telescoping_terms}
    \end{align}
    indicating that they telescope when summed.
    The bounds also produce the negation of the terms
    \begin{align}
        \CP_k := (A_{k-1}\mu + \mu_0)\breg{\X}{\x_k}{\x_{k-1}}, \quad \CD_k := (A_{k-1}\nu + \nu_0)\breg{\Y}{\y_k}{\y_{k-1}},\label{eq:cancellation_terms}
    \end{align}
    which appear in the gap function bound and are used for canceling errors that appear when controlling the primal-dual gap. Analogous terms appear based on the comparison points instead of the iterates. That is, consider for each $I \in [N]$ the additional telescoping terms
    \begin{align}
        \hTP_{k, I} := (A_k\mu + \mu_0)\breg{\X}{\u}{\hxb_{k, I}}, \quad \hTD_k := (A_{k}\nu + \nu_0)\breg{\Y}{\v}{\hyb_{k}},\label{eq:stoch:telescoping_terms}
    \end{align}
    along with similar cancellation terms
    \begin{align}
        \hCP_{k, I} := (A_{k-1}\mu + \mu_0)\breg{\X}{\x_k}{\hxb_{k-1, I}}, \quad \hCD_k := (A_{k-1}\nu + \nu_0)\breg{\Y}{\y_k}{\hyb_{k-1}}.\label{eq:stoch:cancellation_terms}
    \end{align}
    Finally, we will also define the inner product terms
    \begin{align}
        \IP_k = a_k \ip{\grad \fb(\x_k)^\top \y_k - \g_{k-1}, \u - \x_k}, \quad \ID_k = a_k\ip{\fb(\x_k) - \bfb_{k-1/2}, \v - \y_k}, \label{eq:inner_prod_terms}
    \end{align}
    which will comprise the errors that are cancelled by the terms above.
\end{mdframed}

To achieve the lower bound, we compare the objective to the one minimized by $\x_k$ and use properties of Bregman divergences to produce telescoping terms akin to a mirror descent-style analysis. 
\begin{lemma}\label[lemma]{lem:lower}
    For any $k \geq 1$, let $\g_{k-1} \in \R^d$ and $\wP_k, \wP_{k-1} \in [0, 1)$ such that $\wP_{k-1} \leq \wP_{k}$.
    For $k \geq 1$, if $\x_k$ is defined via the update~\eqref{eq:prim_update}, then it holds that
    \begin{align}
        a_k\Lcal(\u, \y_k)
        &\geq a_k \mc{L}(\x_k, \y_k)  + \IP_k \label{eq:stoch:inner_prod_primal}\\
        &\quad + \p{\tfrac{1-\wP_k}{2}\TP_k - \tfrac{1-\wP_{k-1}}{2}\TP_{k-1}} + \tfrac{\wP_k}{2}\p{\TP_k -  \textstyle\sum_{I=1}^N \gamma_I\hTP_{k-1, I}}\notag\\
        &\quad + \tfrac{1-\wP_k}{2}\CP_k + \tfrac{\wP_k}{2}\textstyle\sum_{I=1}^N \gamma_I\hCP_{k, I} + \frac{a_k \mu}{2} \breg{\X}{\u}{\x_k}. \label{eq:stoch:telescope_primal}
    \end{align}
\end{lemma}
\begin{proof}
    Because $\y_k$ is observed in $\Y$, we have that $\x \mapsto \ip{\y_k, \fb(\x)}$ is convex and differentiable. As a result,
    \begin{align}
        &a_k\Lcal(\u, \y_k)
        = a_k(\ip{\y_{k}, \fb(\u)} - \psi(\y_k) + \phi(\u)) \notag\\
        &\geq a_k\ip{\y_k, \fb(\x_k) + \grad \fb(\x_k)(\u - \x_k)} - a_k\psi(\y_k) + a_k\phi(\u). \label{eq:lower:cvx}
    \end{align}
    Then, add and subtract terms from the objective defining~\eqref{eq:prim_update}:
    \begin{align*}
        a_k\Lcal(\u, \y_k) &\geq a_k \ip{\g_{k-1}, \u - \x_k} + a_k \phi(\u) + a_k \ip{\grad \fb(\x_k)^\top \y_k - \g_{k-1}, \u - \x_k} \\
        &\quad + \tfrac{1-\wP_k}{2}(A_{k-1}\mu + \mu_0)\breg{\X}{\u}{\x_{k-1}} -  \tfrac{1-\wP_k}{2}\underbrace{(A_{k-1}\mu + \mu_0)\breg{\X}{\u}{\x_{k-1}}}_{\TP_{k-1} \text{ from~\eqref{eq:telescoping_terms}}} \\
        &\quad + \tfrac{\wP_k}{2}(A_{k-1}\mu + \mu_0) \sum_{I=1}^N \gamma_I \breg{\X}{\u}{\hxb_{k-1, I}} - \tfrac{\wP_k}{2}\underbrace{(A_{k-1}\mu + \mu_0) \sum_{I=1}^N \gamma_I \breg{\X}{\u}{\hxb_{k-1, I}}}_{\sum_I \gamma_I \hTP_{k-1, I} \text{ from~\eqref{eq:stoch:telescoping_terms}}}\\
        &\quad + a_k \ip{\y_k, \fb(\x_k)} - a_k\psi(\y_k).
    \end{align*}
    When applying \Cref{lem:3pt} with $A = A_{k-1}$, $a = a_k$, $\gamma = \mu$, and $\gamma_0 = \mu_0$ and using the definitions of $\CP_k$ from~\eqref{eq:cancellation_terms} and $\hCP_{k, I}$ from~\eqref{eq:stoch:cancellation_terms}, we achieve the inequality
    \begin{align*}
        a_k\Lcal(\u, \y_k) &\geq a_k \phi(\x_k) + a_k \ip{\grad \fb(\x_k)^\top \y_k - \g_{k-1}, \u - \x_k} \\
        &\quad + \tfrac{1-\wP_k}{2}\CP_{k} -  \tfrac{1-\wP_k}{2}\TP_{k-1} \\
        &\quad + \tfrac{\wP_k}{2}\textstyle\sum_{I=1}^N \gamma_I\hCP_{k, I} - \tfrac{\wP_k}{2}\sum_I \gamma_I \hTP_{k-1, I}\\
        &\quad + a_k \ip{\y_k, \fb(\x_k)} - a_k\psi(\y_k)\\
        &\quad + \tfrac{1}{2}\TP_{k} + \tfrac{a_k \mu}{2}\breg{\X}{\u}{\x_k}.
    \end{align*}
    Decompose $\tfrac{1}{2}\TP_{k} = \tfrac{1 - \wP_k}{2}\TP_{k} + \tfrac{\wP_k}{2}\TP_{k}$. Finally, use $ -\tfrac{1-\wP_{k}}{2}\TP_{k-1} \geq -\tfrac{1-\wP_{k-1}}{2}\TP_{k-1}$, the substitution $\Lcal(\x_k, \y_k) = \ip{\y_k, \fb(\x_k)} - \psi(\y_k) + \phi(\x_k)$, and the definition of $\IP_k$ from~\eqref{eq:inner_prod_terms} to complete the proof.
\end{proof}
The upper bound is proved in a nearly identical fashion, except the step that employs convexity in~\eqref{eq:lower:cvx} is not used; the proof is omitted for brevity.
\begin{lemma}\label[lemma]{lem:upper}
    For any $k \geq 1$, let $\bfb_{k-1/2} \in \R^n$ and $\wD_k, \wD_{k-1} \in (0, 1)$ such that $\wD_{k-1} \leq \wD_k$. For $k \geq 1$, if $\y_k$ is defined via the update~\eqref{eq:dual_update}, then it holds that
    \begin{align}
        a_k\Lcal(\x_k, \v)
        &\leq a_k \mc{L}(\x_k, \y_k) + \ID_k \notag\\
        &\quad + \p{\tfrac{1-\wD_{k-1}}{2}\TD_{k-1} - \tfrac{1-\wD_k}{2}\TD_k} + \tfrac{\wD_k}{2}\p{\hTD_{k-1} - \TD_k} \notag\\
        &\quad - \tfrac{1-\wD_k}{2}\CD_k  - \tfrac{\wD_k}{2}\hCD_k - \tfrac{a_k\nu}{2} \breg{\Y}{\v}{\y_k}. \label{eq:stoch:cancellation_dual}
    \end{align}
\end{lemma}
Combining the derived upper and lower bounds and canceling matching terms, we claim the following result. This bound will be the starting point for every subsequent result in the paper, so we reference it heavily in the sequel.
\begin{mdframed}[userdefinedwidth=0.9\linewidth, align=center, linewidth=0.3mm]
\begin{claim}\label{claim:gap}
    Using the notation from~\eqref{eq:telescoping_terms},~\eqref{eq:cancellation_terms},~\eqref{eq:stoch:telescoping_terms},~\eqref{eq:stoch:cancellation_terms}, and~\eqref{eq:inner_prod_terms}, we have
    \begin{align}
        &\sum_{k=1}^t a_k\gap^{\u, \v}(\x_k, \y_k) \leq \p{\tfrac{1-\wP_0}{2}\TP_{0} - \tfrac{1-\wP_t}{2}\TP_{t}} + \p{\tfrac{1-\wD_0}{2}\TD_{0} - \tfrac{1-\wD_t}{2}\TD_{t}} \notag \\
        &\quad + \sum_{k=1}^{t} \sbr{\ID_k - \IP_k} - \sbr{\tfrac{1-\wP_k}{2}\CP_k + \tfrac{1-\wD_k}{2}\CD_k + \tfrac{\wP_k}{2}\textstyle\sum_I \gamma_I \hCP_{k, I} + \tfrac{\wD_k}{2}\hCD_k} \label{eq:stoch:gap2_cancellation} \\
        &\quad + \tfrac{1}{2}\sum_{k=1}^t \sbr{\wP_k\p{\textstyle\sum_{I=1}^N \gamma_I \hTP_{k-1, I} - \TP_{k}}- a_k\mu\breg{\X}{\u}{\x_k}} \label{eq:stoch:gap3_primal}\\
        &\quad + \tfrac{1}{2}\sum_{k=1}^t \sbr{\wD_k\p{\hTD_{k-1} - \TD_{k}} - a_k\nu\breg{\Y}{\v}{\y_k}}  \label{eq:stoch:gap3_dual}
    \end{align}
\end{claim}
\end{mdframed}

The remaining work will be to bound the inner product terms $\IP_k$ and $\ID_k$ by quantities that either telescope or can be cancelled by the remaining terms within~\eqref{eq:stoch:gap2_cancellation}. Then, if $\wP_k > 0$ and $\wD_k > 0$ for any $k$, we will also control~\eqref{eq:stoch:gap3_primal} and~\eqref{eq:stoch:gap3_dual}. For these two lines, we will bound the entire sum over $k$ by a term that does not grow with $t$. 
For all three lines, this will rely both on the specific form of $\g_{k-1}$ and $\bfb_{k-1/2}$ (which we call the primal and dual gradient estimates) and the growth conditions placed on the sequence  $(a_k)_{k \geq 1}$. While these steps may also be modified slightly in the separable setting (see \Cref{def:separable}), the format of the analysis remains the same. 
The growth of the $(a_k)_{k \geq 1}$ sequence (e.g., constant, polynomial, exponential) is determined if the user knows whether the objective is convex or strongly convex and concave or strongly concave in the primal and dual variables, respectively. Regarding hyperparameters, we consider some variants that do not use the historical regularization by setting $\wP_k$ and $\wD_k$ to zero in~\eqref{eq:prim_update} and~\eqref{eq:dual_update}, meaning that $(\gamma_I)_{I=1}^N$ is no longer a hyperparameter that needs to be set. Thus, the number of hyperparameters decreases considerably for each of the cases in \Cref{sec:nonsep:stochastic} and \Cref{sec:sep}. 

The comparison points (when used) are reflective of SAGA-style variance reduction methods \citep{defazio2014saga, palaniappan2016stochastic}. In general, they are snapshots of previous iterates and may be used not only to define proximity terms but in the definitions of $\g_{k-1}$ and $\bfb_{k-1/2}$ as well. By convention, we take any time-dependent element at a negative index to be equal to its initial value (indexed by zero). The number of points stored for the primal updates is equal to the number of blocks $N$, which may be much smaller than $n$ (avoiding the $O(nd)$ complexity of SAGA). There is only a single additional comparison point in the dual update, incurring a storage cost of $O(n)$. 
Finally, although the updates include the strong convexity parameters $\mu$ and $\nu$, this choice is for readability when proceeding through the analysis. Practically, the growth of the sequence $(A_k)_{k \geq 1}$ will be derived in terms of $k$ with an unknown constant scaling. This constant is a hyperparameter to be searched by the algorithm. For example, when $\mu > 0$ and $\nu > 0$, we have that $a_{k+1} = \alpha A_{k} = \alpha \sum_{i=1}^k a_i$ for $k \geq 1$, where $\alpha > 0$ will be a tunable hyperparameter whose optimal value depends on $\mu$ and $\nu$. We conclude this section with a deterministic full vector update method to provide intuition and work with \Cref{claim:gap}, whereas more advanced algorithms are presented in \Cref{sec:nonsep:stochastic} and \Cref{sec:sep}.

\subsection{Warm-up: deterministic algorithm}\label{sec:nonsep:full}
Here, we allow the algorithm to access all first-order oracles $\br{(f_i, \grad f_i)}_{i=1}^n$ in each iteration, for a total per-iteration cost of $O(nd)$. \edit{We first state this deterministic full-batch algorithm in its entirety, which boils down to replacing the subroutines from \Cref{algo:drago_v2} with formulas for the gradient estimates $\g_{k-1}$ and $\bfb_{k-1/2}$, determining the sequence $\br{(a_k, \wP_k, \wD_k)}_{k \geq 1}$, and updating the comparison points (if any). We will first set $\wP_k = \wD_k = 0$ for simplicity, which eschews the need for comparison points. We then motivate our choices of $(\g_{k-1}, \bfb_{k-1/2})$ using the analysis template from \Cref{sec:template:first}. We leave the discussion of $(a_k)_{k \geq 1}$ to accompany the main results and their proofs. Consider \Cref{algo:full_vector}.}

\begin{algorithm}[t]
\setstretch{1.2}
   \caption{Full Vector Update Version of \Cref{algo:drago_v2}}
   \label{algo:full_vector}
    \begin{algorithmic}[1]
   \State {\bfseries Input:} Initial point $(\x_0, \y_0)$, averaging sequence $(a_k)_{k=0}^t$.
   \For {$k=1$ {\bfseries to} $t$}
       \State $\sub_1$: Compute $\g_{k-1}$ via the formula
       \begin{align}
            \g_{k-1} = \grad \fb(\x_{k-1})^\top \y_{k-1} + \frac{a_{k-1}}{a_k}\p{\grad \fb(\x_{k-1})^\top \y_{k-1} - \grad \fb(\x_{k-2})^\top \y_{k-2}}. \label{eq:full:grad}
        \end{align}
       \State Perform the primal update (with $\wP_k$ from~\eqref{eq:prim_update} set to zero):
       \begin{align}
           \x_k &= \argmin_{\x \in \X} \ a_k \ip{\g_{k-1}, \x} + a_k\phi(\x) + \tfrac{A_{k-1}\mu + \mu_0}{2}\breg{\X}{\x}{\x_{k-1}}.\notag
       \end{align}
       \State $\sub_2$: Compute $\bfb_{k-1/2} = \fb(\x_k)$.
       \State Perform the dual update (with $\wD_k$ from~\eqref{eq:dual_update} set to zero):
       \begin{align}
           \y_k &= \argmax_{\y \in \Y} \ a_k \ip{\y, \bfb_{k-1/2}} - a_k\psi(\y) - \tfrac{A_{k-1}\nu + \nu_{0}}{2}  \breg{\Y}{\y}{\y_{k-1}}. \notag
       \end{align}
       \State $\sub_3$: None for this algorithm.
   \EndFor
   \Return $A_t^{-1} \sum_{k=1}^t a_k(\x_k, \y_k)$.
\end{algorithmic}

\end{algorithm}

\edit{Recall that our goal was to employ \Cref{claim:gap} as the main tool to control the primal-dual gap. With $\wP_k = \wD_k = 0$, we are left to control the terms $\IP_k$ and $\ID_k$ from~\eqref{eq:stoch:gap2_cancellation}. Accordingly, by substituting~\eqref{eq:full:grad} into $\IP_k$, we have that}
\begin{align}
    \IP_k &= a_k \ip{\grad \fb(\x_k)^\top \y_k - \grad \fb(\x_{k-1})^\top \y_{k-1}, \u - \x_k} \notag\\
    &\quad- a_{k-1} \ip{\grad \fb(\x_{k-1})^\top \y_{k-1} - \grad \fb(\x_{k-2})^\top \y_{k-2}, \u - \x_{k}}\notag \\
    &= a_k \ip{\grad \fb(\x_k)^\top \y_k - \grad \fb(\x_{k-1})^\top \y_{k-1}, \u - \x_k} \notag\\
    &\quad- a_{k-1} \ip{\grad \fb(\x_{k-1})^\top \y_{k-1} - \grad \fb(\x_{k-2})^\top \y_{k-2}, \u - \x_{k-1}} - \EP_k\label{eq:three_term_decomp}
\end{align}
where we define the error term
\begin{align}
    \EP_k = a_{k-1} \ip{\grad \fb(\x_{k-1})^\top \y_{k-1} - \grad \fb(\x_{k-2})^\top \y_{k-2}, \x_{k-1} - \x_{k}}. \label{eq:full:error}
\end{align}
\edit{Furthermore, because $\x_k$ can be used in the definition of $\bfb_{k-1/2}$ and $O(nd)$ operations are permitted in this algorithm, we may simply set $\bfb_{k-1/2} = \fb(\x_k)$ to make $\ID_k$ vanish. Then, the right-hand side of \Cref{claim:gap} reads as
}
\begin{align}
    &\sum_{k=1}^t a_k\gap^{\u, \v}(\x_k, \y_k) \leq \frac{1}{2}\TP_{0} - \frac{1}{2}\TP_{t} + \frac{1}{2}\TD_{0} - \frac{1}{2}\TD_{t}\notag \\
    &\quad + \sum_{k=1}^{t} \EP_k - a_t \ip{\grad \fb(\x_t)^\top \y_t - \grad \fb(\x_{t-1})^\top \y_{t-1}, \u - \x_t} - \tfrac{1}{2}\sum_{k=1}^{t} \p{\CP_k + \CD_k}. \label{eq:full:gap}
\end{align}
Evidently, the goal is for the sum of terms in~\eqref{eq:full:gap} to be a quantity that does not grow with $t$. A key step in this proof (and the proofs of upcoming results) will be to bound above $\EP_k$ such that the resulting terms can be canceled by $\CP_k$ and $\CD_k$ from~\eqref{eq:cancellation_terms}. \edit{Because $\ID_k = 0$ for our choice of $\bfb_{k-1/2}$, there is no $\ED_k$ term that reflects $\EP_k$ on the dual-side; however, this may appear in other algorithms in which we do not use the choice $\bfb_{k-1/2} = \fb(\x_k)$. Accordingly, we consider the terms $\EP_k$ and $\ED_k$ as recurring characters in the analysis of every algorithm, and wrap their definitions (and the definitions of subroutines from \Cref{algo:drago_v2}) as an identity card for each algorithm in the form below.
}
\renewcommand{\arraystretch}{1.4}
\begin{mdframed}[userdefinedwidth=\linewidth, align=center, linewidth=0.3mm]
\centering
\begin{tabular}{l}
     \textbf{Identity Card 1:} Full vector update method\\
     \toprule
     $\sub_1$: $\g_{k-1} \hspace{-2pt} = \hspace{-2pt} \grad \fb(\x_{k-1})^\top \y_{k-1} \hspace{-2pt} + \hspace{-2pt} \frac{a_{k-1}}{a_k}\p{\grad \fb(\x_{k-1})^\top \y_{k-1} \hspace{-2pt} - \hspace{-2pt} \grad \fb(\x_{k-2})^\top \y_{k-2}}$\\
     $\sub_2$: $\bfb_{k-1/2} = \fb(\x_k)$\\
     $\sub_3$: None\\
     \midrule
     \emph{Primal error:} $\EP_k = a_{k-1} \ip{\grad \fb(\x_{k-1})^\top \y_{k-1} - \grad \fb(\x_{k-2})^\top \y_{k-2}, \x_{k-1} - \x_{k}}$\\
     \emph{Dual error:} $\ED_k = 0$\\
\end{tabular}
\end{mdframed}
\edit{We first state the main result for the complexity of \Cref{algo:full_vector}, which differs based on the whether the objective is strongly convex in the primal, strongly concave in the dual, or both. These conditions also determine the growth rate of the $(a_k)_{k \geq 1}$ sequence required to realize these complexities.}
\begin{thm}\label{thm:nonsep:full}
    Under \Cref{asm:smoothness} and \Cref{asm:str_cvx}, consider any $\u \in \X$, $\v \in \Y$, and precision $\epsilon > 0$. Define the initial distance term
    \begin{align}
        D_0 = \sqrt{\frac{\mu_0}{\nu_0}}\breg{\X}{\u}{\x_0} + \sqrt{\frac{\nu_0}{\mu_0}}\breg{\Y}{\v}{\y_0}. \label{eq:initial_distance}
    \end{align}
    There exists a choice of the sequence $(a_k)_{k=1}^t$ such that \Cref{algo:full_vector} produces an output point $(\tx_t, \ty_t) \in \X \times \Y$ satisfying $\gap^{\u, \v}(\tx_t, \ty_t) \leq \epsilon$ for $t$ that depends on $\epsilon$ according to the following iteration complexities (where $\alpha = \min\br{{\sqrt{\mu\nu}}/\p{\sqrt{2} G}, {\mu}/\p{2 L}}$). 
    \begin{center}
    \begin{adjustbox}{max width=\linewidth}
    \begin{tabular}{ccc}
    \toprule
        {\bf Case} & \edit{{\bf $a_k$ (big-$O$)}} & {\bf Iteration Complexity (big-$O$)}\\
        \midrule
        $\mu > 0$, $\nu > 0$
        &
        \edit{$(1+\alpha)^k$}
        &
        $\p{\frac{L}{\mu} + \frac{G}{\sqrt{\mu\nu}}}\ln\p{\tfrac{\p{L\sqrt{\nu_0/\mu_0}  + G}D_0}{\epsilon}}$
        \\
        $\mu > 0$, $\nu = 0$
        &
        \edit{$k$}
        &
        $\frac{L}{\mu}\ln\p{\tfrac{\p{L\sqrt{\nu_0/\mu_0}  + G}D_0}{\epsilon}} + G\sqrt{\frac{ \sqrt{\mu_0/\nu_0}D_0}{\mu \epsilon}}$
        \\
        $\mu = 0$, $\nu > 0$
        &
        \edit{$1$}
        &
        $\frac{L\sqrt{\nu_0/\mu_0}D_0}{\epsilon} +  G\sqrt{\frac{ \sqrt{\nu_0/\mu_0}D_0}{\nu \epsilon}}$
        \\
        $\mu = 0$, $\nu = 0$
        &
        \edit{$1$}
        &
        $\frac{\p{L\sqrt{\nu_0/\mu_0}  + G}D_0}{\epsilon}$\\
    \bottomrule
    \end{tabular}
    \end{adjustbox}
    \end{center}
\end{thm}
\edit{
\Cref{thm:nonsep:full} describes the required number of iterations of \Cref{algo:full_vector}. We assume that the cost of querying the first-order oracle $\x \mapsto (f_j(\x), \grad f_j(\x))$ is $O(d)$ for any $j = 1, \ldots, n$, and that the optimization problems~\eqref{eq:prim_update} defining the primal update and~\eqref{eq:dual_update} defining the dual update can be solved at $\tilde{O}(d)$ and $\tilde{O}(n)$ cost, respectively. Each iteration requires querying all first-order oracles and furthermore computes matrix-vector products using $n \times d$ matrices, thus the overall arithmetic complexity of each iteration of \Cref{algo:drago_v2} with Identity Card~1 is $\tilde{O}(nd)$.

Returning to the proof template, the upcoming \Cref{prop:nonsep:full} will introduce conditions on $(a_k)_{k \geq 1}$ such that the non-negative terms from~\eqref{eq:full:gap} can be cancelled. The proof of \Cref{thm:nonsep:full}, provided at the end of this section, then shows these conditions can be achieved by providing particular choices of $(a_k)_{k \geq 1}$ for each case.
}
\begin{prop}\label{prop:nonsep:full}
    Let $(\x_0, \y_0) \in \ri(\dom(\phi)) \times \ri(\dom(\psi))$ and $\br{(\x_k, \y_k)}_{k\geq 1}$ be generated by the updates~\eqref{eq:prim_update} and~\eqref{eq:dual_update}, with $\g_{k-1}$ given by~\eqref{eq:full:grad} and $\bfb_{k-1/2} = \fb(\x_k)$. 
    Select $(a_k)_{k \geq 1}$ to satisfy
    \begin{align}
        a_{k} &\leq \frac{1}{4}\min\br{\frac{\sqrt{(A_{k}\mu + \mu_0)(A_{k-1}\nu + \nu_0)}}{G}, \frac{\sqrt{(A_{k}\mu + \mu_0)(A_{k-1}\mu + \mu_0)}}{\sqrt{2}L}},
        \label{eq:full:rate}
    \end{align}
    Recalling the notation of~\eqref{eq:telescoping_terms}, we have that for any $\u \in \X$ and $\v \in \Y$, 
    \begin{align}
        \sum_{k=1}^t a_k \gap^{\u, \v}(\x_k, \y_k) + \tfrac{1}{4}\TP_t + \tfrac{1}{2}\TD_t &\leq \tfrac{1}{2}\p{\TP_0 + \TD_0}.
    \end{align}
\end{prop}
\begin{proof}
    We proceed from steps leading up to the gap bound~\eqref{eq:full:gap}. Consider $k \geq 2$ (as $\EP_k = 0$ for $k \leq 1$).
    Apply Young's inequality with parameter $(A_{k-1}\mu + \mu_0)/4$ and the strong convexity of Bregman divergences to write
    \begin{align}
        \EP_k &\leq \frac{1}{4} \CP_k + \frac{2a_{k-1}^2}{A_{k-1}\mu + \mu_0} \norm{\grad \fb(\x_{k-1})^\top \y_{k-1} - \fb(\x_{k-2})^\top \y_{k-2}}_{\X^*}^2. \notag
    \end{align}
    To bound the second term above, we first decompose it via Young's inequality. Write
    \begin{align}
        &\norm{\grad \fb(\x_{k-1})^\top \y_{k-1} - \grad \fb(\x_{k-2})^\top \y_{k-2}}_{\X^*}^2\notag\\
        &\quad\leq 2\norm{(\grad \fb(\x_{k-1}) - \grad \fb(\x_{k-2}))^\top \y_{k-1}}_{\X^*}^2\notag\\
        &\quad\quad + 2\norm{\grad \fb(\x_{k-2})^\top (\y_{k-1} - \y_{k-2})}_{\X^*}^2\notag\\
        &\quad\leq 4L^2 \breg{\X}{\x_{k-1}}{\x_{k-2}} + 4G^2 \breg{\Y}{\y_{k-1}}{\y_{k-2}}, \label{eq:lg_youngs_ineq}
    \end{align}
    where the last inequality follows by~\eqref{eq:agg_lip} and~\eqref{eq:agg_smth} from \Cref{asm:smoothness}.
    Recall that $\CP_k = (A_{k-1}\mu + \mu_0) \breg{\X}{\x_{k}}{\x_{k-1}}$ and $\CD_k = (A_{k-1}\nu + \nu_0) \breg{\Y}{\y_{k}}{\y_{k-1}}$ (see~\eqref{eq:cancellation_terms}).
    Combining the steps above for $k \geq 2$ and the condition~\eqref{eq:full:rate} yields
    \begin{align}
        \EP_k &\leq \frac{1}{4}\CP_k + \frac{8a_{k-1}^2}{(A_{k-1}\mu + \mu_0)} \sbr{\frac{L^2\CP_{k-1}}{A_{k-2}\mu + \mu_0} + \frac{G^2\CD_{k-1}}{A_{k-2}\nu + \nu_0}} \label{eq:full:lip}\\
        &\leq \frac{1}{4}\CP_k + \frac{1}{4}\CP_{k-1} + \frac{1}{2}\CD_{k-1}. \notag
    \end{align}
    Note that for the case of $k = 2$, our choice of $a_1$ satisfies~\eqref{eq:full:rate} as well.
    Summing up the current gap bound over $k = 1, \ldots, t$ and dropping non-positive terms yields
    \begin{align}
        \sum_{k=1}^t a_k\gap^{\u, \v}(\x_k, \y_k) &\leq \tfrac{1}{2}\TP_0 - \tfrac{1}{2}\TP_t + \tfrac{1}{2}\TD_0 - \tfrac{1}{2}\TD_t - \p{\tfrac{1}{4} \CP_t + \tfrac{1}{2}\CD_t} \notag\\
        &\quad - a_t \ip{\grad \fb(\x_t)^\top \y_t - \grad \fb(\x_{t-1})^\top \y_{t-1}, \u - \x_t}.\label{eq:full:gap1}
    \end{align}
    For the remaining inner product term, we apply Young's inequality with parameter $(A_{t}\mu + \mu_0)/2$ and apply a similar argument as for~\eqref{eq:full:lip}:
    \begin{align*}
        &a_t \ip{\grad \fb(\x_t)^\top \y_t - \grad \fb(\x_{t-1})^\top \y_{t-1}, \u - \x_t} \\
        &\quad= -a_t \ip{\grad \fb(\x_t)^\top \y_t - \grad \fb(\x_{t-1})^\top \y_{t-1}, \u - \x_t} \\
        &\quad\leq \frac{2a_t^2}{A_{t}\mu + \mu_0}\norm{\grad \fb(\x_t)^\top \y_t - \grad \fb(\x_{t-1})^\top \y_{t-1}}_{\X^*}^2 + \frac{1}{4}\TP_t\\
        &\quad\leq \frac{1}{4}\CP_t + \frac{1}{2}\CD_t  + \frac{1}{4}\TP_t,
    \end{align*}
    which will each be cancelled by terms in~\eqref{eq:full:gap1}, leading to the claimed bound.
\end{proof}
To convert the convergence guarantee in \Cref{prop:nonsep:full} into a complexity result, we consider the four possible cases for whether $\mu = 0$ and/or $\nu = 0$, which now allows us to prove \Cref{thm:nonsep:full}.
\begin{proof}
    We first determine the growth of the sequence $A_t$, so that $A_t^{-1}$ gives the convergence rate in terms of the number of iterations. The growth rate can be derived by providing a sequence $(a_k)_{k \geq 0}$ such that~\eqref{eq:full:rate} is satisfied. 
    For the dependence of the required number of iterations on the suboptimality parameter $\epsilon$, we write $G_0 A_t^{-1} \leq \epsilon$ and solve for $t$, noting that $G_0 = \mu_0\breg{\X}{\u}{\x_0} + \nu_0 \breg{\Y}{\v}{\y_0}$, \edit{and that $\gap^{\u, \v}(\tx_t, \ty_t) \leq G_0 A_t^{-1}$ by convexity}. In all cases, set $a_1 = (1/4)\min\br{{\sqrt{\mu_0\nu_0}}/\p{G}, {\mu_0}/\p{\sqrt{2}L}}$ so that the condition~\eqref{eq:full:rate} is satisfied, and 
    \begin{align*}
        \frac{G_0}{a_1} = 4\max\br{\sqrt{2}L\sqrt{\frac{\nu_0}{\mu_0}}, G}D_0 = O\p{\p{L\sqrt{\tfrac{\nu_0}{\mu_0}} + G}D_0}.
    \end{align*}
    Although the $(a_k)_{k\geq 1}$ values are unitless, the relative quantity can help determine the optimal values for $\mu_0$ and $\nu_0$.
    
    \paragraph{Case 1: $\mu > 0$, $\nu > 0$.} Let $\alpha = \frac{1}{4}\min\br{\frac{\sqrt{\mu\nu}}{G}, \frac{\mu}{\sqrt{2} L}}$. For $k \geq 2$, write $A_k - A_{k-1} = a_k = \alpha \sqrt{A_k A_{k-1}} \geq \alpha A_{k-1}$, which implies that $A_t \geq (1 + \alpha)^t a_1$. Then,
    \begin{align*}
        \frac{G_0}{A_t} \leq \frac{G_0}{a_1(1+\alpha)^t} \lesssim (1+\alpha)^{-t}\p{L\sqrt{\tfrac{\nu_0}{\mu_0}} + G}D_0 \overset{\text{want}}{\leq} \epsilon,
    \end{align*}
    which is satisfied for $t$ at the given big-$O$ order.
    \paragraph{Case 2: $\mu > 0$, $\nu = 0$.} We have that $a_k = \min\br{\frac{\mu \nu_0}{32G^2}k, \frac{\mu}{8 L} A_{k-1}}$ for $k \geq 2$ satisfies the rate condition. Then, there exists a $k^\star \geq 0$ such that $A_t \geq (1+\frac{\mu}{8L})^{t}a_1$ for all $t < k^\star$ and $A_t \geq \frac{\mu \nu_0}{32G^2} \sum_{k=k^\star+1}^t k + (1+\frac{\mu}{8L})^{k^\star}a_1$ for $t \geq k^\star$. To compute the complexity, we consider when either term is dominant. For $(1+\frac{\mu}{8L})^{k^\star}a_1$, we apply the same argument as Case 1. Otherwise, we have that
    \begin{align*}
       \frac{G_0}{A_t} \lesssim \frac{G^2}{\mu \nu_0 t^2}\p{\mu_0\breg{\X}{\u}{\x_0} + \nu_0 \breg{\Y}{\v}{\y_0}} =  \frac{G^2}{\mu t^2}\sqrt{\frac{\mu_0}{\nu_0}}D_0 \overset{\text{want}}{\leq} \epsilon.
    \end{align*}
    \paragraph{Case 3: $\mu = 0$, $\nu > 0$.} For $k \geq 2$, we have that $a_k = \min\br{\frac{\mu_0 \nu}{32G^2}k, \frac{\mu_0}{8L}}$ satisfies the rate condition by direct computation. Thus, for $k \leq k^\star = \frac{8G^2}{\nu L}$, we have that $A_t \geq \frac{\mu_0 \nu t(t+1)}{64G^2}$ (for which we argue similarly to Case 2 above), and otherwise, $A_t \geq \frac{\mu_0 (t - k^\star)}{8L} + \frac{\mu_0 \nu k^\star(k^\star+1)}{64G^2}$ (for which we argue similarly to Case 4 below).
    \paragraph{Case 4: $\mu = 0$, $\nu = 0$.} Here, $a_k$ is equal to a constant, so $A_t = a_1t$. Then, arguing similarly to Case 1, 
    \begin{align*}
        \frac{G_0}{A_t} \leq \frac{G_0}{a_1t} \lesssim t^{-1} \p{L\sqrt{\tfrac{\nu_0}{\mu_0}} + G} D_0 \overset{\text{want}}{\leq} \epsilon,
    \end{align*}
    which is satisfied for $t$ at the given big-$O$ order, completing the proof.
\end{proof}
\edit{The next two sections introduce randomized algorithms for when problem~\eqref{eq:semilinear} is non-separable (\Cref{sec:nonsep:stochastic}) and separable (\Cref{sec:sep}), respectively.}

\section{Stochastic Algorithms for General Objectives}\label{sec:nonsep:stochastic}

In the case of a randomized algorithm, we allow the vectors $\g_{k-1}$ and $\bfb_{k-1/2}$ to be defined based on a randomly chosen subset of the indices in $[n]$. This amounts to accessing first-order information $(f_j, \grad f_j)$ for only some $j \in [n]$. The expressions for $\g_{k-1}$ and $\bfb_{k-1/2}$ will depend on historical values of the primal and dual iterations, in the spirit of variance reduction or random extrapolation for convex minimization (see, e.g., \citet{gower2020variance}).
We first describe precisely which comparison points are stored by the algorithm and then specify how they are used to compute $\g_{k-1}$ and $\bfb_{k-1/2}$. 
We store $N$ previous primal \emph{iterates} $\hxb_{k, 1}, \ldots, \hxb_{k, N} \in \X$ and a collection of past dual \emph{coordinate blocks} $\hyb_k = \p{\hyb_{k, 1}, \ldots, \hyb_{k, N}} \in \Y$ associated to each block index.
We define a primal gradient table $\hJb_{k} = (\hgb_{k, 1}, \ldots, \hgb_{k, n})\in \R^{n\times d}$ and dual gradient table $\hfb_{k} = (\hf_{k, 1}, \ldots, \hf_{k, n}) \in \R^n$ constructed via
\begin{align}
    (\hat{f}_{k,i}, \hgb_{k, i}) = (f_i(\hxb_{k, I}), \grad f_i(\hxb_{k, I})) \text{ for all } i \in B_I.\label{eq:first_order_table}
\end{align}
In other words, the tables contain the first-order information of each $f_i$ in block $B_I$ at $\hxb_{k, I}$. Note that we do not necessarily need to store the $O(nd)$-sized gradient table, and need only store the $O(Nd)$-sized table of comparison points, which can be much smaller.
We update these tables randomly at the end of each iteration by independently sampling block indices $R_k$ and $S_k$ (possibly non-uniformly) and setting each block to
\begin{align}
    \hxb_{k, I} = \begin{cases}
        \x_{k} & \text{ if } I = R_k\\
        \hxb_{k-1, I} & \text{ otherwise}
    \end{cases}
    \text{ and } 
    \hyb_{k, I} = \begin{cases}
        \y_{k, I} & \text{ if } I = S_k\\
        \hyb_{k-1, I} & \text{ otherwise}
    \end{cases}.
    \label{eq:table_elements}
\end{align}
As mentioned in \Cref{sec:template},
we define $(\hxb_{k, I}, \hyb_{k, I}) = (\x_0, \y_{0, I})$ for any block $B_I$ and iteration $k < 0$. The probability weights that govern the randomness in $R_k$ and $S_k$ are denoted as $\rb = (r_1, \ldots, r_N)$, and  $\boldsymbol{s} = (s_1, \ldots, s_N)$, respectively.

Next, for computing the primal and dual gradient estimates, we sample two more block indices $P_k$ and $Q_k$ with associated probability mass vectors $\pb = (p_1, \ldots, p_N)$ and $\qb = (q_1, \ldots, q_N)$, respectively.
Letting $\e_j$ denote the $j$-th standard basis vector in $\R^n$, we construct
\begin{align}
    \g_{k-1} &= \hJb_{k-1}^\top \hyb_{k-1} + \frac{a_{k-1}}{p_{P_k} a_k}\sum_{i \in B_{P_k}}\p{y_{k-1, i}\grad f_{i}(\x_{k-1}) - \hy_{k-2, i} \hgb_{k-2, i}} \label{eq:primal_grad_est} \\
    \bfb_{k-1/2} &= \hfb_{k} + \frac{a_{k-1}}{ q_{Q_k} a_k} \sum_{j \in B_{Q_k}} (f_j(\x_{k-1})  - \hf_{k-1, j})\e_j.\label{eq:dual_grad_est}
\end{align}
Even though $\x_k$ is known during the update of $\y_k$, notice that we do not set $\bfb_{k-1/2} = \fb(\x_k)$ in this setting to avoid an $\tilde{O}(nd)$ per-iteration complexity. Also notice that $\hJb_{k-1}^\top \hyb_{k-1}$ can be maintained at an $O(nd/N)$ cost per iteration on average (as opposed to $O(nd)$), because both $\hyb_k$ and $\hJb_k$ only change within a single coordinate block each ($R_k$ and $S_k$, respectively). This is discussed in detail alongside the per-iteration complexities of specific algorithms.

\edit{This section will present two meta-strategies for instantiating \Cref{algo:drago_v2}. While the complexity analyses will hold for any choice of the probability vectors $(\pb, \qb, \rb, \boldsymbol{s})$ (including setting all probabilities uniformly to $1/N$), each algorithm will adapt to the non-uniformity of $(\Gb_1, \ldots, \Gb_N)$ and $(\Lb_1, \ldots, \Lb_N)$ by tuning different hyperparameters. We refer to the first strategy as historical regularization (\Cref{sec:stochastic:historic}), in which $\rb$ and $\boldsymbol{s}$ can harmlessly be set to uniform, whereas the sequences $\br{(\wP_k, \wD_k)}_{k\geq 1}$ and weights of the comparison points $\br{\gamma_I}_{I=1}^N$ are chosen carefully. The second strategy, called non-uniform block replacement, harmlessly sets $\wP_k = \wD_k = 0$ (just as in the previous section), and instead relies on tuning $\rb$ and $\boldsymbol{s}$ to achieve the desired complexity. In both cases, $\pb$ and $\qb$ are chosen to adapt to the choices of the other hyperparameters. We provide the mathematical components that will be common to both analyses upfront. Then, \Cref{sec:stochastic:historic} and \Cref{sec:stochastic:minty} provide the full algorithms (\Cref{algo:stochastic:historic} and \Cref{algo:stochastic:minty}, respectively), and prove the main results.
}

We first collect the probabilistic notation used in this manuscript.
We have introduced four random variables for any iteration $k$: $P_k \sim \pb$, $Q_k \sim \qb$, $R_k \sim \rb$, and $S_k \sim \boldsymbol{s}$. To formally analyze the resulting algorithm, consider a filtered probability space $\mc{P} = (\Omega, \p{\mc{F}_k}_{k \geq 0}, \prob)$, where we use the natural filtration $(\mc{F}_k)_{k \geq 0}$ (with $\mc{F}_0 = \br{\varnothing, \Omega}$ and $\mc{F}_k$ being the $\sigma$-algebra generated by the collection $\{(P_\kappa, Q_\kappa, R_\kappa, S_\kappa)\}_{\kappa=1}^{k}$). We also denote by $\mc{F}_{k-1/2}$ the $\sigma$-algebra generated by $\br{(P_\kappa, Q_\kappa, R_\kappa, S_\kappa)}_{\kappa=1}^{k-1} \cup \br{P_k}$, which captures information up until and including the computation of $\x_k$ (but not $\y_k$). Thus, in the language of probability theory, we will say that $\x_k$ is $\mc{F}_{k-1/2}$-measurable and $\y_k$ is $\mc{F}_{k}$-measurable. The full or marginal expectation on $\mc{P}$ is given by $\Ex$, whereas the conditional expectation given $\mc{F}_k$ is denoted by $\Ex_{k}$. We let $\z_{k, J} = \p{z_{k, j}}_{j \in B_J}$ be the block $B_J$ coordinates of a time-indexed vector $\z_k \in \R^n$. In the arguments below, we will always consider $(\u, \v)$ that is independent of $\{(P_\kappa, Q_\kappa, R_\kappa, S_\kappa)\}_{\kappa \geq 1}$. In \Cref{sec:discussion:sup_exp_vs_exp_sup}, we then precisely describe how this assumption can be relaxed to achieve the same complexity guarantees for a stronger convergence criterion than the one in \Cref{thm:stochastic:historic} and \Cref{thm:stochastic:minty}.

We proceed to the convergence analysis.
The primal-dual sequence $(\x_k, \y_k)_{k \geq 0}$ is now a stochastic process; we do not distinguish random variables and realizations when clear from context. We will use the following fact throughout this section: because $\norm{\cdot}_\Y$ is an $\ell_p$-norm with $p \in [1, 2]$ (see \Cref{sec:preliminaries}), it holds that $\norm{\cdot}_\Y \geq \norm{\cdot}_2$ and $\norm{\cdot}_{\Y^*} \leq \norm{\cdot}_2$. We will use similar techniques as before to upper bound $\sum_{k=1}^t a_k\Ex[\Gap^{\u, \v}(\x_k, \y_k)]$. First, either by using \Cref{lem:lower} and \Cref{lem:upper}, we produce a lower bound for $a_k\Lcal(\u, \y_k)$ and an upper bound for $a_k\Lcal(\x_k, \v)$. 
Before stating these results, we describe the aspects of the analysis that are similar to the algorithm from \Cref{sec:nonsep:full}. Recall the notation box from \Cref{sec:template}.

As in \Cref{sec:nonsep:full}, we need to use the structure of $\g_{k-1}$ (defined in~\eqref{eq:primal_grad_est}) and $\bfb_{k-1/2}$ (defined in~\eqref{eq:dual_grad_est}) and conditions on $(a_k)_{k\geq 1}$ to control (the expected value of) the inner product terms that appear in~\eqref{eq:stoch:error_primal} and~\eqref{eq:stoch:error_dual} below. We describe the analogous argument to the one leading to~\eqref{eq:full:error}, except for the stochastic setting.
Using $\g_{k-1}$ as an example, we take the conditional expectation of $\IP_k$ from~\eqref{eq:inner_prod_terms} given $\mc{F}_{k-1}$ for $k \geq 1$ (recalling that $\u$ is independent of the algorithm randomness) to write
\begin{align}
    \Ex_{k-1}[\IP_k] &= a_k \Ex_{k-1}\ip{\grad \fb(\x_k)^\top \y_k - \hJb_{k-1}^\top \hyb_{k-1}, \u - \x_k}\notag\\
    &\quad- a_{k-1} \ip{\grad \fb(\x_{k-1})^\top \y_{k-1} - \hJb_{k-2}^\top \hyb_{k-2}, \u - \x_{k-1}} - \Ex_{k-1}[\EP_k]\label{eq:stoch:three_term_decomp}
\end{align}
for the error term
\begin{align}
    \EP_k = a_{k-1} \ip{\tfrac{1}{p_{P_k}}\textstyle\sum_{i \in B_{P_k}} (y_{k-1, i}\grad f_{i}(\x_{k-1}) - \hy_{k-2, i} \hgb_{k-2, i}), \x_{k-1} - \x_{k}}. \label{eq:stoch:error_primal}
\end{align}
Note that the telescoping occurs when taking the marginal expectation $\Ex$ over the entire sequence. The term $\ED_k$ is defined analogously by substituting $\bfb_{k-1/2}$ from~\eqref{eq:dual_grad_est} into the expression for $\ID_k$, yielding
\begin{align}
    \ED_k = a_{k-1} \ip{\tfrac{1}{q_{Q_k}}\textstyle\sum_{j \in B_{Q_k}} \p{f_j(\x_{k-1}) - f_j(\hxb_{k-1, Q_k})}\e_j, \y_{k} - \y_{k-1}}. \label{eq:stoch:error_dual}
\end{align}

We summarize the above using an identity card.
\begin{mdframed}[userdefinedwidth=\linewidth, align=center, linewidth=0.3mm]
\centering
\begin{tabular}{l}
     \textbf{Identity Card 2:} Stochastic update method for general objectives\\
     \toprule
     $\sub_1$: $\g_{k-1} = \hJb_{k-1}^\top \hyb_{k-1} + \frac{a_{k-1}}{p_{P_k} a_k}\sum_{i \in B_{P_k}}\p{y_{k-1, i}\grad f_{i}(\x_{k-1}) - \hy_{k-2, i} \hgb_{k-2, i}}$\\
     $\sub_2$: $\bfb_{k-1/2} = \hfb_{k} + \frac{a_{k-1}}{ q_{Q_k} a_k} \sum_{j \in B_{Q_k}} (f_j(\x_{k-1})  - \hf_{k-1, j})\e_j$\\
     $\sub_3$: Update $(\hxb_{k, I}, \hyb_{k, I})$ for all $I$ using~\eqref{eq:table_elements} and $(\hJb_k, \hfb_k)$ using~\eqref{eq:first_order_table}.\\
     \midrule
     \emph{Primal error:} $\EP_k = a_{k-1} \ip{\tfrac{1}{p_{P_k}}\textstyle\sum_{i \in B_{P_k}} (y_{k-1, i}\grad f_{i}(\x_{k-1}) - \hy_{k-2, i} \hgb_{k-2, i}), \x_{k-1} - \x_{k}}$\\
     \emph{Dual error:} $\ED_k = a_{k-1} \ip{\tfrac{1}{q_{Q_k}}\textstyle\sum_{j \in B_{Q_k}} \p{f_j(\x_{k-1}) - f_j(\hxb_{k-1, Q_k})}\e_j, \y_k - \y_{k-1}}$\\
\end{tabular}
\end{mdframed}

Notice that the terms $\EP_k$ and $\ED_k$ that appear in~\eqref{eq:stoch:error_primal} and~\eqref{eq:stoch:error_dual} are measurements of ``table bias'', or how stale the elements in the tables are compared to the current iterates $\x_k$ (for $\EP_k$) and $\y_k$ (for $\ED_k$). The algorithms below provide two different strategies for achieving convergence while controlling these errors. Because terms of the form $\EP_k$ will appear in multiple analyses, we collect a repeatedly used bound below.
\begin{lemma}\label[lemma]{lem:separation}
    Consider $\mc{F}_{k-1}$-measurable random vectors $\x$ and $\hxb_1, \ldots, \hxb_N$ realized in $\X$, and similarly, let $\mc{F}_{k-1}$-measurable $\y$ and $\hyb$ be realized in $\Y$. For any collection of positive constants $(b_I)_{I=1}^N$ and $(c_I)_{I=1}^N$, we have that
    \begin{align*}
        &\Ex\norm{\tfrac{1}{p_{P_k}}\textstyle\sum_{i \in B_{P_k}} (y_{i}\grad f_{i}(\x) - \hy_{i} \grad f_{i}(\hxb_{P_k}))}_{\X^*}^2\\
        &\leq 2 \p{\max_J \frac{\Lb_J^2}{p_J b_J}}  \sum_{I=1}^N b_I \Ex\norm{\x - \hxb_I}_{\X}^2 +  2\p{\max_J \frac{\Gb_J^2}{p_J c_J}}  \sum_{I=1}^N c_I \Ex\norm{\y_{I} - \hyb_{I}}_2^2.
    \end{align*}
\end{lemma}
\begin{proof}
    First, take the conditional expectation given $\mc{F}_{k-1}$ so that
    \begin{align*}
        &\Ex_{k-1}\norm{\tfrac{1}{p_{P_k}}\textstyle\sum_{i \in B_{P_k}} (y_{i}\grad f_{i}(\x) - \hy_{i} \grad f_{i}(\hxb_{P_k}))}_{\X^*}^2\\
        &= \sum_{I=1}^N \frac{1}{p_I}\norm{\textstyle\sum_{i \in B_{I}} (y_{i}\grad f_{i}(\x) - \hy_{i} \grad f_{i}(\hxb_I))}_{\X^*}^2\\
    \end{align*}
    Applying Young's inequality, the quantity above is upper bounded via
    \begin{align*}
        &\sum_{I=1}^N \frac{2}{p_I} \sbr{\norm{\textstyle\sum_{i \in B_{I}} y_{i}(\grad f_{i}(\x) - \grad f_{i}(\hxb_I))}_{\X^*}^2 + \p{\textstyle\sum_{i \in B_{I}} \abs{y_{i} - \hy_{i}}\norm{\grad f_{i}(\hxb_I)}_{\X^*}}^2}\\
        &\overset{(\circ)}{\leq} \sum_{I=1}^N \frac{2}{p_I} \sbr{\Lb_I^2\norm{\x - \hxb_I}_{\X}^2 + \Gb_I^2 \norm{\y_{I} - \hyb_{I}}_2^2}\\
        &\leq 2 \p{\max_J \frac{\Lb_J^2}{p_J b_J}}  \sum_{I=1}^N b_I \norm{\x - \hxb_I}_{\X}^2 +  2\p{\max_J \frac{\Gb_J^2}{p_J c_J}}  \sum_{I=1}^N c_I \norm{\y_{I} - \hyb_{I}}_2^2.
    \end{align*}
    where we used \Cref{asm:smoothness} in $(\circ)$. Take the marginal expectation to complete the proof.
\end{proof}

\subsection{Strategy 1: Non-Uniform Historical Regularization}\label{sec:stochastic:historic}
\begin{algorithm}[t]
\setstretch{1.2}
   \caption{Historical Regularization Version of \Cref{algo:drago_v2} (Non-Separable)}
   \label{algo:stochastic:historic}
    \begin{algorithmic}[1]
   \State {\bfseries Input:} Initial point $(\x_0, \y_0)$, averaging sequence $(a_k)_{k=0}^t$, non-negative weights $(\gamma_I)_{I=1}^N$ that sum to one, balancing sequences $(\wP_k)_{k=1}^t$ and $(\wD_k)_{k=1}^t$, sampling vectors $(\pb, \qb)$.
   \State Initialize the comparison points $\hxb_{0, I} = \x_0$ for all $I \in [N]$ and $\hyb_0 = \y_0$. Initialize the table $\hfb_0 = \fb(\x_0)$ and $\hJb_0 = \grad \fb(\x_0)$, and compute the product $\hJb_{0}^\top \hyb_0$.
   \State Letting $\varphi$ be the generator of $\breg{\X}{\cdot}{\cdot}$, compute $\br{\grad \varphi(\hxb_{0, I})}_{I=1}^N$ and $\sum_I \gamma_I \grad \varphi(\hxb_{0, I})$.
   \For {$k=1$ {\bfseries to} $t$}
       \State $\sub_1$: Compute $\g_{k-1}$ using $\pb$ and~\eqref{eq:primal_grad_est}.
       \State Perform the primal update
       \begin{align}
       \x_k &= \argmin_{\x \in \X} \Big\{a_k \ip{\g_{k-1}, \x} + a_k\phi(\x) + \tfrac{A_{k-1}\mu + \mu_0}{2} \varphi(\x)\notag \\
       &\quad - \tfrac{A_{k-1}\mu + \mu_0}{2}\ip{(1-\wP_k) \grad \varphi(\x_{k-1}) + \wP_k\textstyle\sum_{I=1}^N \gamma_I \grad \varphi(\hxb_{k-1, I}), \x}\Big\}.\notag
       \end{align}
       \State $\sub_2$: Draw $R_k$ uniformly on $[N]$, and update $(\hxb_{k, I})_{I=1}^N$ via~\eqref{eq:table_elements}. Compute $\bfb_{k-1/2}$ using $\qb$ and~\eqref{eq:dual_grad_est}.
       \State Perform the dual update
       \begin{align}
           \y_k &= \argmax_{\y \in \Y} \Big\{ a_k \ip{\y, \bfb_{k-1/2}} - a_k\psi(\y) \notag\\ 
           &\quad - (A_{k-1}\nu + \nu_{0}) \Big(\tfrac{1-\wD_k}{2} \breg{\Y}{\y}{\y_{k-1}} + \tfrac{\wD_k}{2} \breg{\Y}{\y}{\hyb_{k-1}}\Big)\Big\}. \notag
       \end{align}
       \State $\sub_3$: Draw $S_k$ uniformly and update $\hyb_k$ via~\eqref{eq:table_elements}. Update the product $\hJb_{k}^\top \hyb_k$ via~\eqref{eq:g_replacement}, and update $\br{\grad \varphi(\hxb_{k, I})}_{I=1}^N$ and $\sum_I \gamma_I \grad \varphi(\hxb_{k, I})$ via~\eqref{eq:breg_replacement}.
   \EndFor
   \Return $A_t^{-1} \sum_{k=1}^t a_k(\x_k, \y_k)$.
\end{algorithmic}

\end{algorithm}

\edit{
The full form of the algorithm is given in \Cref{algo:stochastic:historic}, for which we may have $\wP_k > 0$ or $\wD_k > 0$ in general. For this reason, we expand the definition of the Bregman divergence $\breg{\X}{\cdot}{\cdot}$ in terms of the generator $\varphi$ to better display the computational aspects of the historical regularization term. In particular, if we show that the optimization problem~\eqref{eq:prim_update} can be solved at $\tilde{O}(d)$ cost, then the total arithmetic complexity is given by $\tilde{O}(n(d+N)/N)$, where we assume uniform block sizes for arithmetic complexity discussions. The relevant terms are the matrix-vector product $\hJb_{k}^\top \hyb_k$ and the sum of weighted Bregman divergences. 
}

For the former, let $\hJb_{k, I} \in \R^{n/N \times d}$ denote the matrix containing the primal gradients for the elements in block $I$. Noting that $R_k$ is the block of the primal table that is updated at each iteration and $S_k$ plays the same role for the dual table, it holds that
\begin{align}
    \hJb_{k}^\top \hyb_k &= \hJb_{k-1}^\top \red{\hyb_k} + (\hJb_{k, R_k} - \hJb_{k-1, R_k})\hyb_{k, R_k} \notag\\
    &= \hJb_{k-1}^\top \p{\red{\hyb_{k-1} + \textstyle\sum_{j \in B_{S_k}}(\hy_{k, j} - \hy_{k-1, j})\e_j }} + (\hJb_{k, R_k} - \hJb_{k-1, R_k})^\top\hyb_{k, R_k} \notag\\
    &= \hJb_{k-1}^\top \hyb_{k-1} + \hJb_{k-1, S_k}^\top(\red{\hyb_{k, S_k} - \hyb_{k-1, S_k}}) + (\hJb_{k, R_k} - \hJb_{k-1, R_k})^\top\blue{\hyb_{k, R_k}} \label{eq:g_replacement}.
\end{align}
Assuming that $\hJb_{k-1}^\top \hyb_{k-1} \in \R^{d}$ is already stored, everything above can be computed with $\hJb_{k, R_k}$, $\hyb_{k, S_k}$, and past information at cost $O(nd/N)$. Furthermore, we need not retain the entire table $\hJb_{k} \in \R^{n \times d}$, as $\hJb_{k-1, R_k}$ and $\hJb_{k-1, S_k}$ above can be recomputed from $\hxb_{k-1, R_k}$ and $\hxb_{k-1, R_k}$. The entire memory footprint is $O(Nd + n)$, which could be much smaller than $O(nd)$ (for instance, when $N = d$).

For the latter, letting $\varphi$ be the generator of $\breg{\X}{\cdot}{\cdot}$, we write
\begin{align*}
    \sum_{I=1}^N \gamma_I \breg{\X}{\x}{\hxb_{k, I}} = \varphi(\x) + \ip{\sum_{I=1}^N \gamma_I\grad \varphi(\hxb_{k, I}), \x} + \operatorname{const}(\x),
\end{align*}
where the term $\operatorname{const}(\x)$ does not vary with respect to $\x$. It then holds that
\begin{align}
    \sum_{I=1}^N \gamma_I\grad \varphi(\hxb_{k, I}) = \sum_{I=1}^N \gamma_I \grad \varphi(\hxb_{k-1, I}) + \gamma_{R_k}\p{\grad \varphi(\hxb_{k, R_k}) - \grad \varphi(\hxb_{k-1, R_k})},\label{eq:breg_replacement}
\end{align}
so we need only compute $\grad \varphi(\hxb_{k, R_k})$ at each iteration. Retaining the $(\grad \varphi(\hxb_{k, I}))_{I=1}^N$ comes at an $O(Nd)$ storage cost, which is the same as the table itself. The total per-iteration complexity is then $\tilde{O}(n(d+N)/N)$. \edit{Having understood the algorithm and its per-iteration complexity, it remains to analyze the number of iterations required to reduce the expected primal-dual gap. Recalling \Cref{claim:gap}, the expected $\IP_k$ and $\ID_k$ terms have been replaced by the expectation of $\EP_k$ and $\ED_k$ from Identity Card 2. We also need to control~\eqref{eq:stoch:gap3_primal} and~\eqref{eq:stoch:gap3_dual}, which vanish only when $\wP_k = \wD_k = 0$.
}

As such, the goal will be to select the balancing sequences $(\wP_k)_{k \geq 1}$ and $(\wD_k)_{k \geq 1}$ to achieve \edit{these cancellations and} the desired complexity guarantee. 
The rate conditions will be stated in terms of three quantities that largely depend on the sampling schemes $\pb$ and $\qb$ (which are used to define $\g_{k-1}$ and $\bfb_{k-1/2}$) and the primal regularization weights $\bgam$. Those are
\begin{align}
    \Gb_{\pb} := \sqrt{\textstyle\max_I \frac{\Gb_I^2}{p_I}}, \Lb_{\pb, \bgam} := \sqrt{\textstyle\max_I\frac{\Lb_I^2}{p_I \gamma_I}}, \text{ and } \Gb_{\qb, \bgam} := \sqrt{\textstyle\max_I \frac{\Gb_I^2}{q_I \gamma_I}}.\label{eq:stoch:historic:constants}
\end{align}
Recall that the vectors $\rb = (r_1, \ldots, r_N)$ and $\sbold = (s_1, \ldots, s_N)$ will be set to the uniform vectors $\rb = \ones/N$ and $\sbold=\ones/N$ as they will not affect the convergence rates in this analysis. \edit{The organization of the upcoming results is similar to \Cref{sec:template:first}. Recall the initial distance $D_0$ from~\eqref{eq:initial_distance}. 
}
\begin{thm}\label{thm:stochastic:historic}
    Under \Cref{asm:smoothness} and \Cref{asm:str_cvx}, consider any $\u \in \X$, $\v \in \Y$ and precision $\epsilon > 0$. 
    There exists a choice of the sequence $(a_k)_{k=1}^t$, 
    and the parameters $\wP_k$ and $\wD_k$ such that \Cref{algo:stochastic:historic} produces an output point $(\tx_t, \ty_t) \in \X \times \Y$ satisfying $\E{}{\gap^{\u, \v}(\tx_t, \ty_t)} \leq \epsilon$ for $t$ according to the following iteration complexities. \edit{In the following, $a_k$, $\wP_k$, $\wD_k$, and the iteration complexity are stated in big-$O$ terms, with 
    \begin{align*}
        \alpha \sim \min\br{{\sqrt{\mu\nu}}/({\sqrt{N}(\Gb_{\pb} \vee  \Gb_{\qb, \bgam}))}, {\mu}/({\sqrt{N}\Lb_{\pb, \bgam})}, 1/N}.
    \end{align*}
    }
    \begin{center}
    \begin{adjustbox}{max width=\linewidth}
    \begin{tabular}{ccccc}
    \toprule
        {\bf Case} & \edit{{\bf $a_k$ }} & \edit{{\bf $\wP_k$}} & \edit{{\bf $\wD_k$}} & {\bf Iteration Complexity}\\
        \midrule
        $\mu > 0$, $\nu > 0$
        & \edit{$(1+\alpha)^k$} & \edit{$1/N$} & \edit{$1/N$}
        &
        $\p{N + \frac{\sqrt{N}\Lb_{\pb, \bgam}}{\mu} + \frac{\sqrt{N}(\Gb_{\pb} \vee  \Gb_{\qb, \bgam})}{\sqrt{\mu\nu}}}\ln\p{\frac{1}{\epsilon}}$
        \\
        $\mu > 0$, $\nu = 0$
        & \edit{$k$} & \edit{$1/N^2$} & \edit{$1/N$}
        &
        $\p{N + \frac{N\Lb_{\pb, \bgam}}{\mu}}\ln\p{\frac{1}{\epsilon}} + (\sqrt{N}\Gb_{\pb} \vee N\Gb_{\qb, \bgam})\sqrt{\frac{\sqrt{\mu_0/\nu_0}D_0 + (a_1\mu/\nu_0)\breg{\X}{\u}{\x_0}}{\mu \epsilon}}$
        \\
        $\mu = 0$, $\nu > 0$
        & \edit{$1$} & \edit{$1/N$} & \edit{$1/N^2$}
        &
        $N\ln\p{\frac{1}{\epsilon}} + \frac{\sqrt{N}\Lb_{\pb, \bgam}\sqrt{\nu_0/\mu_0}D_0}{\epsilon} + (N\Gb_{\pb} \vee \sqrt{N}\Gb_{\qb, \bgam})\sqrt{\frac{\sqrt{\nu_0/\mu_0}D_0+ (a_1\nu/\mu_0)\breg{\Y}{\v}{\y_0}}{\nu \epsilon}}$
        \\
        $\mu = 0$, $\nu = 0$
        & \edit{$1$} & \edit{$1/N$} & \edit{$1/N$}
        &
        $N\ln\p{\frac{1}{\epsilon}} +\frac{\p{\sqrt{N}\Lb_{\pb, \bgam}\sqrt{\nu_0/\mu_0}  + \sqrt{N}(\Gb_{\pb} \vee  \Gb_{\qb, \bgam})}D_0}{\epsilon}$\\
    \bottomrule
    \end{tabular}
    \end{adjustbox}
    \end{center}
\end{thm}
\edit{The problem constants appearing in the complexities can be written in terms of the underlying Lipschitz and smoothness constants from \Cref{asm:smoothness} by selecting the sampling schemes $\pb$ and $\qb$.}
Recall from \Cref{sec:preliminaries} that $\lambb = (\lamb_1, \ldots, \lamb_N)$ for $\lamb_I := \sqrt{\Gb_I^2 + \Lb_I^2}$ along with the constants from~\eqref{eq:stoch:historic:constants}. The non-uniform sampling complexity given below follows by letting $p_I \propto \lamb_I$, $\gamma_I \propto \lamb_I$, and $q_I \propto \Gb_I$. The constants appearing in \Cref{thm:stochastic:historic} are tabulated below.
\begin{center}
\begin{adjustbox}{max width=\linewidth}
\begin{tabular}{c|cc}
\toprule
    {\bf Constant} & {\bf Uniform Sampling} & {\bf Non-Uniform Sampling}\\
    \midrule
    $\Gb_{\pb} \vee \Gb_{\qb, \rb} $ & $N\norm{\Gbb}_\infty$ & $\norm{\lambb}_{1}^{1/2} \norm{\Gbb}_{1}^{1/2}$\\
    $\Lb_{\pb, \bgam}$ & $N\norm{\Lbb}_\infty$ & $\norm{\lambb}_1$\\
    \bottomrule
\end{tabular}
\end{adjustbox}
\end{center}

\edit{\Cref{prop:stochastic:historic} establishes conditions on $\br{(a_k, \wP_k, \wD_k)}_{k=1}^t$ such that convergence can be argued via \Cref{claim:gap}. } For readability, we provide a version of the proof for when~\eqref{eq:semilinear} is dual separable (but not necessarily the feasible set). We discuss how this assumption can be avoided by a minor modification of the algorithm after the proof.
\begin{prop}\label[proposition]{prop:stochastic:historic}
    Let $(\x_0, \y_0) \in \ri(\dom(\phi)) \times \ri(\dom(\psi))$ and $\br{(\x_k, \y_k)}_{k\geq 1}$ be generated by the update from \Cref{lem:lower} with non-decreasing sequence $\wP_k$ and \Cref{lem:upper} with non-decreasing sequence $\wD_k$, with $\g_{k-1}$ and $\bfb_{k-1/2}$ given by~\eqref{eq:primal_grad_est} and~\eqref{eq:dual_grad_est}, respectively. Define $a_1 = \min\br{\frac{\sqrt{\wD_0\mu_0 \nu_0}}{4\Gb_{\pb}}, \frac{\sqrt{\wP_0}\mu_0}{4\sqrt{2}\Lb_{\pb, \bgam}}, \frac{\sqrt{\wP_0\mu_0 \nu_0}}{4\Gb_{\qb, \bgam}}}$ and select $(a_k)_{k \geq 2}$ such that the conditions
    \begin{align}
        a_{k} &\leq \tfrac{1}{4\sqrt{2}}\br{\tfrac{\sqrt{\wD_{k}(A_{k}\mu + \mu_0)(A_{k-1}\nu + \nu_0)}}{\Gb_{\pb}}, \tfrac{\sqrt{\wP_{k}(A_{k}\mu + \mu_0)(A_{k-1}\mu + \mu_0)}}{\Lb_{\pb, \bgam}}, \tfrac{\sqrt{\wP_{k}(A_{k}\nu + \nu_0)(A_{k-1}\mu + \mu_0)}}{\Gb_{\qb, \bgam}}} \label{eq:stoch:cond}
    \end{align}
    are satisfied. In addition, impose that for any $\ell = 1, \ldots, t-1$, it holds that
    \begin{align}
        \frac{\mu}{N}\sum_{k'=0}^{\infty}  \wP_{\ell + k' + 1} A_{\ell + k'}  (1-1/N)^{k'} &\leq  (\wP_\ell A_\ell + a_\ell)\mu, \label{eq:stoch:primal:decay}\\
        \frac{\nu}{N} \sum_{k'=0}^{\infty}  \wD_{\ell + k' + 1} A_{\ell + k'} (1-1/N)^{k'} &\leq  (\wD_\ell A_\ell + a_\ell)\nu. \label{eq:stoch:dual:decay}
    \end{align}
    We have that for any $\u \in \X$ and $\v \in \Y$, 
    \begin{align*}
        \sum_{k=1}^t a_k \Ex[\gap^{\u, \v}(\x_k, \y_k)] + \tfrac{1}{4}\Ex[\TP_t] + \tfrac{1}{4}\Ex[\TD_t]
        &\leq C_{\X}\breg{\X}{\u}{\x_{0}} + C_{\Y}\breg{\Y}{\v}{\y_{0}}.
    \end{align*}
    for constants $C_{\X} = \tfrac{1}{2}\p{\mu_0 + \textstyle\sum_{\ell=0}^{\infty} \wP_{\ell+1} (A_{\ell}\mu + \mu_0)  (1-1/N)^{\ell}}$ and $C_{\Y} = \tfrac{1}{2}\p{\nu_0 + \textstyle\sum_{\ell=0}^{\infty} \wD_{\ell+1} (A_{\ell}\nu + \nu_0)  (1-1/N)^{\ell}}$.
\end{prop}
\begin{proof}
    Our starting point is \Cref{claim:gap}, after which we must show that the terms in~\eqref{eq:stoch:gap2_cancellation},~\eqref{eq:stoch:gap3_primal}, and~\eqref{eq:stoch:gap3_dual} are bounded by a quantity that does not grow with $t$. Recall that for~\eqref{eq:stoch:gap2_cancellation},
    we used the three-term decomposition~\eqref{eq:stoch:three_term_decomp}, which generated two telescoping terms and one error term. The first part of the proof uses this argument and bounds the error term.

    \paragraph{1. Controlling~\eqref{eq:stoch:gap2_cancellation}:} We follow the arguments at the beginning of this section to produce $\EP_k$ in~\eqref{eq:stoch:error_primal} and $\ED_k$ in~\eqref{eq:stoch:error_dual}. We first upper bound the error terms $\Ex[\EP_k]$ for $k = 2, \ldots, t-1$ (noting that $\EP_1 = \ED_1 = 0$) and the last term of the telescoping inner products in~\eqref{eq:stoch:three_term_decomp}.
    By Young's inequality with parameter $(1-\wP_k)(A_{k-1}\mu + \mu_0)/4$, we have
    \begin{align}
        \Ex[\EP_k] &= a_{k-1} \Ex\ip{\tfrac{1}{p_{P_k}}\textstyle\sum_{i \in B_{P_k}} (y_{k-1, i}\grad f_{i}(\x_{k-1}) - \hy_{k-2, i} \hgb_{k-2, i}),  \x_{k-1} - \x_{k}} \label{eq:stoch:primal:middle} \\
        &\leq \frac{1-\wP_k}{4} \Ex[\CP_k]+ \frac{2a_{k-1}^2}{(1-\wP_k)(A_{k-1}\mu + \mu_0)} \Ex\norm{\tfrac{1}{p_{P_k}}\textstyle\sum_{i \in B_{P_k}} (y_{k-1, i}\grad f_{i}(\x_{k-1}) - \hy_{k-2, i} \hgb_{k-2, i})}_{\X^*}^2 \notag
    \end{align}
    and for the last term of the decomposition~\eqref{eq:stoch:three_term_decomp}, \edit{consider a random variable $P_{t+1} \sim \pb$ that is independent of all components of the algorithm. Then, it holds that}
    \begin{align}
        &-\Ex\sbr{a_t \ip{\grad \fb(\x_t)^\top \y_t - \hJb_{t-1}^\top \hyb_{t-1}, \u - \x_t}} = \notag\\
        &-\Ex\sbr{a_t \ip{\tfrac{1}{p_{P_{t+1}}}\textstyle\sum_{i \in B_{P_{t+1}}} (y_{t, i}\grad f_{i}(\x_{t}) - \hy_{t-1, i} \hgb_{t-1, i}), \u - \x_t}} \leq \notag\\
        &\frac{1-\wP_t}{4}\Ex[\TP_{t}] + \frac{2a_{t}^2}{(1-\wP_t)(A_{t}\mu + \mu_0)} \Ex\norm{\tfrac{1}{p_{P_{t+1}}}\textstyle\sum_{i \in B_{P_{t+1}}} (y_{t, i}\grad f_{i}(\x_{t}) - \hy_{t-1, i} \hgb_{t-1, i})}_{\X^*}^2  \label{eq:stoch:primal:last}
    \end{align} 
    \edit{wherein the expectation is taken over $P_{t+1}$ as well.}
    To handle the second term in~\eqref{eq:stoch:primal:middle} and~\eqref{eq:stoch:primal:last} for $k \in \br{2, \ldots, t}$, 
    apply \Cref{lem:separation} with $b_I = \gamma_I$ and $c_I = 1$ to achieve
    \begin{align*}
        &\Ex\norm{\tfrac{1}{p_{P_k}}\textstyle\sum_{i \in B_{P_k}} (y_{k-1, i}\grad f_{i}(\x_{k-1}) - \hy_{k-2, i} \hgb_{k-2, i})}_{\X^*}^2 \\
        &\quad \leq 2\p{\max_I \frac{\Lb_J^2}{p_J \gamma_J}} \cdot \sum_{I=1}^N \gamma_I \Ex\norm{\x_{k-1} - \hxb_{k-2, I}}_\X^2 + 2 \p{\max_J \frac{\Gb_J^2}{p_J}} \Ex\norm{\y_{k-1} - \hyb_{k-2}}_2^2\\
        &\quad \leq 4\underbrace{\p{\max_J \frac{\Lb_J^2}{p_J \gamma_J}}}_{\Lb_{\pb, \bgam}^2} \cdot \sum_{I=1}^N \gamma_I \Ex[\breg{\X}{\x_{k-1}}{\hxb_{k-2, I}}] + 4 \underbrace{\p{\max_J \frac{\Gb_J^2}{p_J}}}_{\Gb_{\pb}^2} \Ex[\breg{\Y}{\y_{k-1}}{\hyb_{k-2}}],
    \end{align*}
    where in the last line we applied $\norm{\cdot}_2 \leq \norm{\cdot}_\Y$ and the strong convexity of Bregman divergences.
    Recall that $\hCP_{k, I} := (A_{k-1}\mu + \mu_0)\breg{\X}{\x_k}{\hxb_{k-1, I}}$ and $\hCD_k := (A_{k-1}\nu + \nu_0)\breg{\Y}{\y_k}{\hyb_{k-1}}$. Combining the steps above, using that $1/(1-\wP_k) \leq 2$, and applying the condition~\eqref{eq:stoch:cond},  we have the upper bound
    \begin{align*}
        \Ex[\EP_k] &\leq \frac{1-\wP_k}{4} \Ex[\CP_k] + \frac{8a_{k-1}^2 \Gb_{\pb}^2  \Ex[\hCD_{k-1}]}{(1-\wP_k)(A_{k-1}\mu + \mu_0)(A_{k-2}\nu + \nu_0)}\\
        &\quad + \frac{8a_{k-1}^2 \Lb_{\pb, \bgam}^2 \E{}{\sum_{I=1}^N \gamma_I \hCP_{k-1, I}}}{(1-\wP_k)(A_{k-1}\mu + \mu_0)(A_{k-2}\mu + \mu_0)}\\
        &\leq  \frac{1-\wP_k}{4} \Ex[\CP_k] + \frac{\wD_{k-1}}{2}\Ex[\hCD_{k-1}]  + \frac{\wP_{k-1}}{4}\E{}{\textstyle\sum_{I=1}^N \gamma_I \hCP_{k-1, I}}  
    \end{align*}
    with a similar bound holding for~\eqref{eq:stoch:primal:last}. These terms will cancel with the corresponding non-positive terms in~\eqref{eq:stoch:gap2_cancellation}.
    
    The upper bounds for $\Ex[\ED_k]$ and $-\Ex[a_k \langle\grad \fb(\x_k) - \hfb_{k}, \v - \y_k\rangle]$ follow by very similar arguments as above.
    Applying Young's inequality with parameter $(1-\wD_k)(A_{k-1}\nu + \nu_0)/4$ we upper bound $\Ex_{k-1/2}[\ED_k]$ via
    \begin{align}
    &a_{k-1} \Ex_{k-1/2}\ip{\tfrac{1}{q_{Q_k}}\textstyle\sum_{j \in B_{Q_k}} (f_{j}(\x_{k-1}) - \hf_{k-1, j})\e_j ,  \y_k - \y_{k-1}} \notag\\
        &\leq \frac{1-\wD_k}{4}\Ex_{k-1/2}[\CD_k] + \frac{2a_{k-1}^2(1-\wD_k)^{-1}}{(A_{k-1} \nu + \nu_0)} \sum_{J=1}^N \frac{1}{q_J}\norm{\textstyle\sum_{j \in B_{J}} \p{f_j(\x_{k-1}) - \hf_{k-1, j}}\e_j}_{\Y^*}^2 \label{eq:ED_k_bound}
    \end{align}
    Recall that the index $R_{k-1}$ determines which block of the primal table is updated. Thus, it holds in conditional expectation that
    \begin{align*}
        \E{k-2}{\norm{\x_{k-1} - \hxb_{k-1, J}}_{\X}^2} &= \E{k-2}{\norm{\x_{k-1} - \hxb_{k-1, J}}_{\X}^2 \mathbbm{1}_{J=R_{k-1}}} + \E{k-2}{\norm{\x_{k-1} - \hxb_{k-1, J}}_{\X}^2 \mathbbm{1}_{J\neq R_{k-1}}}\\
        &= 0 + \E{k-2}{\norm{\x_{k-1} - \hxb_{k-2, J}}_{\X}^2 \mathbbm{1}_{J\neq R_{k-1}}}\\
        &\leq \E{k-2}{\norm{\x_{k-1} - \hxb_{k-2, J}}_{\X}^2}.
    \end{align*}
    Taking the marginal expectation, the second term of~\eqref{eq:ED_k_bound} 
    can be upper bounded as
    \begin{align}
        \sum_{J=1}^N \frac{1}{q_J}\Ex\norm{\textstyle\sum_{j \in B_{J}} \p{f_j(\x_{k-1}) - \hf_{k-1, j}}\e_j}_{\Y^*}^2 &\leq \sum_{J=1}^N \frac{\Gb_J^2}{q_J} \Ex\norm{\x_{k-1} - \hxb_{k-1, J}}_{\X}^2\notag\\
        &\leq \sum_{J=1}^N \frac{\Gb_J^2}{q_J} \Ex\norm{\x_{k-1} - \hxb_{k-2, J}}_{\X}^2\notag\\
        &\leq 2\underbrace{\textstyle\max_I \frac{\Gb_I^2}{q_I \gamma_I}}_{\Gb_{\qb, \bgam}^2} \sum_{J=1}^N  \gamma_J \breg{\X}{\x_{k-1}}{\hxb_{k-2, J}}_{\X}^2.
    \end{align}
    
    Invoking condition~\eqref{eq:stoch:cond} once again and taking the marginal expectation, we have
    \begin{align}
        \Ex[\ED_k] &\leq \frac{1-\wD_k}{4}\Ex[\CD_k] + \frac{4a_{k-1}^2 \Gb_{\qb, \bgam}^2 }{(1-\wD_k)(A_{k-1} \nu + \nu_0)(A_{k-2} \mu + \mu_0)}\E{}{\textstyle\sum_{I=1}^N \gamma_I \hCP_{k-1, I}}\\
        &\leq \frac{1-\wD_k}{4}\Ex[\CD_k] + \frac{\wP_{k-1}}{4}\E{}{\textstyle\sum_{I=1}^N \gamma_I \hCP_{k-1, I}}.
    \end{align}
    Similarly, $-\Ex[a_k \langle\grad \fb(\x_k) - \hfb_{k}, \v - \y_k\rangle] \leq \frac{1-\wD_t}{4}\Ex[\TD_t] + \frac{\wP_{t-1}}{4}\E{}{\textstyle\sum_{I=1}^N \gamma_I \hCP_{t, I}}$.
    All these terms will cancel with corresponding terms in~\eqref{eq:stoch:gap2_cancellation}. 

    \paragraph{2. Controlling~\eqref{eq:stoch:gap3_primal}:}
    For this step, we will express the $\hTP_k$ terms as a function of the $\TP_k$ terms by analyzing the random sampling that governs the block updates. 
    For any $k \geq 1$, write
    \begin{align}
        \sum_{I=1}^N \gamma_I \hTP_{k, I} &= \sum_{I=1}^N \gamma_I \Ex[\breg{\X}{\u}{\hxb_{k, I}}] \notag\\
        &= \sum_{I=1}^N \gamma_I \p{(1/N) \Ex[\breg{\X}{\u}{\x_{k}}] + (1-1/N)\Ex[\breg{\X}{\u}{\hxb_{k-1, I}}]} \notag\\
        &= (1/N)\sum_{k'=0}^{k-1}(1-1/N)^{k'}\Ex[\breg{\X}{\u}{\x_{k-k'}}] + (1-1/N)^k\sum_{I=1}^N \gamma_I\Ex[\breg{\X}{\u}{\hxb_{0, I}}]\notag\\
        &= \red{(1/N)\sum_{k'=0}^{k}(1-1/N)^{k'}\Ex[\breg{\X}{\u}{\x_{k-k'}}]} + \blue{(1-1/N)^{k+1} \Ex[\breg{\X}{\u}{\x_0}]}\label{eq:stoch:primal:table_expansion}
    \end{align}
    Using~\eqref{eq:stoch:primal:table_expansion}, we first expand the non-negative portion of the expression~\eqref{eq:stoch:gap3_primal} via
    \begin{align*}
        \sum_{k=1}^t \wP_k\sum_{I=1}^N \gamma_I \hTP_{k-1, I} &= \sum_{k=0}^{t-1} \wP_{k+1}\sum_{I=1}^N \gamma_I \hTP_{k, I}\\
        &= \sum_{k=0}^{t-1} \wP_{k+1} (A_{k}\mu + \mu_0) \red{(1/N)\sum_{k'=0}^{k}(1-1/N)^{k'}\Ex[\breg{\X}{\u}{\x_{k-k'}}]}\\
        &\quad + \sum_{k=0}^{t-1} \wP_{k+1} (A_{k}\mu + \mu_0)  \blue{(1-1/N)^{k+1} \breg{\X}{\u}{\x_0}}.
    \end{align*}
    We expand the first of the two terms above:
    \begin{align*}
        &\sum_{k=0}^{t-1} \wP_{k+1} (A_{k}\mu + \mu_0) \red{(1/N)\sum_{k'=0}^{k}(1-1/N)^{k'}\Ex[\breg{\X}{\u}{\x_{k-k'}}]}\\
        &= \sum_{k'=0}^{t-1} \sum_{k=k'}^{t-1} \wP_{k+1} (A_{k}\mu + \mu_0) (1/N)(1-1/N))^{k'}\Ex[\breg{\X}{\u}{\x_{k-k'}}] & \text{exchange sums}\\
        &= \sum_{k'=0}^{t-1} \sum_{\ell=0}^{t-1-k'} \wP_{\ell + k' + 1} (A_{\ell + k'}\mu + \mu_0) \Ex[\breg{\X}{\u}{\x_{\ell}}] (1/N) (1-1/N)^{k'} & \text{$\ell := k-k'$}\\
        &\leq \sum_{\ell=0}^{t-1} \sum_{k'=0}^{\infty}  \wP_{\ell + k' + 1} (A_{\ell + k'}\mu + \mu_0) \Ex[\breg{\X}{\u}{\x_{\ell}}] (1/N) (1-1/N)^{k'}.
    \end{align*}
    For the second, we simply factor out the initial distance term via
    \begin{align*}
        &\sum_{k=0}^{t-1} \wP_{k+1} (A_{k}\mu + \mu_0)  \blue{(1-1/N)^{k+1} \breg{\X}{\u}{\x_0}}\\
        &\leq \p{\sum_{\ell=0}^{\infty} \wP_{\ell+1} (A_{\ell}\mu + \mu_0)  (1-1/N)^{\ell+1}}\breg{\X}{\u}{\x_0}.
    \end{align*}
    Using the upper bounds above can decompose~\eqref{eq:stoch:gap3_primal} as
    \begin{align*}
        &\frac{1}{2}\sum_{k=1}^t \Ex\sbr{\wP_k\p{\textstyle\sum_{I=1}^N \gamma_I \hTP_{k-1, I} - \TP_{k}}- a_k\mu\breg{\X}{\u}{\x_k}}\\
        &\leq \frac{1}{2}\sum_{\ell=0}^{t-1} \underbrace{\p{\sum_{k'=0}^{\infty}  \wP_{\ell + k' + 1} (A_{\ell + k'}\mu + \mu_0)  (1/N) (1-1/N)^{k'}}}_{\leq \wP_\ell (A_\ell\mu + \mu_0) + a_\ell\mu \text{ for } \ell \geq 1}\Ex[\breg{\X}{\u}{\x_{\ell}}] \\
        &\quad + \frac{1}{2}\p{\sum_{k'=0}^{\infty} \wP_{k'+1} (A_{\ell}\mu + \mu_0)  (1-1/N)^{k'+1}}\breg{\X}{\u}{\x_0}\\
        &\quad - \frac{1}{2}\sum_{\ell=0}^{t-1}\p{\wP_{\ell+1}(A_{\ell+1}\mu + \mu_0) + a_{\ell+1}\mu}\Ex[\breg{\X}{\u}{\x_{\ell+1}}]
    \end{align*}
     where the inequality under the braces follows from the theorem assumptions. By telescoping the resulting terms,~\eqref{eq:stoch:gap3_primal} is upper bounded by
    \begin{align*}
        \frac{1}{2} \p{\sum_{k'=0}^{\infty}  \wP_{k' + 1} (A_{k'}\mu + \mu_0) (1-1/N)^{k'}}\breg{\X}{\u}{\x_{0}},
    \end{align*}
    \edit{which is the leading constant of $\breg{\X}{\u}{\x_{0}}$ appearing in the statement.}

    \paragraph{3. Controlling~\eqref{eq:stoch:gap3_dual}:} This will follow from similar steps as those shown above, but will rely upon using the probabilistic arguments on each coordinate block of $\hyb_k$. Recall the notation $\breg{I}{\cdot}{\cdot}$ from \Cref{sec:preliminaries}. Using dual-separability of the objective, write
    \begin{align*}
        \Ex[\breg{\Y}{\v}{\hyb_{k}}] &= \sum_{I=1}^N \Ex[\breg{I}{\v_I}{\hyb_{k, I}}]\\
        &= \sum_{I=1}^N \p{(1/N) \Ex[\breg{I}{\v_I}{\y_{k, I}}] + (1-1/N)\Ex[\breg{I}{\v_I}{\hyb_{k-1, I}}]}\\
        &=  (1/N) \sum_{k'=0}^{k}(1-1/N)^{k'} \underbrace{\sum_{I=1}^N \Ex[\breg{I}{\v_I}{\y_{k-k', I}}]}_{\Ex[\breg{\Y}{\v}{\y_{k-k'}}]} + (1-1/N)^{k+1} \breg{\Y}{\v}{\y_{0}}.
    \end{align*}
    The remainder of the argument follows identically to Step 2 and produces the leading constant of $\breg{\Y}{\v}{\y_{0}}$ appearing in the statement.
\end{proof}
Note that the separability of $\breg{\Y}{\cdot}{\cdot}$ is only used in Step 3, which could be eschewed by updating the entirety of $\hyb_{k-1}$ with probability $1/N$ as opposed to the block-wise updates currently being used.
\edit{Furthermore, part of the work in \Cref{thm:stochastic:historic} is providing weight schedules $\br{(\wP_k, \wD_k)}_{k \geq 1}$ that control the constants $C_{\X}$ and $C_{\Y}$ on the initial distance terms. We may now prove \Cref{thm:stochastic:historic}.}
\begin{proof}
    We split the proof into the same case-by-case strategy as employed in \Cref{thm:nonsep:full}. While we may match those arguments exactly for most of the conditions of \Cref{prop:stochastic:historic}, we need to additionally set the correct values of the sequences $(\wP_k)_{k \geq 1}$ and $(\wD)_{k \geq 1}$ to complete the analysis. \edit{First, using that $A_k \leq \p{1+\frac{1}{2N}}^k a_1$ and that $\wP_k \leq 1/N$ for all $k \geq 1$ leads to the following bound on $C_{\X}$ from \Cref{prop:stochastic:historic}:}
    \begin{align*}
        2C_{\X} &= \mu_0 + \sum_{\ell=0}^{\infty} \wP_{\ell+1} (A_{\ell}\mu + \mu_0)  (1-1/N)^{\ell}\\
        &= \mu_0 + \mu\sum_{\ell=0}^{\infty} \wP_{\ell+1} A_{\ell}(1-1/N)^{\ell} + \mu_0 \sum_{\ell=0}^{\infty} \wP_{\ell+1}   (1-1/N)^{\ell}\\
        &= \mu_0 + \frac{a_1\mu}{N}\sum_{\ell=0}^{\infty} (1-1/(2N))^{\ell} + \frac{\mu_0}{N} \sum_{\ell=0}^{\infty} (1-1/N)^{\ell}\\
        &\leq 2\mu_0 + 2a_1\mu,
    \end{align*}
    and similarly, $C_{\Y} \leq a_1\nu + \nu_0$.
    The requirement that $A_k \leq \p{1+\frac{1}{2N}}^k a_1$ introduces a term of the form $N\ln\p{{G_0}/({a_1\epsilon})}$ to the iteration complexity, where $G_0 = (a_1\mu + \mu_0)\breg{\X}{\u}{\x_{0}} + (a_1\nu + \nu_0)\breg{\Y}{\v}{\y_{0}}$. 
    \paragraph{Case 1: $\mu > 0$, $\nu > 0$.} We consider here a constant choice of the sequences and $\wP_{k} = \wP_{0}$ and $\wD_{k} = \wD_{0}$. Then, all conditions on the growth of $(a_k)_{k \geq 1}$ can be satisfied using  $a_k  = \alpha A_{k-1}$, where
    \begin{align*}
        \alpha \lesssim \min\br{\frac{\sqrt{\wD_{0}\mu\nu}}{\Gb_{\pb}}, \frac{\sqrt{\wP_{0}\mu\nu}}{\Gb_{\qb, \bgam}}, \frac{\sqrt{\wP_{0}}\mu}{\Lb_{\pb, \bgam}}}.
    \end{align*}
    This implies that $A_k = (1+\alpha) A_{k-1}$. Then, considering~\eqref{eq:stoch:primal:decay}, we have that
    \begin{align*}
        \frac{1}{N}\sum_{k'=0}^{\infty}  \wP_{0} A_{\ell + k'}  (1-1/N)^{k'} =  \frac{1}{N} \wP_{0} A_{\ell} \sum_{k'=0}^{\infty} (1+\alpha)^{k'}  (1-1/N)^{k'}.
    \end{align*}
    Consider a setting of $\alpha$ and a constant $c > 0$ (which may depend on $N$) such that
    \begin{align}
        \frac{1}{N} \sum_{k'=0}^{\infty} (1+\alpha)^{k'}  (1-1/N)^{k'} \leq 1 + \frac{c\alpha}{1+\alpha} \label{eq:w_constant},
    \end{align}
    where the right-hand side can be tightened to $1 + c\alpha/2$ when $\alpha \leq 1$.
    Then, by setting $\wP_0 \leq 2c^{-1}$, the condition~\eqref{eq:stoch:primal:decay} is satisfied (with an identical argument holding for~\eqref{eq:stoch:dual:decay}). When $N = 1$, the term above vanishes, so we may consider $N \geq 2$. To satisfy~\eqref{eq:w_constant}, we need that $\alpha \leq \frac{N - c/2 - 1}{1-N} = \frac{1}{N-1}$ for $c = 2N$, which imposes the condition that $\alpha \leq \frac{1}{N-1}$.
    \edit{Thus, we may set $\wP_{k} = \wP_k = 1/N$ exactly.}
    \paragraph{Case 2: $\mu > 0$, $\nu = 0$.} 
    \edit{We set $\wD_k = 1/N$ for all $k \geq 0$ and need only set $\wP_k$. 
    }
    \edit{To approach this, we use the two-phase argument of Case 2 from \Cref{thm:nonsep:full}, and consider that for the constant sequence $\wP_k = \wP_0$ for all $k \geq 0$, $A_k$ will grow quadratically for all $k \geq k^\star$, where we recall that $k^\star$ is the last iteration for which the second condition of~\eqref{eq:stoch:cond} is dominant. }   
    In particular, for $\ell > k^\star$, we have $a_\ell = c\mu\nu_0 \min\br{\frac{\wD_{0}}{\Gb_{\pb}^2}, \frac{\wP_{0}}{\Gb_{\qb, \bgam}^2}}\ell$ for an absolute constant $c > 0$, and moreover, that
    \begin{align*}
        A_\ell \leq c\mu\nu_0 \min\br{\frac{\wD_0}{\Gb_{\pb}^2}, \frac{\wP_{0}}{\Gb_{\qb, \bgam}^2}} \frac{\ell(\ell+1)}{2}.
    \end{align*}
    Furthermore, the condition~\eqref{eq:stoch:primal:decay} is satisfied if
    \begin{align*}
        \frac{1}{N} \wP_{0}\sum_{k'=0}^{\infty}  \frac{(\ell + k')(\ell + 1 + k')}{2}  (1-1/N)^{k'} \leq \wP_{0} \frac{\ell(\ell+1)}{2} + \ell
    \end{align*}
    which in turn is satisfied when
    \begin{align*}
        \wP_{0}\underbrace{\frac{1}{N}\sum_{k'=0}^{\infty} \frac{k'(2\ell + k' + 1)}{2}  (1-1/N)^{k'}}_{\leq N\ell + N^2 + 1/2} \leq  \ell,
    \end{align*}
    where the upper bound follows by summing over $k'$ and applying $(1 - 1/N) \leq 1$. Set
    \begin{align*}
        \wP_{0} = \frac{1}{N + N^2 + 1/2},
    \end{align*}
    which does not depend on $k^\star$ and can be used to achieve the given complexity. \edit{For $\ell \leq k^\star$, we repeat the argument of Case 1 above, as $\wP_{0} = \frac{1}{N + N^2 + 1/2} \leq \frac{1}{N}$.}
    \paragraph{Case 3: $\mu = 0$, $\nu > 0$.} 
    Here, we may set $\wP_k = 1/N$ for all $k \geq 0$ and derive the required setting for $(\wD_k)_{k \geq 1}$. We may repeat the argument above to set $\wD_k \sim 1/N^2$ and follow Case 3 of \Cref{thm:nonsep:full}, which achieves the given complexity. 
    \paragraph{Case 4: $\mu = 0$, $\nu = 0$.} The conditions~\eqref{eq:stoch:primal:decay} and~\eqref{eq:stoch:dual:decay} vanish, so we reuse the sequence $A_t = O({\min\br{\sqrt{\mu_0\nu_0}/G, \mu_0/L} t})$ as before to complete the proof.
\end{proof}

\subsection{Strategy 2: Non-Uniform Block Replacement Probabilities }\label{sec:stochastic:minty}
\begin{algorithm}[t]
\setstretch{1.2}
   \caption{Non-Uniform Block Replacement Version of \Cref{algo:drago_v2} (Non-Separable)}
   \label{algo:stochastic:minty}
    \begin{algorithmic}[1]
   \State {\bfseries Input:} Initial point $(\x_0, \y_0)$, averaging sequence $(a_k)_{k=0}^t$, sampling vectors $(\pb, \qb, \rb, \boldsymbol{s})$.
   \State Initialize the comparison points $\hxb_{0, I} = \x_0$ for all $I \in [N]$ and $\hyb_0 = \y_0$. Initialize the table $\hfb_0 = \fb(\x_0)$ and $\hJb_0 = \grad \fb(\x_0)$, and compute the product $\hJb_{0}^\top \hyb_0$.
   \For {$k=1$ {\bfseries to} $t$}
       \State $\sub_1$: Compute $\g_{k-1}$ using $\pb$ and~\eqref{eq:primal_grad_est}.
       \State \State Perform the primal update:
       \begin{align}
           \x_k &= \argmin_{\x \in \X} \ a_k \ip{\g_{k-1}, \x} + a_k\phi(\x) + \tfrac{A_{k-1}\mu + \mu_0}{2}\breg{\X}{\x}{\x_{k-1}}.\notag
       \end{align}
       \State $\sub_2$: Draw $R_k \sim \rb$, and update $(\hxb_{k, I})_{I=1}^N$ via~\eqref{eq:table_elements}. Compute $\bfb_{k-1/2}$ using $\qb$ and~\eqref{eq:dual_grad_est}.
       \State Perform the dual update:
       \begin{align}
           \y_k &= \argmax_{\y \in \Y} \ a_k \ip{\y, \bfb_{k-1/2}} - a_k\psi(\y) - \tfrac{A_{k-1}\nu + \nu_{0}}{2}  \breg{\Y}{\y}{\y_{k-1}}. \notag
       \end{align}
       \State $\sub_3$: Draw $S_k \sim \boldsymbol{s}$ and update $\hyb_k$ via~\eqref{eq:table_elements}. Update the product $\hJb_{k}^\top \hyb_k$ via~\eqref{eq:g_replacement}.
   \EndFor
   \Return $A_t^{-1} \sum_{k=1}^t a_k(\x_k, \y_k)$.
\end{algorithmic}

\end{algorithm}

In the previous approach, we relied on the non-uniform weights $(\gamma_I)_{I=1}^N$ in order to achieve complexities that were independent of the number of blocks $N$. Here, rather than relying on the historical regularization, we will instead tune the sampling probabilities $\rb = (r_1, \ldots, r_N)$ and $\sbold = (s_1, \ldots, s_N)$, which govern the element $R_k$ of $\hxb_{k-1, 1}, \ldots, \hxb_{k-1, N}$ and which coordinate block $S_k$ of $\hyb_{k-1}$ gets updates at each iteration $k$ (see~\eqref{eq:table_elements}). \edit{The full algorithm is given in \Cref{algo:stochastic:minty}. Unlike the historical regularization strategy, we may set $\wP_k=\wD_k = 0$ for this technique. As a result, we bound the same terms as in \Cref{prop:stochastic:historic}, although~\eqref{eq:stoch:gap3_primal} and~\eqref{eq:stoch:gap3_dual} are immediately non-positive.

This method inherits the exact per-iteration complexity as the one analyzed in \Cref{thm:stochastic:historic}. Maintaining the matrix-vector multiplication $\hJb_{k}^\top \hyb_{k} \in \R^{d}$ operates just as in \Cref{algo:stochastic:historic}.
The total arithmetic complexity is given by $\tilde{O}(n(d+N)/N)$. The main difference in the algorithms from the upcoming \Cref{sec:sep}, from a per-iteration complexity viewpoint, is that the full $\tilde{O}(n)$ cost update of $\y_k$ is replaced by a single block update of cost $\tilde{O}(n/N)$. 
}

\edit{Proceeding to the iteration complexity}, the result will depend on the sampling probabilities $\pb$, $\qb$, $\rb$, and $\sbold$ through the following constants:
\begin{align*}
    \Gb_{\pb, \sbold} := \sqrt{\max_{I \in [N]}\frac{\Gb_I^2}{p_I s_I^2}}, \Lb_{\pb, \rb} := \sqrt{\sum_{I=1}^N \frac{\Lb_I^2}{p_Ir_I^2}}, \text{ and } \Gb_{\qb, \rb} := \sqrt{\sum_{I=1}^N \frac{\Gb_I^2}{q_I r_I^2}}.
\end{align*}
Observe the following.
\begin{thm}\label{thm:stochastic:minty}
    Under \Cref{asm:smoothness} and \Cref{asm:str_cvx}, consider any $\u \in \X$, $\v \in \Y$ and precision $\epsilon > 0$. Assume that $\min_I r_I \geq 1/(2N)$ and $\min_I s_I \geq 1/(2N)$.
    There exists a choice of the sequence $(a_k)_{k=1}^t$ such that \Cref{algo:drago_v2} with Identity Card 2 produces an output point $(\tx_t, \ty_t) \in \X \times \Y$ satisfying $\E{}{\gap^{\u, \v}(\tx_t, \ty_t)} \leq \epsilon$ for $t$ that depends on $\epsilon$ according to the following iteration complexities. 
    \edit{Below, we state $a_k$ and the iteration complexity in big-$O$ terms, and set $\alpha \sim \min\br{{\sqrt{\mu\nu}}/{(\Gb_{\pb, \sbold} \vee  \Gb_{\qb, \rb})}, {\mu}/{\Lb_{\pb, \rb}}, 1/N}$.}

    \begin{center}
    \begin{adjustbox}{max width=\linewidth}
    \begin{tabular}{ccc}
    \toprule
        {\bf Case} & \edit{{\bf $a_k$ }} & {\bf Iteration Complexity}\\
        \midrule
        $\mu > 0$, $\nu > 0$
        &
        \edit{$(1+\alpha)^k$}
        &
        $\p{N + \frac{\Lb_{\pb, \rb}}{\mu} + \frac{\Gb_{\pb, \sbold} \vee \Gb_{\qb, \rb}}{\sqrt{\mu\nu}}}\ln\p{\frac{1}{\epsilon}}$
        \\
        $\mu > 0$ and $\nu = 0$
        &
        \edit{$k$}
        &
        $\p{N + \frac{\Lb_{\pb, \rb}}{\mu}}\ln\p{\frac{1}{\epsilon}} + (\Gb_{\pb, \sbold} \vee \Gb_{\qb, \rb})\sqrt{\frac{\sqrt{\mu_0/\nu_0}D_0}{\mu \epsilon}}$
        \\
        $\mu = 0$, $\nu > 0$
        &
        \edit{$1$}
        &
        $N\ln\p{\frac{1}{\epsilon}} + \frac{\Lb_{\pb, \rb}\sqrt{\nu_0/\mu_0}D_0}{\epsilon} + (\Gb_{\pb, \sbold} \vee \Gb_{\qb, \rb})\sqrt{\frac{\sqrt{\nu_0/\mu_0}D_0}{\nu \epsilon}}$
        \\
        $\mu = 0$, $\nu = 0$
        &
        \edit{$1$}
        &
        $N\ln\p{\frac{1}{\epsilon}} + \frac{\p{\Lb_{\pb, \rb}\sqrt{\nu_0/\mu_0}  + (\Gb_{\pb, \sbold} \vee \Gb_{\qb, \rb})}D_0}{\epsilon}$\\
    \bottomrule
    \end{tabular}
    \end{adjustbox}
    \end{center}
\end{thm}

We provide similar computations as those used in \Cref{thm:stochastic:historic} to uncover the dependence on the sampling scheme. We again use $\lambb = (\lamb_1, \ldots, \lamb_N)$ where $\lamb_I := \sqrt{\Gb_I^2 + \Lb_I^2}$ from \Cref{sec:preliminaries}. The non-uniform sampling complexity given below follows by letting $p_I \propto \lamb_I^{1/2}$, $r_I \propto \lamb_I^{1/2}$, $s_I \propto \Gb_I^{1/2}$ and $q_I \propto \Gb_I^{1/2}$.
\begin{center}
\begin{adjustbox}{max width=\linewidth}
\begin{tabular}{c|cc}
\toprule
    {\bf Constant} & {\bf Uniform Sampling} & {\bf Non-Uniform Sampling}\\
    \midrule
    $\Gb_{\pb, \sbold} \vee \Gb_{\qb, \rb} $ & $N^{3/2} \norm{\Gbb}_2$ & $\norm{\lambb}_{1/2}^{1/2} \norm{\Gbb}_{1/2}^{1/2}$\\
    $\Lb_{\pb, \rb}$ & $N^{3/2}\norm{\Lbb}_2$ & $\norm{\lambb}_{1/2}^{3/4}  \norm{\Lbb}_{1/2}^{1/4}$\\
    \bottomrule
\end{tabular}
\end{adjustbox}
\end{center}
Notice that the complexities in \Cref{thm:stochastic:historic} (as opposed to the ones shown in \Cref{thm:stochastic:minty}) depend on additional factors in $N$. Thus, taking $\mu > 0$ and $\nu > 0$ as an example, in the case of uniform sampling, the method of \Cref{thm:stochastic:historic} has an $N^{3/2}(\norm{\Gbb}_\infty + \norm{\Lbb}_\infty)$ dependence on problem constants, whereas \Cref{thm:stochastic:minty} has an $N^{3/2}(\norm{\Gbb}_2 + \norm{\Lbb}_2)$ dependence. These are the same in the case of highly non-uniform Lipschitz constants, but have a $\sqrt{N}$ factor difference for approximately uniform Lipschitz constants. \edit{In \Cref{prop:stochastic:minty}, we derive the conditions on the growth of $(a_k)_{k \geq 1}$ that leads to the choices in \Cref{thm:stochastic:historic}.}

\begin{prop}\label[proposition]{prop:stochastic:minty}
    Let $(\x_0, \y_0) \in \ri(\dom(\phi)) \times \ri(\dom(\psi))$ and $\br{(\x_k, \y_k)}_{k\geq 1}$ be generated using $\g_{k-1}$ and $\bfb_{k-1/2}$ given by~\eqref{eq:primal_grad_est} and~\eqref{eq:dual_grad_est}, respectively. 
    Define $a_1 = \frac{1}{15}\min\br{\frac{\sqrt{\mu_0 \nu_0}}{\Gb_{\pb, \sbold} \vee \Gb_{\qb, \rb}}, \frac{\mu_0}{\Lb_{\pb, \rb}}}$ and select $(a_k)_{k \geq 2}$ such that both the conditions
    \begin{align}
        \frac{a_{k}^2}{A_k\mu + \mu_0} &\leq \min_I (1 + (s_I \wedge r_I)/5)\frac{a_{k-1}^2}{A_{k-1}\mu + \mu_0} \label{eq:stoch:N_mu}\\
        \frac{a_{k}^2}{A_k\nu + \nu_0} &\leq \min_I (1 + r_I/5)\frac{a_{k-1}^2}{A_{k-1}\nu + \nu_0} \label{eq:stoch:N_nu}
    \end{align}
    and 
    \begin{align}
        a_{k} \leq \min\br{\frac{\sqrt{(A_k\mu + \mu_0)(A_{k-1}\nu + \nu_0)}}{30\Gb_{\pb, \sbold}}, \frac{\sqrt{(A_k\mu + \mu_0)(A_{k-1}\mu + \mu_0)}}{40\Lb_{\pb, \rb}}, \frac{\sqrt{(A_k\nu + \nu_0)(A_{k-1}\mu + \mu_0)}}{30\Gb_{\qb, \rb}}}\label{eq:stoch:rate}
    \end{align}
    are satisfied. We have that for any $\u \in \X$ and $\v \in \Y$, 
    \begin{align*}
        \sum_{k=1}^t a_k \Ex[\gap^{\u, \v}(\x_k, \y_k)] + \frac{1}{4}\Ex[\TP_t] + \frac{1}{4}\Ex[\TD_t]
        &\leq \frac{1}{2}\TP_{0} + \frac{1}{2}\TD_{0}.
    \end{align*}
\end{prop}
\begin{proof}
    As stated before, we aim to show that the sum of terms in~\eqref{eq:stoch:gap2_cancellation} is upper bounded by a constant independent of $t$. This is composed of the terms $\Ex[\IP_k]$ and $\Ex[\ID_k]$. We divide this task into bounding the primal and dual components separately. As before, by applying the argument leading to the expressions~\eqref{eq:stoch:error_primal} and~\eqref{eq:stoch:error_dual}, bounding the inner product terms $\Ex[\IP_k]$ and $\Ex[\ID_k]$ reduces to bounding the error terms $\Ex[\EP_k]$ and $\Ex[\ED_k]$ and the final element of the telescoping inner product terms $a_t \ip{\grad \fb(\x_t)^\top \y_t - \hJb_{t-1}^\top \hyb_{t-1}, \u - \x_t}$ and $a_t \langle\grad \fb(\x_t) - \hfb_{t}, \v - \y_t\rangle$.

    \paragraph{1. Controlling $\Ex[\IP_k]$:} The first step follows similarly to Step 1 from the proof of \Cref{prop:stochastic:historic}. 
    By Young's inequality with parameter $(A_{k-1}\mu + \mu_0)/4$, we have
    \begin{align}
        \Ex[\EP_k] &= a_{k-1} \Ex\ip{\tfrac{1}{p_{P_k}}\textstyle\sum_{i \in B_{P_k}} (y_{k-1, i}\grad f_{i}(\x_{k-1}) - \hy_{k-2, i} \hgb_{k-2, i}),  \x_{k-1} - \x_{k}} \notag \\
        &\leq \frac{1}{4} \Ex[\CP_k]+ \frac{2a_{k-1}^2}{A_{k-1}\mu + \mu_0} \Ex\norm{\tfrac{1}{p_{P_k}}\textstyle\sum_{i \in B_{P_k}} (y_{k-1, i}\grad f_{i}(\x_{k-1}) - \hy_{k-2, i} \hgb_{k-2, i})}_{\X^*}^2 \label{eq:stoch:minty:primal:middle}
    \end{align}
    and
    \begin{align}
        &-\Ex[a_t \langle \grad \fb(\x_t)^\top \y_t - \hJb_{t-1}^\top \hyb_{t-1}, \u - \x_t\rangle] \notag\\
        &\quad \leq \frac{1}{4}\Ex[\TP_{t}] + \frac{2a_{t-1}^2}{A_{t-1}\mu + \mu_0} \Ex\norm{\tfrac{1}{p_{P_{t+1}}}\textstyle\sum_{i \in B_{P_{t+1}}} (y_{t, i}\grad f_{i}(\x_{t}) - \hy_{t-1, i} \hgb_{t-1, i})}_{\X^*}^2. \label{eq:stoch:minty:primal:last}
    \end{align}    
    Then, we apply \Cref{lem:separation} with $b_I = \Lb_I^2/p_I$ and $c_I = \Gb_I^2/p_I$ to handle the second term in~\eqref{eq:stoch:minty:primal:middle} (and~\eqref{eq:stoch:minty:primal:last}), and write
    \begin{align*}
         &\Ex\norm{\tfrac{1}{p_{P_k}}\textstyle\sum_{i \in B_{P_k}} (y_{k-1, i}\grad f_{i}(\x_{k-1}) - \hy_{k-2, i} \hgb_{k-2, i})}_{\X^*}^2\\
         &\leq  2\sum_{I=1}^N \frac{\Lb_I^2}{p_I}\red{\Ex\norm{\x_{k-1} - \hxb_{k-2, I}}_{\X}^2} + 2\sum_{I=1}^N \frac{\Gb_I^2}{p_I} \blue{\Ex\norm{\y_{k-1, I} - \hyb_{k-2,I}}_{2}^2}.
    \end{align*}
    
    Individually, the colored table bias terms can be further upper bounded by applying \Cref{lem:table_updates} (stated after this proof), so that the sums of~\eqref{eq:stoch:minty:primal:middle} and~\eqref{eq:stoch:minty:primal:last} can be further developed to
    \begin{align}
        &\sum_{k=2}^t \Ex[\IP_k] \leq \frac{1}{4} \sum_{k=2}^t \Ex[\CP_k] + \frac{1}{4}\Ex[\TP_{t}]\notag\\
        &\quad + 20\sum_{k=1}^{t} \frac{a_{k}^2}{A_{k}\mu + \mu_0} \sum_{I=1}^N \frac{\Gb_I^2}{p_I \red{s_I}} \red{\sum_{k'=1}^k (1-s_I/2)^{k-k'}\Ex\norm{\y_{k', I} - \y_{k'-1, I}}_2^2} \label{eq:geom_dual}\\
        &\quad + 20\sum_{k=1}^{t} \frac{a_{k}^2}{A_{k}\mu + \mu_0} \sum_{I=1}^N \frac{\Lb_I^2}{p_I\blue{r_I}} \blue{\sum_{k'=1}^k (1-r_I/2)^{k-k'}\Ex\norm{\x_{k'} - \x_{k'-1}}_\X^2}. \label{eq:geom_primal}
    \end{align}
    To control the resulting sums~\eqref{eq:geom_dual} and~\eqref{eq:geom_primal}, we exchange the order of summation to compute them. First,~\eqref{eq:geom_dual} can be written as 
    \begin{align}
        &20\sum_{k'=1}^{t}\sum_{I=1}^N \frac{\Gb_I^2}{p_Is_I} \Ex\norm{\y_{k', I} - \y_{k'-1, I}}_2^2 \cdot \sum_{k = k'}^t \frac{a_{k}^2}{A_{k}\mu + \mu_0}  (1-s_I/2)^{k-k'} \notag\\
        &\leq 100\sum_{k'=1}^{t} \frac{a_{k'}^2}{A_{k'}\mu + \mu_0} \sum_{I=1}^N \frac{\Gb_I^2}{p_Is_I^2}  \Ex\norm{\y_{k', I} - \y_{k'-1, I}}_2^2 \label{eq:geom_dual1}\\
        &\leq 100\underbrace{\p{\max_{I \in [N]}\frac{\Gb_I^2}{p_I s_I^2}}}_{\Gb_{\pb, \sbold}^2}\sum_{k'=1}^{t} \frac{a_{k'}^2}{A_{k'}\mu + \mu_0} \Ex\norm{\y_{k'} - \y_{k'-1}}_{\Y}^2,  \label{eq:stoch:ep1}
    \end{align}
    where the inequality~\eqref{eq:geom_dual1} follows by the given assumption~\eqref{eq:stoch:N_mu} that $\frac{a_{k}^2}{A_{k}\mu + \mu_0} \leq \min_I \p{1 + s_I/5}\frac{a_{k-1}^2}{A_{k-1}\mu + \mu_0}$ and the sequence of steps
    \begin{align*}
          \sum_{k=k'}^t \frac{a_{k}^2}{A_{k}\mu + \mu_0} (1-s_I/2)^{k-k'} &\leq  \sum_{k=k'}^t [(1-s_I/2)(1+s_I/5)]^{k-k'} \frac{a_{k'}^2}{A_{k'}\mu + \mu_0}\\
        &\leq   \sum_{k=k'}^t (1-s_I/5)^{k-k'} \frac{a_{k'}^2}{A_{k'}\mu + \mu_0} \leq \frac{5}{s_I} \frac{a_{k'}^2}{A_{k'}\mu + \mu_0}.
    \end{align*}
    In~\eqref{eq:stoch:ep1}, we also used that $\norm{\cdot}_2 \leq \norm{\cdot}_\Y$. This argument is the most technical part of the analysis and is repeated two more times in the remainder of the proof. The first of the two is to upper bound~\eqref{eq:geom_primal} by the quantity
    \begin{align}
        100\underbrace{\p{\sum_{I=1}^N \frac{\Lb_I^2}{p_Ir_I^2}}}_{\Lb_{\pb, \rb}^2} \sum_{k'=1}^{t} \frac{a_{k'}^2}{A_{k'}\mu + \mu_0}  \Ex\norm{\x_{k'} - \x_{k'-1}}_\X^2, \label{eq:stoch:ep2}
    \end{align}
    whereas the second appears in the steps used to bound the dual error terms below. 
    
    \paragraph{2. Controlling $\Ex[\ID_k]$:}
    The following steps are the dual analog of the ones shown above for the primal terms. First, we have the decomposition of the inner product term
    \begin{align}
        \Ex_{k-1/2}[\ID_k] &= a_k \Ex_{k-1/2}\ip{\fb(\x_k) - \hfb_{k}, \v - \y_k}\notag\\
        &\quad- a_{k-1} \ip{\fb(\x_{k-1}) - \hfb_{k-1}, \v - \y_{k-1}} - \Ex_{k-1/2}[\ED_k].\notag
    \end{align}
    Applying Young's inequality with parameter $(A_{k-1}\nu + \nu_0)/4$ we have that
    \begin{align}
    \Ex_{k-1/2}[\ED_k] &= a_{k-1} \Ex_{k-1/2}\ip{\tfrac{1}{q_{Q_k}}\textstyle\sum_{j \in B_{Q_k}} (f_{j}(\x_{k-1}) - \hf_{k-1, j})\e_j ,  \y_k - \y_{k-1}} \notag\\
        &\leq \frac{1}{4}\Ex_{k-1/2}[\CD_k] + \frac{2a_{k-1}^2}{A_{k-1} \nu + \nu_0} \sum_{J=1}^N \frac{1}{q_J}\norm{\textstyle\sum_{j \in B_{J}} \p{f_j(\x_{k-1}) - \hf_{k-1, j}}\e_j}_{\Y^*}^2 \notag\\
        &\leq \frac{1}{4}\Ex_{k-1/2}[\CD_k] + \frac{2a_{k-1}^2}{A_{k-1} \nu + \nu_0} \sum_{J=1}^N \frac{1}{q_J}\norm{\textstyle\sum_{j \in B_{J}} \p{f_j(\x_{k-1}) - \hf_{k-1, j}}\e_j}_{2}^2 \label{eq:dual_norm_bound}\\
        &\leq \frac{1}{4}\Ex_{k-1/2}[\CD_k] + \frac{2a_{k-1}^2}{A_{k-1} \nu + \nu_0} \sum_{J=1}^N \frac{\Gb_J^2}{q_J} \norm{\x_{k-1} - \hxb_{k-1, J}}_{\X}^2 \notag
    \end{align}
    where in~\eqref{eq:dual_norm_bound} we used that $\norm{\cdot}_{\Y^*} \leq \norm{\cdot}_{2}$. Summing over $k$ and taking the marginal expectation, 
    \begin{align}
        \sum_{k=2}^{t} \Ex[\ID_k] \leq \frac{1}{4} \sum_{k=2}^t \Ex[\CD_k] + \frac{1}{4}\Ex[\TD_{t}]+  \frac{2a_{k-1}^2}{A_{k-1} \nu + \nu_0} \sum_{J=1}^N \frac{\Gb_J^2}{q_J} \norm{\x_{k-1} - \hxb_{k-1, J}}_{\X}^2,\label{eq:geom_primal1}
    \end{align}
    where~\eqref{eq:geom_primal1} can be upper bounded using the same arguments leading to~\eqref{eq:stoch:ep2} under the given assumption~\eqref{eq:stoch:N_nu}, yielding 
    \begin{align}
        50\underbrace{\p{\sum_{J=1}^N \frac{\Gb_J^2}{q_J r_J^2}}}_{\Gb_{\qb, \rb}^2} \sum_{k'=1}^{t} \frac{a_{k'}^2}{A_{k'}\nu + \nu_0}  \Ex\norm{\x_{k'} - \x_{k'-1}}_\X^2. \label{eq:stoch:ed}
    \end{align}
    To summarize progress thus far, we must cancel the terms~\eqref{eq:stoch:ep1},~\eqref{eq:stoch:ep2}, and~\eqref{eq:stoch:ed} to complete the proof, which requires setting the appropriate conditions on the sequence $(a_k)_{k\geq 1}$.

    \paragraph{3. Deriving the rate conditions:}
    Under the condition~\eqref{eq:stoch:rate}, we may bound~\eqref{eq:stoch:ep1} by $\frac{1}{4}\sum_{k'=1}^t \Ex[\CD_{k'}]$,~\eqref{eq:stoch:ep2} by $\frac{1}{8}\sum_{k'=1}^t \Ex[\CP_{k'}]$, and~\eqref{eq:stoch:ed} by $\frac{1}{8}\sum_{k'=1}^t \Ex[\CP_{k'}]$. All terms of~\eqref{eq:stoch:gap2_cancellation} now cancel, completing the proof.
\end{proof}

The following technical lemma was used to express the terms that quantified the table bias terms (distance between the iterates and their counterparts in the respective tables) with terms that can be canceled by quantities in~\eqref{eq:stoch:gap2_cancellation}.
\begin{lemma}{\citep[Lemma 2]{diakonikolas2025block}}\label[lemma]{lem:table_updates}
    For any $k \geq 1$ and $I \in [N]$, the following hold:
    \begin{align*}
        \Ex\norm{\x_k - \hxb_{k-1, I}}_\X^2 &\leq \frac{5}{r_I} \sum_{k'=1}^k (1-r_I/2)^{k-k'}\Ex\norm{\x_{k'} - \x_{k'-1}}_\X^2,\\
        \Ex\norm{\y_{k, I} - \hyb_{k-1, I}}_2^2 &\leq \frac{5}{s_I} \sum_{k'=1}^k (1-s_I/2)^{k-k'}\Ex\norm{\y_{k', I} - \y_{k'-1, I}}_2^2.
    \end{align*}
\end{lemma}
We now convert the \Cref{prop:stochastic:minty} into a complexity guarantee, \edit{by proving \Cref{thm:stochastic:minty}.}
\begin{proof}
    The exact case-by-case strategy of \Cref{thm:nonsep:full} can be applied (ignoring absolute constant factors) by setting $G \leftarrow \Gb_{\pb, \sbold} \vee \Gb_{\qb, \rb}$ and $L \leftarrow \Lb_{\pb, \rb}$. The only additional conditions that need to be incorporated are~\eqref{eq:stoch:N_mu} and~\eqref{eq:stoch:N_nu}. Under the given assumption that $\min_I r_I \geq 1/(2N)$ and $\min_I s_I \geq 1/(2N)$, these conditions will be satisfied when, for all $k \geq 2$, it holds that $\frac{a_{k}^2}{A_k\mu + \mu_0} \leq \p{1 + 1/(10N)}\frac{a_{k-1}^2}{A_{k-1}\mu + \mu_0}$ and $\frac{a_{k}^2}{A_k\nu + \nu_0} \leq \p{1 + 1/(10N)}\frac{a_{k-1}^2}{A_{k-1}\nu + \nu_0}$.
    Taking the first condition as an example, it can be rewritten as
    \begin{align*}
        \frac{a_k^2}{a_{k-1}^2} \leq \p{1 + \frac{1}{10N}} \frac{A_k\mu + \mu_0}{A_{k-1}\mu + \mu_0}.
    \end{align*}
    When $\mu = 0$ and $\nu = 0$, this condition is satisfied automatically as the ratio on the left-hand side is equal to $1$ (because $a_k$ is a constant sequence). Otherwise, because $A_k$ is an increasing sequence, we can reduce the condition to $a_k^2 \leq (1+ 1/(10N)) a_{k-1}^2$. The fastest growth condition on $(a_k)_{k \geq 1}$ that is possible under the constraint is $a_k \leq (1+\alpha)a_{k-1}$, where $(1+\alpha) \leq \sqrt{1 + 1/(10N)}$. Then,
    \begin{align*}
        \sqrt{1 + 1/(10N)} \geq \sqrt{1 + 69/(900N)} \geq 1+1/(30N),
    \end{align*}
    so the imposition $\alpha \leq \frac{1}{30N}$ suffices. This adds an $O\p{N\ln\p{{1}/{\epsilon}}}$
    term to the resulting complexities and completes the proof.
\end{proof}

\section{An Algorithm for Dual-Separable Problems}\label{sec:sep}
Given \Cref{def:separable}, we design an algorithm variant that performs stochastic block-wise updates in the dual variable, akin to similar strategies applied to bilinearly coupled objectives \citep{song2021variance}. Precisely, we will only update a single randomly chosen block $Q_k$ on each iteration $k$, in an effort to achieve an improved complexity guarantee. Updates in this form introduce additional technical challenges because different blocks of the dual iterate $\y_k = (\y_{k, 1}, \ldots, \y_{k, N})$ have different dependences on the block index $Q_k$. As such, we carefully handle the expectations computations in the upcoming \Cref{lem:expect}. Furthermore, a key difference in the proof structure of this section is that we will track an auxiliary sequence $(\byb_k)_{k\geq 1}$ of \emph{return values}, such that the algorithm returns $(\x_t, \byb_t)$ in the final iteration instead of $(\x_t, \y_t)$. \edit{This sequence is determined by taking the full update of the dual variable on iteration $k$ given the result of iteration $k - 1$ and $\wD_k = 0$, with a slight modification to handle the block coordinate-wise nature of the updates. Precisely, we consider}
\begin{align}
    \byb_k = \argmax_{\y \in \Y} \Big\{ &a_k (\ip{\y, N\bfb_{k-1/2} - (N-1)\grad \psi(\y_{k-1})} \notag\\
    &\quad - \psi(\y)) - (A_{k-1}\nu + \nu_{0})\breg{\Y}{\y}{\y_{k-1}} \Big\},\label{eq:full_dual_update_sep}
\end{align}
\edit{where $\grad \psi(\y) := \p{\grad\psi_1(\y_1), \ldots, \grad\psi_N(\y_N)}$ is an arbitrary but consistently chosen subgradient of $\psi$. }
Conceptually, each block $\byb_{k, J}$ represents the $J$-th block of $\y_k$ if $J = Q_k$, or if block $J$ was the one updated at time $k$. In other words, it stores all possible block updates that could have occurred from step $k-1$ to step $k$ in one vector. 
\edit{The actual update used in the algorithm (which considers only block $Q_k$), is described in the upper bound \Cref{lem:sep:upper}.}
Similar to before, we will define our update sequences in the process of deriving upper and lower bounds on $a_k\Ex[\Lcal(\x_k, \v)]$ and $a_k\Ex[\Lcal(\u, \byb_k)]$.

Crucially, we do not need to compute the elements of the sequence $(\byb_k)_{k\geq 1}$, as doing so would defeat the purpose of considering coordinate-wise updates. 
Instead, %
we may realize~\eqref{eq:gap_bound} in expectation by computing only one instance of $\byb_k$ with the following trick \citep{alacaoglu2022complexity}: we randomly draw an index $\hat{t}$ (independent of all other randomness in the algorithm and of $(\u, \v)$) from $\{1, \ldots, t\}$, where $\hat{t} = k$ with probability $a_k/A_t$. Thus, by computing the conditional expectation over $\hat{t}$ given the sequence of iterates,
\begin{align*}
    \Ex[\gap^{\u, \v}(\x_{\hat{t}}, \byb_{\hat{t}}) | \mc{F}_t] = \sum_{k=1}^t a_k \gap^{\u, \v}(\x_k, \byb_k).
\end{align*}
In practice, we may simply run the algorithm to iteration $\hat{t}$. \edit{As before, we produce two algorithms, based on the ideas of historical regularization (\Cref{sec:sep:historic}) and non-uniform block replacement (\Cref{sec:sep:minty}). One difference between \Cref{sec:sep} and \Cref{sec:nonsep:stochastic} is that we will modify the dual update to modify a single coordinate block, and therefore will not consider the parameter $\wD_k$. We will provide a new lower and upper bound on objective that will replace \Cref{claim:gap} as our initial primal-dual gap bound. As before, we collect results that are common to both algorithms below, and then complete the analysis for either algorithm in \Cref{sec:sep:historic} and \Cref{sec:sep:minty}, respectively.}

First, we use the following technical lemma to provide expectation formulas regarding $\y_k$ and $\byb_k$, \edit{which will be used in the upper bound that is analogous to \Cref{lem:upper}.}
\begin{lemma}\label[lemma]{lem:expect}
    Let $h: \Y \times \X \to \R$ be block separable in its first argument, i.e., $h(\y, \x) = \sum_{J=1}^N h_J(\y_{J}, \x)$ for $J \in [N]$ and $\y \in \Y$. Assume that $\y_{k, J} = \y_{k-1, J}$ for all $J \neq Q_k$, that $\byb_{k, Q_k} = \y_{k, Q_k}$, and that $\byb_{k}$ is $\F_{k-1/2}$-measurable. Then, if $Q_k$ is sampled uniformly on $[N]$, it holds that
    \begin{align}
        N\E{}{h_{Q_k}(\y_{k, Q_k}, \x_k)} = N\E{}{h(\y_k, \x_k)} - (N-1) \E{}{h(\y_{k-1}, \x_k)} = \E{}{h(\byb_k, \x_k)}. \label{eq:y_bar_expt}
    \end{align}
\end{lemma}
\begin{proof}
    Write
    \begin{align*}
        \E{k-1/2}{h_{Q_k}(\y_{k, Q_k}, \x_k)} 
        &= \E{k-1/2}{h(\y_k, \x_k)} - \E{k-1/2}{\textstyle\sum_{J \neq Q_k} h_J(\y_{k, J}, \x_k)}\\
        &= \E{k-1/2}{h(\y_k, \x_k)} - \E{k-1/2}{\textstyle\sum_{J \neq Q_k} h_J(\y_{k-1, J}, \x_k)}\\
        &= \E{k-1/2}{h(\y_k, \x_k)} - \frac{N-1}{N} h(\y_{k-1}, \x_k).
    \end{align*}
    Take the marginal expectation of both terms to prove the first result. For the second,
    \begin{align}
        \E{k-1/2}{h(\y_k, \x_k)}
        &= \frac{1}{N}\sum_{J=1}^N \E{k-1/2}{h_J(\byb_{k, J}, \x_k) | J = Q_k} \notag\\
        &\quad + \frac{N-1}{N} \sum_{J=1}^N \E{k-1/2}{h_J(\y_{k-1, J}, \x_k)| J \neq Q_k} \notag\\
        &= \frac{1}{N}\sum_{J=1}^N h_J(\byb_{k, J}, \x_k) + \frac{N-1}{N} \sum_{J=1}^N h_J(\y_{k-1, J}, \x_k) \notag\\
        &= \frac{1}{N}h(\byb_k, \x_k) + \frac{N-1}{N}h(\y_{k-1}, \x_k).\notag
    \end{align}
    Rearrange terms and apply the marginal expectation to achieve the second result.
\end{proof}
The condition that $\byb_{k}$ is $\F_{k-1/2}$-measurable reflects the viewpoint that $\byb_k$ can be computed via a deterministic update given $\x_k$, after which $\y_k$ can be computed exactly from the random triple $(\y_{k-1}, \byb_k, Q_k)$.

We proceed to the details of the convergence analysis. The formula for $\g_{k-1} \in \R^d$ is given after rigorously introducing the sequence $(\byb_k)_{k \geq 0}$, on which $\g_{k-1}$ depends. We will first specify the upper bound and dual update.
We define the update for $\y_k$ (\edit{which elucidates the update of $\byb_k$}). Because of separability, we may perform this update at $O(b)$ for $b := n/N$ cost on average across blocks.  We do so by including only one additive component of $\psi$ in the objective that $\y_k$ maximizes. 
We notice that the strong concavity constant of the objective defining $\y_k$ may change if we only use some components of $\psi$. To account for this, we design a slightly different proximity term from the one in \Cref{algo:drago_v2}. 
Recall the definition of the Bregman divergences $\breg{J}{\cdot}{\cdot}$ from \Cref{sec:preliminaries}.
The upper bound is stated in expectation, as the results may not hold for all realizations (as was the case for previous versions of the initial gap bounds).
\begin{lemma}\label[lemma]{lem:sep:upper}
    For all $k \geq 1$, consider the update
    \begin{align}
        \y_k = \argmax_{\y \in \Y} \Big\{ &a_k \p{N\y_{Q_k}\barf_{k-1/2, Q_k} - N\psi_{Q_k}(\y_{Q_k})} + Na_k\nu\breg{Q_k}{\y_{Q_k}}{\y_{k-1, Q_k}} \notag\\
        & - a_k \nu \breg{\Y}{\y}{\y_{k-1}}  - \tfrac{A_{k-1}\nu + \nu_{0}}{2}\breg{\Y}{\y}{\y_{k-1}} \Big\},\label{eq:dual_update_sep}
    \end{align}
    If we have that $a_k\nu \leq \frac{A_{k-1}\nu + \nu_0}{4(N-1)}$ (when $N \geq 2$),
    then it holds that
    \begin{align}
         a_k\Ex[\Lcal(\x_k, \v)]
        &\leq  a_k\Ex\sbr{\mc{L}(\x_k, \byb_k)} + \tfrac{1}{2}\Ex[\TD_{k-1}] - \tfrac{1}{2}\Ex[\TD_k] - \tfrac{1}{4}\Ex[\CD_k] \notag\\
        &\quad + a_k\Ex\ip{\v - \byb_k, \fb(\x_k) - \bfb_{k-1/2}}.
    \end{align}
\end{lemma}
\begin{proof}
    We add and subtract terms to expose the objective defining $\y_k$. Write
    \begin{align*}
        a_k\Ex[\Lcal(\x_k, \v)]
        &= a_k\Ex\sbr{\ip{\v, \fb(\x_k)} - \psi(\v) + \phi(\x_k)}\\
        &= a_k\Ex\sbr{N\v_{Q_k} \barf_{k-1/2, Q_k} - N\psi_{Q_k}(\v_{Q_k})}\\
        &\quad + a_k\Ex\sbr{N \nu \breg{Q_k}{\v_{Q_k}}{\y_{k-1, Q_k}} - \nu \breg{\Y}{\v}{\y_{k-1}}} \\
        &\quad - \tfrac{A_{k-1}\nu + \nu_0}{2}\Ex[\breg{\Y}{\v}{\y_{k-1}}] + \tfrac{A_{k-1}\nu + \nu_0}{2}\Ex[\breg{\Y}{\v}{\y_{k-1}}] \\
        &\quad + a_k\Ex\ip{\v, \fb(\x_k) - \bfb_{k-1/2}}  + \phi(\x_k)
    \end{align*}
    where we use that $(\v, \fb(\x_k), \bfb_{k - 1/2}, \y_{k-1})$ is independent of $Q_k$ and that
    \begin{align*}
        \Ex[N \nu \breg{Q_k}{\v_{Q_k}}{\y_{k-1, Q_k}} - \nu \breg{\Y}{\v}{\y_{k-1}}] = 0.
    \end{align*}
    Notice that the sum $- N\psi_{Q_k}(\v_{Q_k}) + N \nu\breg{Q_k}{\v_{Q_k}}{\y_{k-1, Q_k}}$ is linear, hence has strong concavity constant zero; we have that $\v \rightarrow -a_k \nu \breg{\Y}{\v}{\y_{k-1}}$ provides the required $(a_k \nu)$-strong concavity.
    Next, we use the definition of $\y_k$ and \Cref{lem:3pt} to achieve
    \begin{align*}
        a_k\Ex[\Lcal(\x_k, \v)]
        &\leq a_k\Ex\sbr{N \y_{k, Q_k} \barf_{k-1/2, Q_k} - N\psi_{Q_k}(\y_{k, Q_k})} \\
        &\quad - a_k\Ex\sbr{\nu\breg{\Y}{\y_k}{\y_{k-1}} + N \nu \breg{Q_k}{\y_{k, Q_k}}{\y_{k-1, Q_k}}} \\
        &\quad + \phi(\x_k) - \tfrac{1}{2}\Ex[\CD_k] + \tfrac{1}{2}\Ex[\TD_{k-1}]  - \tfrac{1}{2}\Ex[\TD_{k}] + a_k\Ex\ip{\v, \fb(\x_k) - \bfb_{k-1/2}}.
    \end{align*}
    Apply \Cref{lem:expect} to compute each of the expectations involving $Q_k$, and see that
    \begin{align*}
        \E{}{N \y_{k, Q_k} \barf_{k-1/2, Q_k}} &= \E{}{\ip{\byb_{k}, \bfb_{k-1/2}}}\\
        \E{}{N\psi_{Q_k}(\y_{k, Q_k})} &= \E{}{\psi(\byb_{k})}\\
        \Ex[N \nu \breg{Q_k}{\y_{k, Q_k}}{\y_{k-1, Q_k}}] &= N \nu\Ex[\breg{\Y}{\y_k}{\y_{k-1}}].
    \end{align*}
    Plugged into the display above, this achieves
    \begin{align*}
        a_k\Ex[\Lcal(\x_k, \v)]
        &\leq a_k\Ex[\Lcal(\x_k, \byb_k)]- \tfrac{1}{2}\Ex[\CD_k] + \tfrac{1}{2}\Ex[\TD_{k-1}] - \tfrac{1}{2}\Ex[\TD_k]  \\
        &\quad + (N-1) a_k\nu\Ex[\breg{\Y}{\y_k}{\y_{k-1}}] \\
        &\quad + a_k\Ex\ip{\v - \byb_k, \fb(\x_k) - \bfb_{k-1/2}}.
    \end{align*}
    The term $(N-1) a_k\nu\Ex[\breg{\Y}{\y_k}{\y_{k-1}}]$ cancels with $\tfrac{1}{4}\Ex[\CD_k]$ by the assumption on $a_k$ and the strong convexity of Bregman divergences.
\end{proof}

Defining the primal update will reflect two different strategies---namely, those used in \Cref{sec:nonsep:stochastic}. For both cases, recall the table $\hxb_{k, 1}, \ldots, \hxb_{k, N}$ introduced in~\eqref{eq:first_order_table} and~\eqref{eq:table_elements}. We will use the primal gradient estimate
\begin{align}
    \g_{k-1} &= \hJb_{k-1}^\top \y_{k-1} + \frac{a_{k-1}}{p_{P_k} a_k}\sum_{i \in B_{P_k}}\p{\by_{k-1, i}\grad f_{i}(\x_{k-1}) - y_{k-2, i} \hgb_{k-2, i}}.\label{eq:sep:primal_grad_est}
\end{align}
Notice that we do not need to use the table $\hyb_k$ introduced in~\eqref{eq:primal_grad_est} from \Cref{sec:nonsep:stochastic}, as the matrix-vector product $\hJb_{k-1}^\top \y_{k-1}$ above can be maintained in $O(bd)$ on average because only $b = n/N$ components of $\y_k$ change each iteration. When using these tables of iterates, we will encounter the familiar terms $\hTP_{k, I} = (A_k\mu + \mu_0)\breg{\X}{\u}{\hxb_{k, I}}$ from~\eqref{eq:stoch:telescoping_terms} and $\hCP_{k, I} = (A_{k-1}\mu + \mu_0)\breg{\X}{\x_k}{\hxb_{k-1, I}}$ from~\eqref{eq:stoch:cancellation_terms} for $I = 1, \ldots, N$. 

With these elements in hand, we produce the following lower bound, which follows from identical steps to \Cref{lem:lower} and so has its proof omitted. Recall the use of probability weights $\gamma_1, \ldots, \gamma_N$ in the update~\eqref{eq:prim_update}.
\begin{lemma}\label[lemma]{lem:sep:lower}
    For any $k \geq 1$, let $\g_{k-1} \in \R^d$ and $\wP_k \in [0, 1)$, consider the update
    \begin{align}
        \x_k = \argmin_{\x \in \X} &\Big\{a_k \ip{\g_{k-1}, \x} + a_k\phi(\x) + \tfrac{1-\wP_k}{2} (A_{k-1}\mu + \mu_0)\breg{\X}{\x}{\x_{k-1}} \notag\\
        &\quad + \tfrac{\wP_k}{2}(A_{k-1}\mu + \mu_0)  \sum_{I=1}^N \gamma_I \breg{\X}{\x}{\hxb_{k-1, I}}\Big\}. \label{eq:sep:primal_update}
    \end{align}
    For $k \geq 1$, it holds that
    \begin{align}
        a_k\Lcal(\u, \byb_k)
        &\geq a_k \mc{L}(\x_k, \byb_k)  + a_k \ip{\grad \fb(\x_k)^\top \byb_k - \g_{k-1}, \u - \x_k} \label{eq:sep:inner_prod_primal}\\
        &\quad + \p{\tfrac{1-\wP_k}{2}\TP_k - \tfrac{1-\wP_{k-1}}{2}\TP_{k-1}} + \tfrac{\wP_k}{2}\p{\TP_k - \textstyle\sum_{I=1}^N \gamma_I\hTP_{k-1, I}}\\
        &\quad + \tfrac{1-\wP_k}{2}\CP_k + \tfrac{\wP_k}{2}\textstyle\sum_{I=1}^N \gamma_I\hCP_{k, I} + \frac{a_k \mu}{2} \breg{\X}{\u}{\x_k}. \label{eq:sep:telescope_primal}
    \end{align}
\end{lemma}
The only difference between \Cref{lem:sep:lower} and \Cref{lem:lower} is the replacement of $\y_k$ by $\byb_k$, which is also an element of $\Y$. On the dual side, we set $\bfb_{k-1/2} = \fb(\x_k)$, so the inner product vanishes in expectation (in other words, $\ID_k$ can be thought of as zero). 
This does not come at an additional computational cost, as the dual update~\eqref{eq:dual_update_sep} only depends on a single block of the coordinate of $\bfb_{k-1/2}$. 
Following the same argument used to produce $\EP_k$ in~\eqref{eq:stoch:error_primal}, we may now build the identity card.
\begin{mdframed}[userdefinedwidth=\linewidth, align=center, linewidth=0.3mm]
\centering
\begin{tabular}{l}
     \textbf{Identity Card 3:} Stochastic update method for separable objectives\\
     \toprule
     $\sub_1$: $\hJb_{k-1}^\top \y_{k-1} + \frac{N a_{k-1}}{a_k}\sum_{i \in B_{P_k}}\p{\by_{k-1, i}\grad f_{i}(\x_{k-1}) - y_{k-2, i} \hgb_{k-2, i}}$\\
     $\sub_2$: $\bfb_{k-1/2} = \fb(\x_k)$\\
     $\sub_3$: Update $\hxb_{k, I}$ for all $I$ using~\eqref{eq:table_elements} and $(\hJb_k, \hfb_k)$ using~\eqref{eq:first_order_table}.\\
     \midrule
     \emph{Primal error:} $\EP_k = a_{k-1} \ip{\tfrac{1}{p_{P_k}}\textstyle\sum_{i \in B_{P_k}} (\by_{k-1, i}\grad f_{i}(\x_{k-1}) - y_{k-2, i} \hgb_{k-2, i}), \x_{k-1} - \x_{k}}$\\
     \emph{Dual error:} $\ED_k = 0$\\
\end{tabular}
\end{mdframed}
Using \Cref{lem:sep:upper} and \Cref{lem:sep:lower}, we may produce a version of \Cref{claim:gap}:
\begin{align}
    &\sum_{k=1}^t a_k\Ex[\gap^{\u, \v}(\x_k, \y_k)] \leq \tfrac{1-\wP_k}{2}\p{\TP_{0} - \Ex[\TP_{t}]} + \tfrac{1}{2}\p{\TD_{0} - \Ex[\TD_{t}]} \notag \\
    &\quad + \sum_{k=1}^{t} \Ex\sbr{\IP_k} - \underbrace{\Ex\sbr{\tfrac{1-\wP_k}{2}\CP_k + \tfrac{1}{4}\CD_k + \tfrac{\wP_k}{2}\textstyle\sum_J \gamma_I \hCP_{k, J}}}_{\text{cancellation terms from~\Cref{lem:sep:upper} and \Cref{lem:sep:lower}}} \label{eq:sep:gap2_cancellation} \\
    &\quad + \underbrace{\tfrac{1}{2}\sum_{k=1}^t \Ex\sbr{\wP_k\p{\textstyle\sum_{I=1}^N \gamma_I \hTP_{k-1, I} - \TP_{k}}- a_k\mu\breg{\X}{\u}{\x_k}}}_{\text{primal table terms from \Cref{lem:sep:lower}}}. \label{eq:sep:gap3_primal}
\end{align}
We proceed to the individual analyses to cancel the lines~\eqref{eq:sep:gap2_cancellation} and~\eqref{eq:sep:gap3_primal}.

\subsection{Strategy 1: Non-Uniform Historical Regularization}\label{sec:sep:historic}
\begin{algorithm}[t]
\setstretch{1.2}
   \caption{Historical Regularization Version of \Cref{algo:drago_v2} (Separable)}
   \label{algo:sep:historic}
    \begin{algorithmic}[1]
   \State {\bfseries Input:} Initial point $(\x_0, \y_0)$, averaging sequence $(a_k)_{k=0}^t$, non-negative weights $(\gamma_I)_{I=1}^N$ that sum to one, balancing sequence $(\wP_k)_{k=1}^t$, sampling vector $\pb$.
   \State Initialize the comparison points $\hxb_{0, I} = \x_0$ for all $I \in [N]$. Initialize the $\hJb_0 = \grad \fb(\x_0)$ and compute the product $\hJb_{0}^\top \y_0$.
   \State Letting $\varphi$ be the generator of $\breg{\X}{\cdot}{\cdot}$, compute $\br{\grad \varphi(\hxb_{0, I})}_{I=1}^N$ and $\sum_I \gamma_I \grad \varphi(\hxb_{0, I})$.
   \State Draw the terminal iteration $\hat{t}$ according to $(a_1/A_t, \ldots, a_t/A_t)$.
   \For {$k=1$ {\bfseries to} $t$}
       \State $\sub_1$: Compute $\g_{k-1}$ using $\pb$ and~\eqref{eq:sep:primal_grad_est}.
       \State Perform the primal update
       \begin{align}
       \x_k &= \argmin_{\x \in \X} \Big\{a_k \ip{\g_{k-1}, \x} + a_k\phi(\x) + \tfrac{A_{k-1}\mu + \mu_0}{2} \varphi(\x)\notag \\
       &\quad - \tfrac{A_{k-1}\mu + \mu_0}{2}\ip{(1-\wP_k) \grad \varphi(\x_{k-1}) + \wP_k\textstyle\sum_{I=1}^N \gamma_I \grad \varphi(\hxb_{k-1, I}), \x}\Big\}.\notag
       \end{align}
       \State $\sub_2$: Sample $Q_k$ uniformly and compute $\bfb_{k-1/2, Q_k} = f_{Q_k}(\x_k)$.
       \State If $k = \hat{t}$, compute $\byb_k$ according to~\eqref{eq:full_dual_update_sep} and return $(\x_k, \byb_k)$. 
       Otherwise, compute $\y_{k, Q_k}$ via the update in \Cref{lem:sep:upper}.
       Set $\y_{k, J} = \y_{k-1, J}$ for $J \neq Q_k$.
       \State $\sub_3$: Draw $R_k$ uniformly on $[N]$, and update $(\hxb_{k, I})_{I=1}^N$ via~\eqref{eq:table_elements}. Update the product $\hJb_{k}^\top \y_k$, and update $\br{\grad \varphi(\hxb_{k, I})}_{I=1}^N$ and $\sum_I \gamma_I \grad \varphi(\hxb_{k, I})$ via~\eqref{eq:breg_replacement}.
   \EndFor
\end{algorithmic}

\end{algorithm}

\edit{The full algorithm variant is given in \Cref{algo:sep:historic}. Regarding per-iteration complexity, the exact arguments of \Cref{sec:stochastic:historic} apply, except that the update defining $\y_k$ occurs at cost $\tilde{O}(n/N)$ instead of $\tilde{O}(n)$. Thus, the total per-iteration complexity is $\tilde{O}(nd/N)$.}

As for the analysis, we do not assume that $\wP_k = 0$, and so will cancel both the lines~\eqref{eq:sep:gap2_cancellation} and~\eqref{eq:sep:gap3_primal}.
The resulting complexity will depend on the sampling probabilities $\pb$ and regularization weights $\bgam$ through the following constants:
\begin{align*}
    \Gb_{\pb} := \sqrt{\max_{I \in [N]}\frac{\Gb_I^2}{p_I}} \text{ and } \Lb_{\pb, \bgam} := \sqrt{\max_{I\in[N]} \frac{\Lb_I^2}{p_I \gamma_I}}.
\end{align*}
\begin{thm}\label{thm:sep:historic}
    Under \Cref{asm:smoothness} and \Cref{asm:str_cvx}, consider any $\u \in \X$, $\v \in \Y$ and precision $\epsilon > 0$. 
    There exists a choice of the sequence $(a_k)_{k=1}^t$ and parameters $(\wP_k)_{k=0}^t$ 
    such that \Cref{algo:sep:historic} produces an output point $(\tx_t, \ty_t) \in \X \times \Y$ satisfying $\E{}{\gap^{\u, \v}(\tx_t, \ty_t)} \leq \epsilon$ for $t$ that depends on $\epsilon$ according to the following iteration complexities. \edit{In the following, $a_k$, $\wP_k$, and the iteration complexity are stated in big-$O$ terms, with $\alpha \sim \min\br{{\sqrt{\mu\nu}}/{\Gb_{\pb}}, {\mu}/({\sqrt{N}\Lb_{\pb, \bgam})}, 1/N}$.}
    \begin{center}
    \begin{adjustbox}{max width=\linewidth}
    \begin{tabular}{cccc}
    \toprule
        {\bf Case} & $a_k$ & $\wP_k$ & {\bf Iteration Complexity}\\
        \midrule
        $\mu > 0$, $\nu > 0$
        & \edit{$(1+\alpha)^k$} & \edit{$1/N$}
        &
        $\p{N + \frac{\sqrt{N}\Lb_{\pb, \bgam}}{\mu} + \frac{\sqrt{N}\Gb_{\pb}}{\sqrt{\mu\nu}}}\ln\p{\frac{1}{\epsilon}}$
        \\
        $\mu > 0$, $\nu = 0$
        & \edit{$k$} & \edit{$1/N$}
        &
        $\p{N + \frac{\sqrt{N}\Lb_{\pb, \bgam}}{\mu}}\ln\p{\frac{1}{\epsilon}} + \sqrt{N}\Gb_{\pb}\sqrt{\frac{\sqrt{\mu_0/\nu_0}D_{0} + (a_1\mu/\nu_0) \breg{\X}{\u}{\x_0}}{\mu \epsilon}}$
        \\
        $\mu = 0$, $\nu > 0$
        & \edit{$1$} & \edit{$1/N$}
        &
        $\frac{\sqrt{N}\Lb_{\pb, \bgam} \sqrt{\nu_0/\mu_0}D_{0}}{\epsilon} + \sqrt{N}\Gb_{\pb}\sqrt{\frac{\sqrt{\nu_0/\mu_0} D_{0}}{\nu \epsilon}}$
        \\
        $\mu = 0$, $\nu = 0$
        & \edit{$1$} & \edit{$1/N$}
        &
        $\frac{\sqrt{N}\p{\Lb_{\pb, \bgam}\sqrt{\nu_0/\mu_0}  + \Gb_{\pb}}D_{0}}{\epsilon}$\\
    \bottomrule
    \end{tabular}
    \end{adjustbox}
    \end{center}
\end{thm}

We recall from \Cref{sec:preliminaries} that $\lambb = (\lamb_1, \ldots, \lamb_N)$ for $\lamb_I := \sqrt{\Gb_I^2 + \Lb_I^2}$. The non-uniform sampling complexity given below follows by letting $p_I \propto \lamb_I$ and $\gamma_I \propto \Lb_I$.
\begin{center}
\begin{adjustbox}{max width=\linewidth}
\begin{tabular}{c|cc}
\toprule
    {\bf Constant} & {\bf Uniform Sampling} & {\bf Non-Uniform Sampling}\\
    \midrule
    $\Gb_{\pb}$ & $\sqrt{N}\norm{\Gbb}_\infty$  & $\norm{\lambb}_1^{1/2}\norm{\Gbb}_\infty^{1/2}$ \\
    $\Lb_{\pb, \bgam}$ & $N\norm{\Lbb}_\infty$ & $\norm{\lambb}_1^{1/2}\norm{\Lbb}_1^{1/2}$\\
    \bottomrule
\end{tabular}
\end{adjustbox}
\end{center}
There are strict advantages both in the dependence on $\Gbb$ and $\Lbb$ in the separable case over, say, the complexities given in \Cref{thm:stochastic:historic}. This primarily results from the freedom to select the constants $\gamma_1, \ldots, \gamma_N$ based only on the smoothness constants $\Lbb$. Thus, in the case of non-uniform sampling, the dependencies from \Cref{thm:stochastic:historic} (using the same historical regularization strategy) reduce in \Cref{thm:sep:historic} from $\norm{\lambb}_1^{1/2}\norm{\Gbb}_1^{1/2}$ to $\norm{\lambb}_1^{1/2}\norm{\Gbb}_\infty^{1/2}$ and from $\norm{\lambb}_1$ to $\norm{\lambb}_1^{1/2}\norm{\Lbb}_1^{1/2}$.

The associated \Cref{prop:sep:historic} follows very similar steps to the proof of \Cref{prop:stochastic:historic}.
\begin{prop}\label[proposition]{prop:sep:historic}
    Let $(\x_0, \y_0) \in \ri(\dom(\phi)) \times \ri(\dom(\psi))$ and $\br{(\x_k, \y_k)}_{k\geq 1}$ be generated by the update from \Cref{lem:sep:lower} with non-decreasing sequence $\wP_k$ and \Cref{lem:sep:upper}, with $\g_{k-1}$ and $\bfb_{k-1/2}$ given by~\eqref{eq:sep:primal_grad_est}, and $\fb(\x_k)$, respectively. Define $a_1 = \min\br{\tfrac{\sqrt{(1-\wP_0)\mu_0\nu_0}}{4\sqrt{N}\Gb_{\pb}}, \tfrac{\sqrt{\wP_0(1-\wP_0)}\mu_0}{4\Lb_{\pb, \bgam}}}$ and select $(a_k)_{k \geq 2}$ such that the conditions
    \begin{align}
        a_{k} &\leq \min\br{\tfrac{\sqrt{(1-\wP_k)(A_{k}\mu + \mu_0)(A_{k-1}\nu + \nu_0)}}{4\sqrt{N}\Gb_{\pb}}, \tfrac{\sqrt{\wP_k(1-\wP_k)(A_{k}\mu + \mu_0)(A_{k-1}\mu + \mu_0)}}{4\Lb_{\pb, \bgam}}}, \label{eq:sep:cond1}
    \end{align}
    are satisfied. In addition, impose that for any $\ell = 1, \ldots, t-1$, it holds that
    \begin{align}
        \frac{\mu}{N}\sum_{k'=0}^{\infty}  \wP_{\ell + k' + 1} A_{\ell + k'}  (1-1/N)^{k'} &\leq  (\wP_\ell A_\ell + a_\ell)\mu.\label{eq:sep:primal:decay}
    \end{align}
    We have that for any $\u \in \X$ and $\v \in \Y$, 
    \begin{align*}
        \sum_{k=1}^t a_k \Ex[\gap^{\u, \v}(\x_k, \byb_k)] + \tfrac{1-\wP_k}{4}\Ex[\TP_t] + \tfrac{1}{4}\Ex[\TD_t]
        &\leq C_{\X}\breg{\X}{\u}{\x_0} + \tfrac{1}{2}\TD_{0}.
    \end{align*}
    for $C_{\X} := \tfrac{1}{2}\p{\mu_0 + \textstyle\sum_{\ell=0}^{\infty} \wP_{\ell+1} (A_{\ell}\mu + \mu_0)  (1-1/N)^{\ell}}$.
\end{prop}

\edit{Because \Cref{thm:sep:historic} mirrors the logic of \Cref{thm:stochastic:historic}, the proof is omitted. }

\subsection{Strategy 2: Non-Uniform Block Replacement Probabilities }\label{sec:sep:minty}
\begin{algorithm}[t]
\setstretch{1.2}
   \caption{Non-Uniform Block Replacement Version of \Cref{algo:drago_v2} (Separable)}
   \label{algo:sep:minty}
    \begin{algorithmic}[1]
   \State {\bfseries Input:} Initial point $(\x_0, \y_0)$, averaging sequence $(a_k)_{k=0}^t$, sampling vectors $(\pb, \rb)$.
   \State Initialize the comparison points $\hxb_{0, I} = \x_0$ for all $I \in [N]$. Initialize the $\hJb_0 = \grad \fb(\x_0)$ and compute the product $\hJb_{0}^\top \y_0$.
   \State Draw the terminal iteration $\hat{t}$ according to $(a_1/A_t, \ldots, a_t/A_t)$.
   \For {$k=1$ {\bfseries to} $t$}
       \State $\sub_1$: Compute $\g_{k-1}$ using $\pb$ and~\eqref{eq:sep:primal_grad_est}.
       \State Perform the primal update:
       \begin{align}
           \x_k &= \argmin_{\x \in \X} \ a_k \ip{\g_{k-1}, \x} + a_k\phi(\x) + \tfrac{A_{k-1}\mu + \mu_0}{2}\breg{\X}{\x}{\x_{k-1}}.\notag
       \end{align}
       \State $\sub_2$: Sample $Q_k$ uniformly and compute $\bfb_{k-1/2, Q_k} = f_{Q_k}(\x_k)$.
       \State If $k = \hat{t}$, compute $\byb_k$ according to~\eqref{eq:full_dual_update_sep} and return $(\x_k, \byb_k)$. 
       Otherwise, compute $\y_{k, Q_k}$ via the update in \Cref{lem:sep:upper}.
       Set $\y_{k, J} = \y_{k-1, J}$ for $J \neq Q_k$.
       \State $\sub_3$: Draw $R_k \sim \rb$ and update $(\hxb_{k, I})_{I=1}^N$ via~\eqref{eq:table_elements}. Update the product $\hJb_{k}^\top \hyb_k$ via~\eqref{eq:g_replacement}.
   \EndFor
\end{algorithmic}

\end{algorithm}

\edit{See \Cref{algo:sep:minty} for the full algorithm variant. For its per-iteration complexity, the exact arguments of \Cref{sec:stochastic:minty} apply, except that the update defining $\y_k$ may now occur at cost $\tilde{O}(n/N)$ instead of $\tilde{O}(n)$. Thus, the total per-iteration complexity is $\tilde{O}(nd/N)$.}

Like its counterpart \Cref{sec:stochastic:minty}, this section will rely on choosing the parameter $\rb = (r_1, \ldots, r_N)$, which governs the update probabilities of the primal table $\hxb_{k, 1}, \ldots, \hxb_{k, N}$. Accordingly, we may set $\wP_k = 0$ to simplify our gap bound to
\begin{align}
    &\sum_{k=1}^t a_k\Ex[\gap^{\u, \v}(\x_k, \y_k)] \leq \tfrac{1}{2}\p{\TP_{0} - \Ex[\TP_{t}]} + \tfrac{1}{2}\p{\TD_{0} - \Ex[\TD_{t}]} \notag \\
    &\quad + \sum_{k=1}^{t-1} \Ex\sbr{\IP_k} - \sum_{k=1}^{t} \Ex\sbr{\tfrac{1}{2}\CP_k + \tfrac{1}{4}\CD_k}. \label{eq:sep:gap:simplified}
\end{align}
By this point, all arguments used in the analysis have been seen before, in that ideas related to separability were employed in \Cref{sec:sep:historic} and ideas related to block replacement probabilities were employed in \Cref{sec:stochastic:minty}. Thus, the proofs are relatively short in this section.
We cancel~\eqref{eq:sep:gap:simplified} in \Cref{prop:sep:minty}. The resulting complexity will depend on the sampling probabilities $\pb$ and $\rb$ through the constants
\begin{align*}
    \Gb_{\pb} := \sqrt{\max_{I \in [N]}\frac{\Gb_I^2}{p_I}} \text{ and } \Lb_{\pb, \rb} := \sqrt{\sum_{I=1}^N \frac{\Lb_I^2}{p_Ir_I^2}},
\end{align*}
employed in \Cref{prop:sep:minty}.

As in \Cref{sec:sep:historic}, because the following result follows the same argument as \Cref{thm:stochastic:minty}, the proof is omitted. 
We use initial distance quantity $D_0$ from \Cref{thm:nonsep:full}.
\begin{thm}\label{thm:sep:minty}
    Under \Cref{asm:smoothness} and \Cref{asm:str_cvx}, consider any $\u \in \X$, $\v \in \Y$ and precision $\epsilon > 0$.  
    There exists a choice of the sequences $(a_k)_{k=1}^t$ and $(\wP_k)_{k=1}^t$ such that \Cref{algo:sep:minty} produces an output point $(\tx_t, \ty_t) \in \X \times \Y$ satisfying $\E{}{\gap^{\u, \v}(\tx_t, \ty_t)} \leq \epsilon$ for $t$ that depends on $\epsilon$ according to the following iteration complexities. 
    \edit{In the following, $a_k$, and the iteration complexity are stated in big-$O$ terms, with $\alpha \sim \min\br{{\sqrt{\mu\nu}}/{\Gb_{\pb}}, {\mu}/({\Lb_{\pb, \rb})}, 1/N}$.}
    \begin{center}
    \begin{adjustbox}{max width=\linewidth}
    \begin{tabular}{ccc}
    \toprule
        {\bf Case} & $a_k$ & {\bf Iteration Complexity}\\
        \midrule
        $\mu > 0$, $\nu > 0$
        & $(1+\alpha)^k$
        &
        $O\p{N + \frac{\Lb_{\pb, \rb}}{\mu} + \frac{\sqrt{N}\Gb_{\pb}}{\sqrt{\mu\nu}}}\ln\p{\frac{1}{\epsilon}}$
        \\
        $\mu > 0$, $\nu = 0$
        & $k$
        &
        $\p{N + \frac{\Lb_{\pb, \rb}}{\mu}}\ln\p{\frac{1}{\epsilon}} + \sqrt{N}\Gb_{\pb}\sqrt{\frac{\sqrt{\mu_0/\nu_0}D_{0}}{\mu \epsilon}}$
        \\
        $\mu = 0$, $\nu > 0$
        & $1$
        &
         $N\ln \p{\frac{1}{\epsilon}} + \frac{\Lb_{\pb, \rb}\sqrt{\nu_0/\mu_0}D_{0}}{\epsilon} + \sqrt{N}\Gb_{\pb}\sqrt{\frac{\sqrt{\nu_0/\mu_0}D_{0}}{\nu \epsilon}}$
        \\
        $\mu = 0$, $\nu = 0$
        & $1$
        &
        $N\ln \p{\frac{1}{\epsilon}} + \frac{\p{\Lb_{\pb, \rb}\sqrt{\nu_0/\mu_0}  + \sqrt{N}\Gb_{\pb}}D_{0}}{\epsilon}$\\
    \bottomrule
    \end{tabular}
    \end{adjustbox}
    \end{center}
\end{thm}
The non-uniform sampling complexity given below follows by letting $p_I \propto \lamb_I^{1/2}$ and $r_I \propto \Lb_I^{1/2}$.
\begin{center}
\begin{adjustbox}{max width=\linewidth}
\begin{tabular}{c|cc}
\toprule
    {\bf Constant} & {\bf Uniform Sampling} & {\bf Non-Uniform Sampling}\\
    \midrule
    $\Gb_{\pb}$ & $\sqrt{N}\norm{\Gbb}_\infty$  & $\norm{\lambb}_{1/2}^{1/4} \norm{\Gbb}_\infty^{3/4}$\\
    $\Lb_{\pb, \rb}$ & $N^{3/2}\norm{\Lbb}_2$ & $\norm{\lambb}_{1/2}^{1/4}  \norm{\Lbb}_{1/2}^{3/4}$\\
    \bottomrule
\end{tabular}
\end{adjustbox}
\end{center}
We compare the resulting complexity to the non-separable analog in \Cref{thm:stochastic:minty}. There are improvements both in the dependence on $\Gbb$ and $\Lbb$ when the sampling scheme can be tuned. \Cref{thm:sep:minty} improves the dependence on $\Gbb$ from $\norm{\lambb}_{1/2}^{1/2} \norm{\Gbb}_\infty^{1/4}$ to $\norm{\lambb}_{1/2}^{1/4} \norm{\Gbb}_\infty^{3/4}$. As for the dependence on $\Lbb$, this improves from $\norm{\lambb}_{1/2}^{3/4}  \norm{\Lbb}_{1/2}^{1/4}$ to $\norm{\lambb}_{1/2}^{1/4}  \norm{\Lbb}_{1/2}^{3/4}$. The mechanism is analogous to the improvement from \Cref{thm:stochastic:historic} to \Cref{thm:sep:historic}; because $\ID_k = 0$, the constants do not have to adapt in order to control an error term of the form $\ED_k$.

To justify \Cref{thm:sep:minty}, we provide the associated \Cref{prop:sep:minty} to establish the growth conditions on $(a_k)_{k\geq 1}$.
\begin{prop}\label[proposition]{prop:sep:minty}
    Let $(\x_0, \y_0) \in \ri(\dom(\phi)) \times \ri(\dom(\psi))$ and $\br{(\x_k, \y_k)}_{k\geq 1}$ be generated using $\g_{k-1}$ and $\bfb_{k-1/2}$ given by~\eqref{eq:sep:primal_grad_est} and $\fb(\x_k)$, respectively. Define $a_1 = \min\br{\frac{\sqrt{\mu_0\nu_0}}{4\sqrt{2}\sqrt{N}\Gb_{\pb}}, \frac{\mu_0}{30\Lb_{\pb, \rb}}}$ and select $(a_k)_{k \geq 2}$ such that the conditions
    \begin{align}
        a_{k} \leq \min\br{\frac{\sqrt{(A_k\mu + \mu_0)(A_{k-1}\nu + \nu_0)}}{4\sqrt{2}\sqrt{N}\Gb_{\pb}}, \frac{\sqrt{(A_k\mu + \mu_0)(A_{k-1}\mu + \mu_0)}}{30\Lb_{\pb, \rb}}}\label{eq:sep:rate}
    \end{align}
    and $\frac{a_{k}^2}{A_k\mu + \mu_0} \leq \textstyle\min_I (1 + r_I/5)\frac{a_{k-1}^2}{A_{k-1}\mu + \mu_0}$ hold.
    We have that for any $\u \in \X$ and $\v \in \Y$, 
    \begin{align*}
        \sum_{k=1}^t a_k \Ex[\gap^{\u, \v}(\x_k, \y_k)] + \tfrac{1}{4}\Ex[\TP_t] + \tfrac{1}{4}\Ex[\TD_t]
        &\leq \tfrac{1}{2}\TP_{0} + \tfrac{1}{2}\TD_{0}.
    \end{align*}
\end{prop}
\begin{proof}
    Mirroring the proof of \Cref{prop:stochastic:minty}, by Young's inequality with parameter $(A_{k-1}\mu + \mu_0)/4$, we have for $k = 2, \ldots, t-1$ that
    \begin{align}
        \Ex[\EP_k] &= a_{k-1} \Ex\ip{\tfrac{1}{p_{P_k}}\textstyle\sum_{i \in B_{P_k}} (\by_{k-1, i}\grad f_{i}(\x_{k-1}) - y_{k-2, i} \hgb_{k-2, i}),  \x_{k-1} - \x_{k}} \notag \\
        &\leq \frac{1}{4} \Ex[\CP_k]+ \frac{2a_{k-1}^2}{A_{k-1}\mu + \mu_0} \Ex\norm{\tfrac{1}{p_{P_k}}\textstyle\sum_{i \in B_{P_k}} (\by_{k-1, i}\grad f_{i}(\x_{k-1}) - y_{k-2, i} \hgb_{k-2, i})}_{\X^*}^2 \label{eq:sep:minty:primal:middle}\\
    \end{align}
    and
    \begin{align}
        &-a_t \Ex[\ip{\grad \fb(\x_t)^\top \byb_t - \hJb_{t-1}^\top \y_{t-1}, \u - \x_t}] \notag\\
        &\leq \frac{1}{4}\Ex[\TP_{t}] + \frac{2a_{t-1}^2}{A_{t-1}\mu + \mu_0} \Ex\norm{\tfrac{1}{p_{P_{t+1}}}\textstyle\sum_{i \in B_{P_{t+1}}} (\by_{t-1, i}\grad f_{i}(\x_{t-1}) - y_{t-2, i} \hgb_{t-2, i})}_{\X^*}^2. \label{eq:sep:minty:primal:last}
    \end{align}    
    Apply \Cref{lem:separation} to achieve
    \begin{align*}
         &\Ex\norm{\tfrac{1}{p_{P_k}}\textstyle\sum_{i \in B_{P_k}} (\by_{k-1, i}\grad f_{i}(\x_{k-1}) - y_{k-2, i} \hgb_{k-2, i})}_{\X^*}^2\\
         &\leq 2\sum_{I=1}^N\frac{\Gb_I^2}{p_I}\Ex\norm{\byb_{k-1, I} - \y_{k-2,I}}_{2}^2 + 2\sum_{I=1}^N\frac{\Lb_I^2}{p_I}\Ex\norm{\x_{k-1} - \hxb_{k-2, I}}_{\X}^2,
    \end{align*}
    where the last line follows from Young's inequality. For the first term, we compute its expectation using \Cref{lem:expect}, %
    \begin{align*}
        2\textstyle\sum_{I=1}^N\frac{\Gb_I^2}{p_I}\Ex\norm{\byb_{k-1, I} - \y_{k-2,I}}_{2}^2 &\leq2 \textstyle\max_{I}\frac{\Gb_I^2}{p_I}\Ex\norm{\byb_{k-1} - \y_{k-2}}_{2}^2\\
        &\leq 2\textstyle\max_{I}\frac{\Gb_I^2}{p_I}\Ex\norm{\byb_{k-1} - \y_{k-2}}_{\Y}^2\\
        &= 2N\textstyle\max_{I}\frac{\Gb_I^2}{p_I}\Ex\norm{\y_{k-1} - \y_{k-2}}_{\Y}^2\\
        &\leq 4N\underbrace{\textstyle\max_{I}\frac{\Gb_I^2}{p_I}}_{\Gb_{\pb}^2}\Ex[\breg{\Y}{\y_{k-1}}{\y_{k-2}}],
    \end{align*}
    where we used  $\norm{\cdot}_2 \leq \norm{\cdot}_\Y$ in the second inequality. For the second term, use the same argument leading to~\eqref{eq:stoch:ep2} in the proof of \Cref{prop:stochastic:minty} to achieve
    \begin{align}
        2\sum_{I=1}^N\frac{\Lb_I^2}{p_I}\Ex\norm{\x_{k-1} - \hxb_{k-2, I}}_{\X}^2  &\leq 50\p{\sum_{I=1}^N \frac{\Lb_I^2}{p_Ir_I^2}} \sum_{k'=1}^{t} \frac{a_{k'}^2}{A_{k'}\mu + \mu_0}  \Ex\norm{\x_{k'} - \x_{k'-1}}_\X^2 \notag\\
        &\leq 100\underbrace{\p{\sum_{I=1}^N \frac{\Lb_I^2}{p_Ir_I^2}}}_{\Lb_{\pb, \rb}^2} \sum_{k'=1}^{t} \frac{a_{k'}^2}{A_{k'}\mu + \mu_0}  \Ex[\breg{\X}{\x_{k'}}{\x_{k'-1}}]. \notag
    \end{align}
    Thus, under the conditions~\eqref{eq:sep:rate}, it holds that
    \begin{align*}
        \sum_{k=1}^{t-1} \Ex\sbr{\IP_k} \leq \tfrac{1}{4}\sum_{k=1}^{t} \Ex\sbr{\CP_k + \CD_k},
    \end{align*}
    which completes the proof.
\end{proof}

\section{Discussion}\label{sec:discussion}
Our discussion covers internal comparisons between full vector methods and stochastic methods, as well as external comparisons to contemporary methods for solving saddle point and variational inequality problems. We briefly comment on our chosen convergence criterion, as it may differ slightly from those used in comparisons.

\subsection{Stronger Convergence Criteria}\label{sec:discussion:sup_exp_vs_exp_sup}
The results discussed in this section are expressed in terms of \emph{global complexity} or \emph{arithmetic complexity}, which is computed by multiplying the number of iterations shown in \Crefrange{thm:nonsep:full}{thm:sep:minty} by $\tilde{O}(nd)$ for full vector update methods;
for block-wise methods with $N$ blocks of size $n/N$, we use $\tilde{O}(n(d/N + 1))$ for full updates of $\y_k$ and $\tilde{O}(n(d/N))$ for partial updates of $\y_k$. In order to effectively compare to methods using our block coordinate-wise Lipschitz and smoothness constant from \Cref{sec:preliminaries}, we will apply a particular finite sum decomposition (see~\eqref{eq:per_oracle_complexity}) that will allow us to directly apply methods from the finite sum variational inequality literature.

Importantly, in the case of randomized algorithms, the complexity in terms of the number of iterations is itself determined by the number $t$ such that the algorithm may output a point $(\tx_t, \ty_t)$ satisfying $\Ex[\gap^{\u, \v}(\tx_t, \ty_t)] \leq \epsilon$, where the expectation is taken over all algorithm randomness with $\u \in \X$ and $\v \in \Y$ fixed (i.e., $(\u, \v)$ is independent of the algorithm randomness). When $\mu > 0$, the criterion is made meaningful by setting $\u = \x_\star$ as the unique minimizer of the strongly convex objective $\x \mapsto \max_{\y \in \Y} \Lcal(\x, \y)$. Otherwise, we choose a compact set $\mc{U} \sse \X$ and consider $\sup_{\u \in \X} \Ex[\gap^{\u, \v}(\tx_t, \ty_t)] \leq \epsilon$, where $\v$ is replaced by the unique maximizer $\y_\star$ when $\nu > 0$ or another supremum is taken over $\v \in \mc{V} \sse \Y$ for $\mc{V}$ compact otherwise. As described in \citet[Example 1]{alacaoglu2022complexity}, the ``supremum of expected gap'' criterion is weaker than the ``expected supremum of gap'' criterion, as algorithms with divergent behavior can still converge according to the first criterion. We render guarantees for the stronger criterion as a technical detail, as in light of previous work, largely similar steps can be applied to achieve the same complexity guarantee for the expected supremum of gap. To not overcomplicate the proofs, we only highlight the parts of the analysis that change. Consider the argument used to derive~\eqref{eq:stoch:three_term_decomp}, in which the expectation is applied to the terms $a_k \ip{\grad \fb(\x_k)^\top \y_k - \hJb_{k-1}^\top \hyb_{k-1}, \u - \x_k}$, after which they telescope. If the supremum is to be taken, we can no longer apply the expectation to these terms directly. Instead,
\begin{align}
    a_k \langle\grad \fb(\x_k)^\top \y_k - \g_{k-1}&, \u - \x_k\rangle = a_k \ip{\grad \fb(\x_k)^\top \y_k - \hJb_{k-1}^\top \hyb_{k-1}, \u - \x_k}\notag\\
    &- \langle \underbrace{\tfrac{a_{k-1}}{p_{P_k}}\textstyle\sum_{i \in B_{P_k}} (y_{k-1, i}\grad f_{i}(\x_{k-1}) - \hy_{k-2, i} \hgb_{k-2, i})}_{\z_k}, \u - \x_{k-1} \rangle\notag\\
    &- \underbrace{a_{k-1} \ip{\tfrac{1}{p_{P_k}}\textstyle\sum_{i \in B_{P_k}} (y_{k-1, i}\grad f_{i}(\x_{k-1}) - \hy_{k-2, i} \hgb_{k-2, i}), \x_{k-1} - \x_k}}_{\EP_k},\notag
\end{align}
and by taking the expectation over $P_k$ yields the identity
\begin{align}
    a_k \ipsmall{\grad \fb(\x_k)^\top \y_k - \g_{k-1}, \u - \x_k} &=  a_k \ipsmall{\grad \fb(\x_k)^\top \y_k - \hJb_{k-1}^\top \hyb_{k-1}, \u - \x_k} \notag\\
    &\quad - a_{k-1} \ipsmall{\grad \fb(\x_{k-1})^\top \y_{k-1} - \hJb_{k-2}^\top \hyb_{k-2}, \u - \x_{k-1}}\notag\\
    &\quad - \ip{\z_k - \E{k-1}{\z_k}, \x_{k-1}}\notag\\
    &\quad - \ip{\z_k - \E{k-1}{\z_k}, \u}  - \EP_k.
\end{align}
The familiar term $\EP_k$ does not depend on $\u$ and can be bounded using the same techniques as in the proofs of \Cref{thm:stochastic:historic} and \Cref{thm:stochastic:minty}, whereas $\z_k - \E{k-1}{\z_k}$ is zero-mean conditional on $\F_{k-1}$ (i.e., a martingale difference sequence), but is not necessarily independent of $\u$. Because the supremum over $\u$ does not affect the $\ip{\z_k - \E{k-1}{\z_k}, \x_{k-1}}$ term, it will cancel with the full expectation is taken.
Next, we may apply \citet[Lemma 4]{diakonikolas2025block} (adapted from \citet[Lemma 3.5]{alacaoglu2022stochastic}) to achieve
\begin{align*}
    - \E{}{\sum_{k=1}^t \ip{\z_k - \E{k-1}{\z_k}, \u_\star}} &\leq \E{}{\breg{\X}{\u_\star}{\x_0}} + \frac{1}{2}\sum_{k=1}^t \Ex\norm{\z_k - \E{k-1}{\z_k}}_2^2\\
    &\leq \E{}{\breg{\X}{\u_\star}{\x_0}} + \frac{1}{2}\sum_{k=1}^t \Ex\norm{\z_k}_2^2,
\end{align*}
where $\u_\star$ is the element of $\u \in \mc{U}$ that achieves the supremum in the gap criterion (recalling that $\mc{U}$ is chosen to be compact). The term $\E{}{\breg{\X}{\u_\star}{\x_0}}$ is upper bounded by a constant, whereas the $\Ex\norm{\z_k}_2^2$ terms are bounded using the exact same techniques used to bound the $\EP_k$ terms. We choose to describe the argument in the manner above (as opposed to including it formally in the proofs) as it changes neither the other technical ideas nor the resulting complexities; we comment on this subtlety for the sake of completeness.

\subsection{Full Vector Update versus Stochastic Methods}
To compare methods both within this manuscript and alternatives for solving nonbilinearly coupled min-max problems, we first state some relationships between the constants introduced in \Cref{asm:smoothness}. We then proceed with fine-grained comparisons to alternatives in the existing literature on min-max optimization and monotone variational inequalities. To simplify some comparisons, we assume that $L/\mu \geq 1$ and $G/\sqrt{\mu\nu} \geq 1$, so that they may be interpreted as ``primal'' and ``mixed'' condition numbers, respectively.
First, observe that by the triangle inequality, in terms of the constants $\Gbb = (\Gb_1, \ldots, \Gb_N)$ and $\Lbb = (\Lb_1, \ldots, \Lb_N)$, the constants $G$ and $L$ can be upper bounded as
\begin{align*}
    G \leq \norm{\Gbb}_1 \text{ and } L \leq \norm{\Lbb}_1.
\end{align*}
Comparing the complexities for $\mu, \nu > 0$ in \Cref{tab:complexity_internal}, the findings are summarized as follows.

\renewcommand{\arraystretch}{1.4}
\begin{table*}[t]
\centering
    \begin{adjustbox}{max width=\linewidth}
    \begin{tabular}{ccc}
    \toprule
        {\bf Algorithm Type} & {\bf Global Complexity (big-$O$)} \\
        \midrule
        Full vector (\Cref{thm:nonsep:full}) & $nd\p{\frac{L}{\mu} + \frac{G}{\sqrt{\mu\nu}}} \ln\p{\frac{1}{\epsilon}}$\\
        Stochastic (\Cref{thm:stochastic:historic}) & $\p{\frac{nd}{N}+n}\p{N + \frac{\sqrt{N}\norm{\lambb}_1}{\mu} + \frac{\sqrt{N}\norm{\lambb}_1^{1/2}\norm{\Gbb}_1^{1/2}}{\sqrt{\mu\nu}}}\ln\p{\frac{1}{\epsilon}}$\\
        Stochastic (\Cref{thm:stochastic:minty}) & $\p{\frac{nd}{N}+n}\p{N + \frac{\norm{\lambb}_{1/2}^{3/4} \norm{\Lbb}_{1/2}^{1/4}}{\mu} + \frac{\norm{\lambb}_{1/2}^{1/2} \norm{\Gbb}_{1/2}^{1/2}}{\sqrt{\mu\nu}}} \ln\p{\frac{1}{\epsilon}}$\\        
        Block Coordinate-wise (\Cref{thm:sep:historic}) & $\frac{nd}{N}\p{N + \frac{\sqrt{N}\norm{\lambb}_1^{1/2}\norm{\Lbb}_1^{1/2}}{\mu} + \frac{\sqrt{N}\norm{\lambb}_1^{1/2}\norm{\Gbb}_\infty^{1/2}}{\sqrt{\mu\nu}}}\ln\p{\frac{1}{\epsilon}}$\\
        Block Coordinate-wise (\Cref{thm:sep:minty}) & $\frac{nd}{N}\p{N + \frac{\norm{\lambb}_{1/2}^{1/4}  \norm{\Lbb}_{1/2}^{3/4}}{\mu} + \frac{\sqrt{N}\norm{\lambb}_{1/2}^{1/4} \norm{\Gbb}_\infty^{3/4}}{\sqrt{\mu\nu}}}\ln\p{\frac{1}{\epsilon}}$\\
        \bottomrule
    \end{tabular}
    \end{adjustbox}
    \vspace{6pt}
\caption{{\bf Complexity Bounds for Full Vector and Stochastic Methods for the Case $\mu, \nu > 0$.} Arithmetic or global complexity (i.e., the total number of elementary operations required to compute $(\x, \y)$ satisfying $\E{}{\gap^{\u, \v}(\x, \y) \leq \epsilon}$ for fixed $(\u, \v) \in \X \times \Y$, with the expectation taken over all algorithmic randomness. 
}
\label{tab:complexity_internal}
\end{table*}

\begin{remark}
    We observe the following in \Cref{tab:complexity_internal}.
    \begin{itemize}
        \item We highlight a limitation in the historical regularization method (\Cref{thm:stochastic:historic} and \Cref{thm:sep:historic}) as a factor $\sqrt{N}$ is gained in terms of the dependence on $\Gbb$ and $\Lbb$, which is not observed for the full vector method from \Cref{thm:nonsep:full}. When comparing  \Cref{thm:stochastic:historic} to \Cref{thm:nonsep:full}, the dependence on the smoothness constants renders as $L$ versus $\norm{\lambb}_1/\sqrt{N}$ and on the Lipschitz constants as $G$ versus $\norm{\lambb}_1^{1/2}\norm{\Gbb}_1^{1/2}/\sqrt{N}$. In both cases, when the constants are highly non-uniform, we are still afforded an up to $\sqrt{n}$ improvement in complexity from the stochastic method. Note that \Cref{thm:sep:historic} is a strict improvement over \Cref{thm:stochastic:historic} due to separability.
        \item On the other hand, in the highly non-uniform setting, we may gain an up to $d$ factor (resp.~$\sqrt{n}$ factor) of improvement in terms of complexity using the methods of \Cref{thm:stochastic:minty} (resp.~\Cref{thm:sep:minty}) over the full vector method. The results of \Cref{thm:stochastic:historic} and \Cref{thm:stochastic:minty} (and by analogy, the results of  \Cref{thm:sep:historic} and \Cref{thm:sep:minty}) do not have a uniformly dominating method, as the extra factor of $\sqrt{N}$ may be on par with the improvement of the $\norm{\cdot}_1$ norm over the $\norm{\cdot}_{1/2}$ in terms of the dependence on $(\Gbb, \Lbb, \lambb)$.
    \end{itemize}
\end{remark}
\edit{While focusing on the $\mu > 0, \nu > 0$ case in the main text, we also include the tabulated results for the non-strongly convex cases in \Cref{sec:a:tables}.}

\subsection{Alternative Methods for Min-Max and Variational Inequality Problems}
In some comparisons, we must access the smoothness constants of the function $(\x, \y) \mapsto \Lcal(\x, \y)$ both with respect to $\x$ and $\y$ separately, and additionally, with respect to the pair $(\x, \y)$. 
For the sake of comparison, assume that $\phi$ and $\psi$  are differentiable and 
\begin{align*}
    \norm{\grad \phi(\x) - \grad \phi(\x')}_{\X^*} \leq \mu \norm{\x - \x'}_\X \text{ and } \norm{\grad \psi(\y) - \grad \psi(\y')}_{\Y^*} \leq \nu \norm{\y - \y'}_\Y.
\end{align*}
Then, the following relations hold:
\begin{align}
    \sup_{\y \in \Y} \norm{\grad_{\x} \Lcal(\x, \y) - \grad_{\x'} \Lcal(\x', \y)}_{\X^*} &\leq (L + \mu) \norm{\x - \x'}_\X, \label{eq:smth:x} \\
    \sup_{\x \in \X} \norm{\grad_{\y} \Lcal(\x, \y) - \grad_{\y'} \Lcal(\x, \y')}_{\Y^*} &\leq \nu \norm{\y - \y'}_\Y, \notag \\
    \sup_{\y \in \Y} \norm{\grad_{\y} \Lcal(\x, \y) - \grad_{\y} \Lcal(\x', \y)}_{\Y^*} &\leq G \norm{\x - \x'}_\X, \notag \\
    \sup_{\x \in \X} \norm{\grad_{\x} \Lcal(\x, \y) - \grad_{\x} \Lcal(\x, \y')}_{\X^*} &\leq G \norm{\y - \y}_\Y. \notag
\end{align}
As is done in the case of the $\ell_2$-norm, define the norms
\begin{align*}
    \norm{(\x, \y)}^2 := \norm{\x}_\X^2 + \norm{\y}_\Y^2 \text{ and } \norm{\grad \Lcal(\x, \y)}_*^2 := \norm{\grad_{\x} \Lcal(\x, \y)}_{\X^*}^2 + \norm{\grad_{\y} \Lcal(\x, \y)}_{\Y^*}^2.
\end{align*}
Then, we finally have  that
\begin{align}
    \norm{\grad \Lcal(\x, \y) - \grad \Lcal(\x', \y')}_* \leq \sqrt{3\max\br{L^2 + G^2 + \mu^2, G^2 + \nu^2}}\norm{(\x, \y) - (\x', \y')}.\label{eq:smth:xy}
\end{align}
\begin{remark}
    We do \emph{not} assume that $\phi$ and $\psi$ are smooth in our convergence guarantees, as the individual components $\phi$ and $\psi$ are allowed to be non-differentiable in our framework. The additional assumptions above are made only for the sake of comparison. To this end, we will assume that $\nu \leq G$ (and maintain $\mu \leq L$). Then the constants in~\eqref{eq:smth:x} and~\eqref{eq:smth:xy} can be simplified to absolute constants times $L$ and $\sqrt{L^2 + G^2}$, respectively. Thus, we define the vector field $(\x, \y) \mapsto \grad \Lcal(\x, \y)$, which is $C\sqrt{L^2 + G^2}$-Lipschitz (for an absolute constant $C > 0$) and $(\mu \wedge \nu)$-strongly monotone.
\end{remark}
Given the above,  we use these constants to compare to algorithms designed for solving variational inequality problems. In the upcoming remarks, we discuss the convex-concave (monotone) and strongly convex-strongly concave (strongly monotone) settings. First, consider the case in which $\mu \wedge \nu > 0$, so that we may observe the dependence on all problem constants in the full vector update settings.
\begin{remark}
    We observe the following in \Cref{tab:complexity_full}.
    \begin{itemize}
        \item When using classical methods for monotone variational inequalities, such as dual extrapolation \citep{nesterov2006solving}, we highlight the $(\mu \wedge \nu)$ term, which may suffer when there is a large asymmetry in the strong convexity and strong concavity constants. As emphasized in \Cref{sec:intro}, $\nu$ is often chosen as a small approximation or smoothing parameter, meaning that the theoretical complexity will not necessarily improve for large values of the primal strong convexity constant $\mu$. Observe that all other methods will improve with increasing $\mu$.
        \item In recent works \citep{jin2022sharper, li2023nesterov} and ours, the Lipschitz constants from different components of the objective function are separated in the complexity. In particular, the constants $L$ and $\nu$ are decoupled, which is especially helpful in scenarios in which $L$ is much larger than $G$. Examples include losses that are themselves smooth approximations of non-smooth losses, such as Huber approximations of the mean absolute error function. We improve over \citet{jin2022sharper, li2023nesterov}, which are designed for general nonbilinearly-coupled objectives, by a logarithmic factor in iteration complexity and achieve the same result in global complexity. 
    \end{itemize}
\end{remark}

\renewcommand{\arraystretch}{1.4}
\begin{table*}[t]
\centering
    \begin{adjustbox}{max width=\linewidth}
    \begin{tabular}{cc}
    \toprule
        {\bf Reference} & {\bf Global Complexity} \\
        \midrule
        \citet[Theorem 3]{nesterov2006solving}
        &
        $O\p{nd\frac{\sqrt{L^2 + G^2}}{\mu \wedge \nu} \ln (1/\epsilon)}$
        \\
        \midrule
        \citet[Theorem 3]{wang2020improved}
        &
        $\tilde{O}\p{nd\sqrt{\frac{L}{\mu} + \frac{G \cdot \max\br{L, G}}{\mu \nu}} \ln (1/\epsilon)}$
        \\
        \midrule
        \begin{tabular}{c}
             \citet[Theorem 9]{lin2020near}\\
             \citet[Theorem 3]{borodich2024nearoptimal}
        \end{tabular}
        &
        $\tilde{O}\p{nd\sqrt{\frac{L^2 + G^2}{\mu \nu}} \ln(1/\epsilon)}$
        \\
        \midrule
        \begin{tabular}{c}
             \citet[Theorem 3]{carmon2022recapp}\\
             \citet[Theorem 4.2]{lan2023novel}
        \end{tabular}
        &
        $O\p{nd\sqrt{\frac{L^2 + G^2}{\mu \nu}} \ln(1/\epsilon)}$
        \\
        \midrule
        \citet[Theorem 3]{kovalev2022thefirst}
        &
        $O\p{nd\sqrt{\frac{L^2 + G^2}{\mu (\mu \wedge \nu)}} \ln(1/\epsilon)}$
        \\
        \midrule
        \begin{tabular}{c}
             \citet[Theorem 1]{jin2022sharper} \\
             \citet[Corollary 3.4]{li2023nesterov}
        \end{tabular}
        &
        $\tilde{O}\p{nd\p{\frac{L}{\mu} + \frac{G}{\sqrt{\mu \nu}}}\ln(1/\epsilon)}$
        \\
        \midrule
        {\bf This work (\Cref{thm:nonsep:full})}
        &
        $O\p{nd\p{\frac{L}{\mu} + \frac{G}{\sqrt{\mu \nu}}}\ln(1/\epsilon)}$
        \\
        \bottomrule
    \end{tabular}
    \end{adjustbox}
    \vspace{6pt}
\caption{{\bf Complexity Bounds for \blue{General Nonbilinarly-Coupled} Objectives for $\mu, \nu > 0$. } Runtime or global complexity (i.e.~the total number of elementary operations required to compute $(\x, \y)$ satisfying $\gap^{\u, \v}(\x, \y) \leq \epsilon$ for fixed $(\u, \v) \in \X \times \Y$. The methods considered call the entire list of primal first-order oracles $(f_j, \grad f_j)$ for $j = 1, \ldots, n$ on each iteration. The method of \citet[Corollary 1]{kovalev2022thefirst} achieves its claim by swapping the role of $\x$ and $\y$, which is not possible for~\eqref{eq:semilinear}. \citet[Theorem 3]{carmon2022recapp} requires a bounded diameter assumption on $\Y$, which is generally required for $\Lb_1, \ldots, \Lb_N$ to be finite if each of $f_1, \ldots, f_n$ is nonlinear.}
\label{tab:complexity_full}
\end{table*}

When viewed as a saddle point or variational inequality problem, notice that~\eqref{eq:semilinear} has a finite sum structure. In order to make direct comparisons, we decompose the objective block-wise, that is, $\Lcal(\x, \y) = \sum_{J=1}^N \Lcal_J(\x, \y)$, where
\begin{align}
    \Lcal_J(\x, \y) = \begin{cases}
        \sum_{j \in B_J} y_j f_j(\x) - \psi_J(\y_J) + \frac{1}{N}\phi(\x) & \text{ if \Cref{def:separable} (separability) is satisfied}\\
        \sum_{j \in B_J} y_j f_j(\x) - \frac{1}{N}\psi(\y) + \frac{1}{N}\phi(\x) & \text{ otherwise}
    \end{cases}. \label{eq:per_oracle_complexity}
\end{align}
Thus, when comparing to methods designed for finite sum objectives, we may consider the overall complexity of querying the oracle $(\Lcal_J, \grad \Lcal_J)$ to be $O(nd/N)$ if the objective is separable or $O(n(d/N + 1))$ if it is non-separable. Even if $(\Lcal_J, \grad \Lcal_J)$ is of cost $O(nd/N)$ to compute, each oracle call may be associated with a single step of the algorithm, which is still $O(n)$ if it updates the primal and dual variables in their entirety. This is the first consideration when computing the global complexities in \Cref{tab:finite_sum_cvx} and \Cref{tab:finite_sum_str_cvx}. For the second consideration, we compute the Lipschitzness and smoothness assumption of the individual component functions \emph{on average} with uniform or non-uniform sampling, as is commonly used in analyses of methods for sum-decomposable objectives. 
Many contemporary results are stated in terms of ``on average'' smoothness (for min-max problems) and Lipschitzness (for variational inequality problems). Using the same norms defined in~\eqref{eq:smth:xy}, we say that $\Lcal_1, \ldots, \Lcal_N$ are \emph{$\Lavg$-smooth on average} according to sampling weights $\pb = (p_1, \ldots, p_N)$ if 
\begin{align*}
    \Ex_{J \sim \pb}\norm{(1/p_J)\p{\grad \Lcal_J(\x, \y) - \grad \Lcal_J(\x', \y')}}_*^2 &\leq \Lavg^2\norm{(\x, \y) - (\x', \y')}^2.
\end{align*}
Recall the constants $\lambb = (\lamb_1, \ldots, \lamb_N)$, where $\lamb_J = \sqrt{\Gb_J^2 + \Lb_J^2}$ are the Lipschitz constants of each $\Lcal_J$ with respect to the norm $\norm{\cdot}$ defined above.
The prototypical sampling schemes are the uniform and importance-weighted schemes
\begin{align*}
    \Ex_{J \sim \mathrm{unif}[N]}\norm{n\p{\grad \Lcal_J(\x, \y) - \grad \Lcal_J(\x', \y')}}_*^2 &\leq \underbrace{N\norm{\lambb}_2^2}_{\Lunif^2}\norm{(\x, \y) - (\x', \y')}^2\\
    \Ex_{J \sim \lambb}\norm{(1/\lamb_J)\p{\grad \Lcal_J(\x, \y) - \grad \Lcal_J(\x', \y')}}_*^2 &\leq \underbrace{\norm{\lambb}_1^2}_{\Limp^2}\norm{(\x, \y) - (\x', \y')}^2.
\end{align*}
Note that under the same sampling scheme, $\Lavg$-smoothness on average is also implied by $(1/\Lavg)$-cocoercivity on average \citep[Assumption 3]{cai2024variance}. Finally, for $\epsilon \searrow 0$, we have that $O\p{N + {\sqrt{N} \Lunif}/{\epsilon}} = O\p{{\sqrt{N} \Lunif}/{\epsilon}}$ and $O\p{N + {\sqrt{N} \Limp}/{\epsilon}} =   O\p{{\sqrt{N} \Limp}/{\epsilon}}$ which, combined with the oracle cost for~\eqref{eq:per_oracle_complexity} , leads to the results in \Cref{tab:finite_sum_cvx} and \Cref{tab:finite_sum_str_cvx}.

\begin{table*}[t]
\centering
\begin{adjustbox}{max width=\linewidth}
\begin{tabular}{ccc}
\toprule
    {\bf Reference}& {\bf Additional Structure} & {\bf Global Complexity (big-$\tilde{O})$} \\
    \midrule
    \begin{tabular}{c}
         \citet[Corollary 6]{alacaoglu2022stochastic}\\
         \citet[Theorem 4.2]{cai2024variance}\\
    \end{tabular}        
    &
    Constants known
    &
    $\frac{n(d+N)}{\sqrt{N}\epsilon}\norm{\lambb}_1$
    \\
    \midrule
    \begin{tabular}{c}
         \citet[Corollary 6]{alacaoglu2022stochastic}\\
         \citet[Theorem 4.2]{cai2024variance}\\
         \citet[Corollary 1]{pichugin2024method}
    \end{tabular}        
    &
    &
    $\frac{n\p{d + N}}{\epsilon}\norm{\lambb}_2$
    \\
    \midrule
    \citet[Thm. 1 \& Eq. (38)]{diakonikolas2025block}
    &
    \begin{tabular}{c}
         Constants known \\
         + Separable 
    \end{tabular}
    &
    \begin{tabular}{c}
         $\frac{n(d + N)}{N\epsilon}\norm{\lambb}_{1/2}$ \\
         $\frac{nd}{N\epsilon}\norm{\lambb}_{1/2}$ 
    \end{tabular}
    \\
    \midrule
    \citet[Thm. 1 \& Eq. (36)]{diakonikolas2025block}
    &
    \begin{tabular}{c}
         \quad \\
         + Separable 
    \end{tabular}
    &
    \begin{tabular}{c}
         $\frac{n(d+N)}{\epsilon}\sqrt{N}\norm{\lambb}_{2}$ \\
         $\frac{nd}{\epsilon}\sqrt{N}\norm{\lambb}_{2}$
    \end{tabular}
    \\
    \midrule
    \begin{tabular}{c}
         {\bf This work(\Cref{thm:stochastic:historic})}\\
         {\bf This work (\Cref{thm:sep:historic})}
    \end{tabular}
    &
    \begin{tabular}{c}
         Constants known\\
         + Separable
    \end{tabular}  
    &
    \begin{tabular}{c}
         $\frac{n(d+N)}{\edit{\sqrt{N}}\epsilon}\p{\norm{\lambb}_{1}^{1/2} (\norm{\lambb}_{1}^{1/2} +  \norm{\Gbb}_{1}^{1/2})}$\\
         $\frac{nd}{\sqrt{N}\epsilon}\p{\norm{\lambb}_1^{1/2}(\norm{\Lbb}_1^{1/2} + \norm{\Gbb}_\infty^{1/2})}$\\
    \end{tabular}
    \\
    \midrule
    
    \begin{tabular}{c}
         {\bf This work (\Cref{thm:stochastic:minty})}\\
         {\bf This work (\Cref{thm:sep:minty})}
    \end{tabular}
    &
    \begin{tabular}{c}
         Constants known\\
         + Separable
    \end{tabular}  
    &
    \begin{tabular}{c}
         $\frac{n(d+N)}{N\epsilon}\p{\norm{\lambb}_{1/2}^{1/2} (\norm{\lambb}_{1/2}^{1/4}  \norm{\Lbb}_{1/2}^{1/4} + \norm{\Gbb}_{1/2}^{1/2})}$\\
         $\frac{nd}{N\epsilon}\p{\norm{\lambb}_{1/2}^{1/4} ( \norm{\Lbb}_{1/2}^{3/4} + \sqrt{N}\norm{\Gbb}_\infty^{3/4})}$\\
    \end{tabular}
    \\
    \bottomrule
\end{tabular}

\end{adjustbox}
\vspace{6pt}
\caption{{\bf Complexity Bounds for \blue{Convex-Concave Finite-Sum} Objectives. } Arithmetic or global complexity (i.e., the total number of elementary operations required to compute $(\x, \y)$ satisfying $\Ex[\gap^{\u, \v}(\x, \y)] \leq \epsilon$ for fixed $(\u, \v) \in \X \times \Y$. We use $\lambb$ as defined in \Cref{sec:preliminaries}. The objective is assumed to be {\bf convex-concave} and have a finite sum structure $\Lcal(\x, \y) = \sum_{J=1}^N \Lcal_J(\x, \y)$. The expectation is taken over any randomness incurred by the algorithm.}\label{tab:finite_sum_cvx}
\end{table*}

\begin{remark}
    We observe the following in \Cref{tab:finite_sum_cvx} and \Cref{tab:finite_sum_str_cvx}.
    \begin{itemize}
        \item For the improved dependence on the problem constants $(\Gbb, \Lbb)$ and strong convexity constants $(\mu, \nu)$ there are two main themes, which both involve ``decoupling'' of the two constants. Indeed, all results except for ours depend only on the aggregate Lipschitz constants $\lambb$ and the minimum of the strong convexity constants $(\mu \wedge \nu)$, whereas we may (partially) separate these into dependences on ``$\Lbb$ over $\mu$'' terms and ``$\Gbb$ over $\sqrt{\mu\nu}$'' terms.
        \item In \Cref{tab:finite_sum_cvx}, without considering the differences between the aggregate constants $\lambb$ and decoupled constants $(\Gbb, \Lbb)$, we notice \edit{that historical regularization (\Cref{thm:stochastic:historic}) performs on par with the known constants variant of \citet{alacaoglu2022stochastic} and \citet{cai2024variance}. On the other hand, non-uniform block replacement (\Cref{thm:stochastic:minty}) may improve by a factor of $\sqrt{N}$.
        }
        \item In \Cref{tab:finite_sum_cvx} for the non-separable case, \edit{the $\sqrt{N}$ times $\ell_1$-norm scaling of \Cref{thm:stochastic:historic} may improve upon the $\ell_{1/2}$-norm scaling of \citet{diakonikolas2025block} for more uniformly spread problem constants.}
    \end{itemize}
\end{remark}

\begin{table*}[t]
\centering
\begin{adjustbox}{max width=\linewidth}
\begin{tabular}{ccc}
\toprule
    {\bf Reference}& {\bf Additional Structure} & {\bf Global Complexity (big-$\tilde{O})$} \\
    \midrule
    \citet[Theorem 2]{palaniappan2016stochastic}
    &
    Constants known
    &
    $\frac{n(n+d)}{N}\p{N + \frac{L^2 + G^2}{(\mu \wedge \nu)^2}}\ln(1/\epsilon)$
    \\
    \midrule
    \begin{tabular}{c}
        \citet[Corollary 27]{alacaoglu2022stochastic}\\
        \citet[Theorem 4.6]{cai2024variance}
    \end{tabular}
    &
    Constants known
    &
    $\frac{n(d+N)}{N}\p{N + \frac{\sqrt{N}\norm{\lambb}_1}{\mu \wedge \nu}}\ln\p{\frac{1}{\epsilon}}$
    \\
    \midrule
   \begin{tabular}{c}
        \citet[Corollary 27]{alacaoglu2022stochastic}\\
        \citet[Theorem 4.6]{cai2024variance}
    \end{tabular}
    &
    
    &
    $\frac{n(d+N)}{N}\p{N + \frac{N\norm{\lambb}_2}{\mu \wedge \nu}}\ln\p{\frac{1}{\epsilon}}$
    \\
    \midrule
    \citet[Thm. 1 \& Eq. (38)]{diakonikolas2025block} 
    &
    \begin{tabular}{c}
         Constants known \\
         + Separable 
    \end{tabular}
    &
    \begin{tabular}{c}
        $\frac{n(d+N)}{N}\p{N + \frac{\norm{\lambb}_{1/2}}{\mu \wedge \nu}} \ln\p{\frac{1}{\epsilon}}$ \\
        $\frac{nd}{N}\p{N + \frac{\norm{\lambb}_{1/2}}{\mu \wedge \nu}} \ln\p{\frac{1}{\epsilon}}$
    \end{tabular}
    \\
    \midrule
    \citet[Thm. 1 \& Eq. (36)]{diakonikolas2025block} 
    &
    \begin{tabular}{c}
         \quad \\
         + Separable 
    \end{tabular}
    &
    \begin{tabular}{c}
        $\frac{n(d+N)}{N}\p{N + \frac{N^{3/2}\norm{\lambb}_{2}}{\mu \wedge \nu}} \ln\p{\frac{1}{\epsilon}}$ \\
        $\frac{nd}{N}\p{N + \frac{N^{3/2}\norm{\lambb}_{2}}{\mu \wedge \nu}} \ln\p{\frac{1}{\epsilon}}$ 
    \end{tabular}
    \\
    \midrule
    \begin{tabular}{c}
         {\bf This work(\Cref{thm:stochastic:historic})}\\
         {\bf This work (\Cref{thm:sep:historic})}
    \end{tabular}
    &
    \begin{tabular}{c}
         Constants known\\
         + Separable
    \end{tabular}  
    &
    \begin{tabular}{c}
         $\frac{n(d+N)}{N}\p{N + \frac{\sqrt{N}\norm{\lambb}_1}{\mu} + \frac{\sqrt{N}\norm{\lambb}_1^{1/2}\norm{\Gbb}_1^{1/2}}{\sqrt{\mu\nu}}}\ln\p{\frac{1}{\epsilon}}$\\
         $\frac{nd}{N}\p{N + \frac{\sqrt{N}\norm{\lambb}_1^{1/2}\norm{\Lbb}_1^{1/2}}{\mu} + \frac{\sqrt{N}\norm{\lambb}_1^{1/2}\norm{\Gbb}_\infty^{1/2}}{\sqrt{\mu\nu}}}\ln\p{\frac{1}{\epsilon}}$\\
    \end{tabular}
    \\
    \midrule
    \begin{tabular}{c}
         {\bf This work (\Cref{thm:stochastic:minty})}\\
         {\bf This work (\Cref{thm:sep:minty})}
    \end{tabular}
    &
    \begin{tabular}{c}
         Constants known\\
         + Separable
    \end{tabular}  
    &
    \begin{tabular}{c}
         $\frac{n(d+N)}{N}\p{N +\frac{\norm{\lambb}_{1/2}^{3/4} \norm{\Lbb}_{1/2}^{1/4}}{\mu} + \frac{\norm{\lambb}_{1/2}^{1/2} \norm{\Gbb}_{1/2}^{1/2}}{\sqrt{\mu\nu}}} \ln\p{\frac{1}{\epsilon}}$\\
         $\frac{nd}{N}\p{N + \frac{\norm{\lambb}_{1/2}^{1/4}  \norm{\Lbb}_{1/2}^{3/4}}{\mu} + \frac{\edit{\sqrt{N}}\norm{\lambb}_{1/2}^{1/4} \norm{\Gbb}_\infty^{3/4}}{\sqrt{\mu\nu}}}\ln\p{\frac{1}{\epsilon}}$\\
    \end{tabular}
    \\
    \bottomrule
\end{tabular}

\end{adjustbox}
\vspace{6pt}
\caption{{\bf Complexity Bounds for \blue{Strongly Convex-Strongly Concave Finite-Sum} Objectives}: Arithmetic or global complexity (i.e., the total number of elementary operations required to compute $(\x, \y)$ satisfying $\Ex[\gap^{\u, \v}(\x, \y)] \leq \epsilon$ for fixed $(\u, \v) \in \X \times \Y$. We use $\lambb$ as defined in \Cref{sec:preliminaries}. The objective is assumed to be {\bf $(\mu, \nu)$-strongly convex-strongly concave} and have a finite sum structure $\Lcal(\x, \y) = \sum_{J=1}^N \Lcal_J(\x, \y)$. The expectation is taken over any randomness in the algorithm.}\label{tab:finite_sum_str_cvx}
\end{table*}
We also mention the work of \citet{boob2024optimal}, which solves a functionally constrained variational inequality formulation akin to Example 3 from \Cref{sec:intro}. They operate under a completely different set of assumptions, largely to handle the possible unboundedness of the domain $\Y$ of the Lagrange multipliers. Furthermore, the gradient operator and the functional constraints satisfy non-standard deviation control (as opposed to Lipschitzness) inequalities (see \citet[Eq. (1.3) and (1.4)]{boob2024optimal}, and thus this work is not directly comparable to ours.

\subsection{Bilinearly Coupled Problems}
While originally motivated by nonbilinearly-coupled min-max problems, the bilinearly-coupled setting constitutes an important special case of~\eqref{eq:semilinear}, defined via $\fb(\x) = \A \x$ for $\A \in \R^{n \times d}$. For the sake of discussion, we consider the convex-concave case (which includes matrix games for which $\phi \equiv 0$ and $\psi \equiv 0$).
Randomized algorithms (such as randomized mirror-prox) for such problems were explored in \citet{juditsky2011first, juditsky2013randomized} to achieve complexity guarantees of the form $O(\sigma^2\epsilon^{-2} + \norm{\A}_{\X, \Y^*}\epsilon^{-1})$, where $\sigma^2$ is a measurement of noise arising from using stochastic estimates of matrix-vector multiplications and $\norm{\A}_{\X, \Y^*} = \sup\br{\norm{\A\x}_{\Y^*}: \norm{\x}_\X = 1}$ is the induced matrix norm. 
For general problems, the approaches of \citet{chambolle2011afirstorder, alacaoglu2022stochastic} achieve
\begin{align*}
    \min\br{O\p{nd\norm{\A}_{2, 2}\epsilon^{-1}}, O\p{nd + \sqrt{nd(n+d)\norm{\A}_{\Fro}\epsilon^{-1}}}}
\end{align*}
whereas these are reduced by \citet{song2021variance} and \citet{alacaoglu2022complexity}\footnote{Guarantees might be for the expected supremum of gap or supremum of expected gap. We compare them side-by-side due to the argument in the earlier part of the section.} to
\begin{align}
    O\p{nd + d \frac{n\max_{j} \norm{\A_{j\cdot}}_2}{\epsilon}},\label{eq:dense}
\end{align}
when the objective function is \emph{separable}, where $\A_{i \cdot}$ denotes the $i$-th row of $\A$. In our notation, $\Gc_i = \norm{\A_{i\cdot}}_\infty$ when we equip $\X$ with the $\ell_1$-norm.
Thus, due to separability, we may apply \Cref{thm:sep:historic} with $N = n$ to achieve a global complexity of 
\begin{align}
    O\p{nd + d\frac{\sqrt{(\sum_{i=1}^n \norm{\A_{i\cdot}}_\infty) \cdot (n\max_j \norm{\A_{j\cdot}}_\infty)} }{\epsilon}}. \label{eq:dense_ours}
\end{align}
Because
\begin{align*}
    n\textstyle\max_{j} \norm{\A_{j\cdot}}_2 \overset{(\circ)}{\geq} n\textstyle\max_{j} \norm{\A_{j\cdot}}_\infty  \overset{(*)}{\geq} \sum_{i=1}^n \norm{\A_{i\cdot}}_\infty,
\end{align*}
and $(\circ)$ offers an up to a $\sqrt{d}$-factor improvement and $(*)$ offers an up to $n$-factor improvement, our result can offer up to an order-$\sqrt{nd}$ improvement overall. This improvement is realized when within-row entries are highly uniform and within-column entries are highly non-uniform, leading to highly non-uniform infinity norms of the rows.
For linearly constrained problems, \citet{alacaoglu2022complexity} achieve $O\p{nd + {d\sum_{i=1}^n \norm{\A_{i\cdot}}_2}/{\epsilon}}$. It is relevant to note that our improvement relies heavily on the non-uniform sampling and non-uniform historical regularization applied in \Cref{sec:stochastic:historic} and \Cref{sec:sep:historic}. Non-uniform sampling has been applied in the works above as well as earlier works such as \citet{alacoglu2017smooth} and \citet{chambolle2018stochastic}, but our findings show that sampling strategies may not be sufficient to remove the extraneous dimension factors in the iteration complexity.

\subsection{Conclusion}
In this paper, we provided a class of algorithms for dual-linear min-max problems under separability and non-separability assumptions.
A primary direction for further investigation would be establishing lower bounds for both the full vector and stochastic case, following recent trends in the saddle-point optimization literature \citep{Zhang2022OnLower}. In particular, achieving lower bounds that 1) account for non-smooth uncoupled components $\phi$ and $\psi$ and 2) depend on the individual Lipschitz and smoothness parameters of $f_1, \ldots, f_n$, would be of interest. This contrasts current results, which depend on the smoothness and strong convexity constants of the entire objective $(\x, \y) \mapsto \Lcal(\x, \y)$ jointly in both the primal and dual vectors. %
On the side of upper bounds, one may pursue complexity guarantees that fully decouple the Lipschitz and smoothness constants of each objective. For instance, \Cref{thm:stochastic:historic} has a ``partially decoupled'' complexity which depends on the constant $\norm{\lambb}_{1}^{1/2} (\norm{\lambb}_{1}^{1/2} +  \norm{\Gbb}_{1}^{1/2})$, which is upper bounded by $\norm{\lambb}_{1}$ (the constant associated with a generic stochastic variational inequality method) and lower bounded by $\norm{\Lbb}_{1} +  \norm{\Gbb}_{1}$. Accordingly, we ask whether a complexity dependent on $\norm{\Lbb}_{1} +  \norm{\Gbb}_{1}$ can be achieved, with analogous results for strongly convex-strongly concave settings.

On a practical note, we also raise the possibility of extending our algorithm class to parameter-free or adaptive variants. For instance, inspired by the golden ratio algorithm of \citet{malitsky2020golden}, an ``auto-conditioned'' variant of PDHG was recently proposed \citep{lan2024autoconditioned}. This approach employs not only extrapolated gradient estimates, but also an additional sequence of extrapolated prox-centers; these alternative prox-centers may bear similarity to the historical regularization technique described in \Cref{sec:stochastic:historic} and \Cref{sec:sep:historic}. These techniques, developed primarily for bilinearly coupled min-max problems, avoid computation of the spectral norm of the coupling matrix $\A$ (analogous to adaptively estimating $\Gb_1, \ldots, \Gb_N$ in our setup). \citet{fercoq2024monitoring} also proposes techniques for adaptive step size selection in PDHG, using a spectral norm estimation technique. Given that PDHG is a conceptual predecessor of our algorithm, we hypothesize that similar techniques may be employed in the dual-linear setting. We leave these directions as interesting avenues for future research.

\paragraph{Acknowledgements}
RM was supported by NSF DMS-2023166, CCF-2019844, and DMS-2134012. 
JD was supported in part by the Air Force Office of Scientific Research under award number FA9550-24-1-0076, by the U.S.\ Office of Naval Research under contract number  N00014-22-1-2348, and by the NSF CAREER Award CCF-2440563. ZH was supported by NSF CCF-2019844, DMS-2134012, NIH, and IARPA 2022-22072200003. Part of this work was performed while RM and ZH were visiting the Simons Institute for the Theory of Computing. 
Any opinions, findings and conclusions or recommendations expressed in this material are those of the author(s) and do not necessarily reflect the views of the U.S.\ Department of Defense.

\bibliographystyle{abbrvnat}
\bibliography{bib}

\clearpage
\appendix

\section{Additional Results}\label{sec:a:tables}
\edit{To complement the discussion surrounding \Cref{tab:complexity_internal} in \Cref{sec:discussion}, we also include similar tables for cases in which the strong convexity constants satisfy $\mu > 0$ and $\nu = 0$ (\Cref{tab:complexity_internal_nu0}), $\mu = 0$ and $\nu > 0$ (\Cref{tab:complexity_internal_mu0}), and $\mu = \nu = 0$ (\Cref{tab:complexity_internal_mu0_nu0}). 
}

\renewcommand{\arraystretch}{1.4}
\begin{table*}[h]
\centering
    \begin{adjustbox}{max width=\linewidth}
    \begin{tabular}{ccc}
    \toprule
        {\bf Algorithm Type} & {\bf Global Complexity (big-$O$)} \\
        \midrule
        Full vector (\Cref{thm:nonsep:full}) & $nd\p{\frac{L}{\mu}\ln\p{\tfrac{1}{\epsilon}} + G\sqrt{\frac{1}{\mu \epsilon}}}$\\
        Stochastic (\Cref{thm:stochastic:historic}) & $\p{\frac{nd}{N}+n}\p{\p{N + \frac{N\norm{\lambb}_1}{\mu}}\ln\p{\frac{1}{\epsilon}} + N\norm{\lambb}_1^{1/2}\norm{\Gbb}_1^{1/2}\sqrt{\frac{1}{\mu \epsilon}}}$\\
        Stochastic (\Cref{thm:stochastic:minty}) & $\p{\frac{nd}{N}+n}\p{\p{N + \frac{\norm{\lambb}_{1/2}^{3/4}\norm{\Lbb}_{1/2}^{1/4}}{\mu}}\ln\p{\frac{1}{\epsilon}} + \norm{\lambb}_{1/2}^{1/2}\norm{\Gbb}_{1/2}^{1/2}\sqrt{\frac{1}{\mu \epsilon}}}$\\        
        Block Coordinate-wise (\Cref{thm:sep:historic}) & 
        $\frac{nd}{N}\p{\p{N + \frac{\sqrt{N}\norm{\lambb}_1^{1/2}\norm{\Lbb}_1^{1/2}}{\mu}}\ln\p{\frac{1}{\epsilon}} + \sqrt{N}\norm{\lambb}_1^{1/2}\norm{\Gbb}_{\infty}^{1/2}\sqrt{\frac{1}{\mu \epsilon}}}$\\
        Block Coordinate-wise (\Cref{thm:sep:minty}) & 
        $\frac{nd}{N}\p{\p{N + \frac{\norm{\lambb}_{1/2}^{1/4}\norm{\Lbb}_{1/2}^{3/4}}{\mu}}\ln\p{\frac{1}{\epsilon}} + \sqrt{N}\norm{\lambb}_{1/2}^{1/4}\norm{\Gbb}_{\infty}^{3/4}\sqrt{\frac{1}{\mu \epsilon}}}$\\
        \bottomrule
    \end{tabular}
    \end{adjustbox}
    \vspace{6pt}
\caption{{\bf Complexity Bounds for Full Vector and Stochastic Methods for the Case $\mu > 0, \nu = 0$.} Arithmetic or global complexity (i.e., the total number of elementary operations required to compute $(\x, \y)$ satisfying $\E{}{\gap^{\u, \v}(\x, \y) \leq \epsilon}$ for fixed $(\u, \v) \in \X \times \Y$, with the expectation taken over all algorithmic randomness. 
}
\label{tab:complexity_internal_nu0}
\end{table*}

\renewcommand{\arraystretch}{1.4}
\begin{table*}[h]
\centering
    \begin{adjustbox}{max width=\linewidth}
    \begin{tabular}{ccc}
    \toprule
        {\bf Algorithm Type} & {\bf Global Complexity (big-$O$)} \\
        \midrule
        Full vector (\Cref{thm:nonsep:full}) & $nd\p{\frac{L}{\epsilon} +  G\sqrt{\frac{1}{\nu \epsilon}}}$\\
        Stochastic (\Cref{thm:stochastic:historic}) & $\p{\frac{nd}{N}+n}\p{N\ln\p{\frac{1}{\epsilon}} + \frac{\sqrt{N}\norm{\lambb}_1}{\epsilon} + (N\norm{\lambb}_1^{1/2}\norm{\Gbb}_1^{1/2})\sqrt{\frac{1}{\nu \epsilon}}}$\\
        Stochastic (\Cref{thm:stochastic:minty}) & 
        $\p{\frac{nd}{N}+n}\p{N\ln\p{\frac{1}{\epsilon}} + \frac{\norm{\lambb}_{1/2}^{3/4}\norm{\Lbb}_{1/2}^{1/4}}{\epsilon} + \norm{\lambb}_{1/2}^{1/2}\norm{\Gbb}_{1/2}^{1/2}\sqrt{\frac{1}{\nu \epsilon}}}$\\  
        Block Coordinate-wise (\Cref{thm:sep:historic}) & 
        $\frac{nd}{N}\p{\frac{\sqrt{N}\norm{\lambb}_1^{1/2}\norm{\Lbb}_1^{1/2}}{\epsilon} + \sqrt{N}\norm{\lambb}_1^{1/2}\norm{\Gbb}_{\infty}^{1/2}\sqrt{\frac{1}{\nu \epsilon}}}$\\
        Block Coordinate-wise (\Cref{thm:sep:minty}) & 
        $\frac{nd}{N}\p{N\ln \p{\frac{N}{\epsilon}} + \frac{\norm{\lambb}_{1/2}^{1/4}\norm{\Lbb}_{1/2}^{3/4}}{\epsilon} + \sqrt{N}\norm{\lambb}_{1/2}^{1/4}\norm{\Gbb}_{\infty}^{3/4}\sqrt{\frac{1}{\nu \epsilon}}}$\\
        \bottomrule
    \end{tabular}
    \end{adjustbox}
    \vspace{6pt}
\caption{{\bf Complexity Bounds for Full Vector and Stochastic Methods for the Case $\mu = 0, \nu > 0$.} Arithmetic or global complexity (i.e., the total number of elementary operations required to compute $(\x, \y)$ satisfying $\E{}{\gap^{\u, \v}(\x, \y) \leq \epsilon}$ for fixed $(\u, \v) \in \X \times \Y$, with the expectation taken over all algorithmic randomness. 
}
\label{tab:complexity_internal_mu0}
\end{table*}

\renewcommand{\arraystretch}{1.4}
\begin{table*}[h]
\centering
    \begin{adjustbox}{max width=\linewidth}
    \begin{tabular}{ccc}
    \toprule
        {\bf Algorithm Type} & {\bf Global Complexity (big-$O$)} \\
        \midrule
        Full vector (\Cref{thm:nonsep:full}) & $nd\p{\frac{L  + G}{\epsilon}}$\\
        Stochastic (\Cref{thm:stochastic:historic}) & $\p{\frac{nd}{N}+n}\p{\frac{\sqrt{N}\norm{\lambb}_1  + \sqrt{N}\norm{\lambb}_1^{1/2}\norm{\Gbb}_1^{1/2}}{\epsilon}}$\\
        Stochastic (\Cref{thm:stochastic:minty}) & 
        $\p{\frac{nd}{N}+n}\p{N\ln\p{\frac{1}{\epsilon}} + \frac{\norm{\lambb}_{1/2}^{3/4}\norm{\Lbb}_{1/2}^{1/4}  + \norm{\lambb}_{1/2}^{1/2}\norm{\Gbb}_{1/2}^{1/2}}{\epsilon}}$\\
        Block Coordinate-wise (\Cref{thm:sep:historic}) & 
        $\frac{nd}{N}\p{\frac{\sqrt{N}\norm{\lambb}_1^{1/2}\norm{\Lbb}_1^{1/2}  + \sqrt{N}\norm{\lambb}_1^{1/2}\norm{\Gbb}_{\infty}^{1/2}}{\epsilon}}$\\
        Block Coordinate-wise (\Cref{thm:sep:minty}) & 
        $\frac{nd}{N}\p{N\ln \p{\frac{N}{\epsilon}} + \frac{\norm{\lambb}_{1/2}^{1/4}\norm{\Lbb}_{1/2}^{3/4}  + \sqrt{N}\norm{\lambb}_{1/2}^{1/4}\norm{\Gbb}_{\infty}^{3/4}}{\epsilon}}$\\
        \bottomrule
    \end{tabular}
    \end{adjustbox}
    \vspace{6pt}
\caption{{\bf Complexity Bounds for Full Vector and Stochastic Methods for the Case $\mu = \nu = 0$.} Arithmetic or global complexity (i.e., the total number of elementary operations required to compute $(\x, \y)$ satisfying $\E{}{\gap^{\u, \v}(\x, \y) \leq \epsilon}$ for fixed $(\u, \v) \in \X \times \Y$, with the expectation taken over all algorithmic randomness. 
}
\label{tab:complexity_internal_mu0_nu0}
\end{table*}

\end{document}